\documentclass[12pt]{article}
\usepackage{amsmath}
\usepackage{graphicx}
\usepackage{enumerate}
\usepackage{natbib}
\usepackage{url} 
\usepackage{mathrsfs}
\usepackage{hyperref}
\usepackage{color} 
\usepackage{bm} 
\usepackage{amsthm} 
\usepackage{amssymb} 
\usepackage{verbatim} 
\usepackage{enumitem} 
\usepackage{epsfig} 
\usepackage{fullpage}
\usepackage{setspace}
\usepackage{rotating} 
\usepackage{ulem}   

\theoremstyle{plain}
\newtheorem{lemma}{Lemma}[section]
\newtheorem{theorem}{Theorem}[section]
\newtheorem{corollary}[theorem]{Corollary}
\newtheorem{assumption}{Assumption}

\theoremstyle{remark}
\newtheorem{remark}{Remark}[section]

\theoremstyle{definition}
\newtheorem{example}{Example}

\graphicspath{{./figs/}}
\DeclareGraphicsExtensions{.eps,.ps,.pdf}

\newcommand{\Ex}{\textsf{E}}
\newcommand{\Cov}{\textsf{Cov}}
\newcommand {\Prob}{\textsf{P}}
\newcommand{\Var}{\textup{Var}}

\renewcommand{\theequation}{\thesection.\arabic{equation}}

\begin{document}

\title{\bf Asymptotic Properties of Multi-Treatment  Covariate Adaptive Randomization Procedures for Balancing Observed and Unobserved Covariates\footnote{Research supported by National Key
R\&D Program of China (No. 2024YFA1013502),  NSF of China (Grant Nos. U23A2064) and   the Summit Advancement Disciplines of Zhejiang Province (Zhejiang Gongshang University - Statistics).
}
}

\author{Li-Xin Zhang\footnote{School  of Statistics and Mathematics, Zhejiang Gongshang University, Hangzhou 310018 {\rm 310018}, P.R. China.   Email:stazlx@mail.zjgsu.edu.cn}
}
\date{\today}
\maketitle

\thispagestyle{empty}

\begin{abstract}
Balancing treatment allocation for influential covariates has become increasingly important in today's clinical trials.  Covariate adaptive randomization (CAR) procedures are extensively used to reduce the likelihood of covariate imbalances occurring.
In literatures, most studies have focused on the balancing of discrete covariates.
Applications of CAR  for balancing continuous covariates remain comparatively rare, especially in multi-treatment clinical trials, and the theoretical properties of multi-treatment  CAR have remained largely elusive for decades. In this paper,  we consider a general framework of CAR procedures for multi-treatment clinal trials which can balance general covariate features, such as quadratic and interaction terms which can be discrete, continuous, and mixing. We show that under widely satisfied conditions the proposed procedures have superior balancing properties; in particular, the convergence rate of imbalance vectors can attain the best rate $O_P(1)$ for discrete covariates, continuous covariates, or combinations of both discrete and continuous covariates, and at the same time, the convergence rate of the imbalance of unobserved covariates is  $O_P(\sqrt n)$, where $n$ is the sample size.
The general framework not only unifies many existing methods and related theories,  introduces a much broader class of new and useful CAR procedures, but also provides new insights and a complete picture of the properties of CAR procedures. The favorable balancing properties lead to the precision of the treatment effect test in the presence of a heteroscedastic linear model with dependent covariate features. As an application, the properties of the test of treatment effect with unobserved covariates are studied under the CAR procedures, and consistent tests are proposed so that the test has an asymptotic precise type I error rate even if the working model is wrong and covariates are unobserved in the analysis.

  \end{abstract}

Besides the previous  updating which shows that when $\bm \Lambda_n=\sum_{i=1}^n(T\bm 
T_i-1)\odot \bm \phi(\bm X_i)$ is a positive Harris recurrent Markov chain, then 
$$ \frac{\Ex[\sum_{i=1}^n(T T_i^{(t)}-1)m(\bm X_i))]^2}{n}\to 0 \text{ if and only if } m(\bm X)=\langle \bm x_0,\bm \phi(\bm X)\rangle, $$
the main updating of the version is showing that: the imbalance $\bm \Lambda_n$ is bounded in probability without the condition for the Markov chain being irreducible, so that the convergence rate $O(n^{1/(2\gamma-2)})$  of \cite{Ma2022} is improved to the fastest one $O(1)$.

\thispagestyle{empty}

\newpage

\clearpage
\setcounter{page}{1}
\section{Introduction}\label{sec:Introduction}
\setcounter{equation}{0}

It is well known that covariates usually play important role in clinical trials. Clinical trialists are often concerned about unbalanced treatment arms with respect to key covariates of interest. Covariate adaptive randomization (CAR) procedures have often been implemented for
balancing treatment allocation over influential covariates in clinical trials \citep{McEntegart2003, Rosenberger2015, Taves2010, Lin2015}. The main purposes are to enhance the credibility of trial results and to increase the efficiency of treatment effect estimation.
Many such procedures have been proposed to balance treatment assignments within strata and over covariate margins, provided that the covariates under consideration are discrete  (categorical with two or more levels) \citep{Pocock1975, Antognini2011, Hu2012, Rosenberger2015}.
In contrast, studies using CAR procedures with continuous covariates have been comparatively rare. In the literature, continuous covariates are typically discretized in order to be included in the randomization scheme \citep{Taves2010, Ma2013}.
 However, the breakdown of a continuous covariate into subcategories means loss of information   \citep{Scott2002}. The CAR procedures with discretization can not balance the covariates very satisfactorily and the uncontrolled imbalances in continuous covariates may affect the statistical analysis of the trial outcome
\citep{Ciolino2011}.      Recently,   a variety of methods for balancing continuous covariates have been proposed by minimization of prescribed imbalance measures.
\cite{Ciolino2011} measures the imbalance by ranking the pooled covariates and taking the ratio of the
sum of ranks from the experimental treatment group to the control group. Other measures can be found in the literature such as target differences in covariate means \citep{Li2018},  variance \citep{Nishi2004, Endo2006}, rank \citep{Hoehler1987, Stigsby2010},   Mahalanobis distances \citep{Qin2016},   p-values of analysis of variance \citep{Frane1998}, and certain nonparametric estimators of covariate distributions \citep{Lin2012, Ma2013, Jiang2017}.
However, these methods have been created for ad hoc purposes, and their theoretical properties remain largely unknown, limiting their applicability.

To assess how covariates are balanced, and thus assure the validity of a CAR procedure, the convergence rate of the imbalance vectors must be described.
The imbalance vectors typically have the form of $\sum_{i=1}^{n}(2T_i-1)\bm X_i$, where $n$ is the sample size and $T_i=1$ for the treatment and $T_i=0$ for the control.
To make the statistical analysis valid under the CAR procedure, it is expected that the covariate imbalance  $\sum_{i=1}^{n}(2T_i-1)\bm X_i$ and treatment allocation imbalance $\sum_{i=1}^n(2T_i-1)$ are both convergent at least in a rate  $o_P(\sqrt{n})$ \citep{Ma2015, Ma2022}, relative to $O_P(\sqrt n)$ in the central limit theorem so that imbalances do not affect the asymptotic distribution of statistics for analysis.      Sometimes the convergence rate $o_P(n)$ of $\sum_{i=1}^n(2T_i-1)m(\bm X_i)$ is also required at the same time for other known or unknown function $m(\bm X_i)$.
Even for the CAR procedures with discrete covariates in which   $\bm X_i$ is the indicator of a covariate margin or stratum,  the theories which give such convergence rates are very rare.
\cite{Hu2012} develops a class of covariate-adaptive biased coin randomization procedures that gives a trade-off overall, marginal and within-stratum balances. However, the theoretical properties are induced under a strict condition (condition (C) of Theorem 3.2 in \cite{Hu2012}).
Therefore,   the properties are not suitable for most existing randomization schemes, especially  \cite{Pocock1975}'s procedure,  because of the limitation of the condition. Though Pocock and Simon's procedure  is one of the most popular  covariate-adaptive designs,   all studies of this method had been  merely carried
out by simulations and there was  no theoretical justification
that the procedure even works as intended for decades \citep{Rosenberger2008}. By applying the theory on Markov chains, the theoretical properties are first obtained by \cite{Ma2015} for the two-treatment case and most recently by \cite{Hu2022} for the general multi-treatment case. It is shown that the best available convergence rate of the imbalance vectors is $O_P(1)$, relative to $O_P(\sqrt{n})$ under complete randomization.

Recently, \cite{Ma2022} proposes a general CAR procedure for balancing a covariate feature $\bm \phi(\bm X_i)$, where $\bm\phi(\bm X_i)$ can contain additional covariate features besides the covariate vector $\bm X_i$, such as quadratic and interaction terms.  It is shown that the convergence rate of the imbalance vectors $\sum_{i=1}^{n}(2T_i-1)\bm \phi(\bm X_i)$ is $O_P(n^{\epsilon})$ for any $\epsilon>0$ if
all of the moments of $\bm \phi(\bm X_i)$ are finite. To obtain the best convergence rate $O_P(1)$ and the convergence rate $o_P(n)$ of $\sum_{i=1}^n (2T_i-1)m(\bm X_i)$ for a function $m(\bm X_i)$  which is different from the feature $\bm \phi(\bm X_i)$ and usually unknown,  an additional condition on the distribution of $\bm \phi(\bm X_i)$ is assumed so that the Markov chains induced by the imbalance vectors are verified to be irreducible. Under this condition, the feature $\bm \phi(\bm X_i)$ can not contain discrete covariates, and so, the results are not suitable for discrete covariates or combinations of both discrete and continuous covariates. Under a CAR procedure with such a condition, the best convergence rate of the treatment allocation imbalance $\sum_{i=1}^n (2T_i-1)$  is $O_P(\sqrt{n})$ instead of $o_P(\sqrt{n})$.
 The framework of \cite{Ma2022} unifies a lot of existing CAR procedures, but the theories do not include the existing ones as for  \cite{Pocock1975}'s procedure and \cite{Hu2012}'s procedure, etc.

  On the other hand and more importantly, the CAR procedures for balancing continuous covariates with related theories are merely proposed under the two-treatment case  \citep{Qin2016, Li2018, Ma2022}. The multi-treatment randomized clinical trial has been an important topic in some situations \citep{Pocock1975, Tymofyeyev2007, Hu2012}. Simultaneously investigating multiple treatments in a single study achieves considerable efficiency in contrast to the traditional two-treatment trials \citep{Saville2016}. The multi-treatment covariate-adaptive randomized clinical trial dramatically reduces the required sample sizes compared to conducting several traditional two-treatment randomized clinical trials separately for various involving regimens and would attract more patient recruitment to the trial since fewer patients are required to the placebo and then patients would be allocated to a promising treatment with higher probability.

   The purpose of this paper is to propose a general framework of  CAR procedures for balancing both discrete and continuous covariates among multi-treatments and to study related theoretical properties. The same as in \cite{Ma2022}, the CAR procedures are proposed by basing on a covariate feature $\bm\phi(\bm X)$   but now can generate a much broader range of CAR procedures because of the flexibility to define various feature maps $\bm \phi(\bm X_i)$. The proposed  CAR procedures and the theories are suitable for discrete covariates, continuous covariates, and combinations of both discrete and continuous covariates in both two-treatment and general multi-treatment clinical trials,   and give a uniform framework with related theories of the stratified randomization procedure, \cite{Pocock1975}'s procedure and  \cite{Hu2012}'s procedure for discrete covariates and those randomization procedures for minimizing imbalance measures of continuous covariates proposed by  \cite{Li2018},   \cite{Qin2016} and \cite{Ma2022} etc. New properties of   \cite{Pocock1975}'s procedure are also found through the theoretical properties. 

The general framework of the CAR procedures for balancing multi-treatments and theocratical properties are given in Section \ref{sec:multi}. In Section \ref{sec:two-treatment}, we consider the special case with two treatments for comparing the theories with the existing ones. Examples are given to show how the procedures can be applied for balancing various covariates. As an application, in Section \ref{sec:Inference}, we study the test for treatment effect under the CAR procedure based on a linear working model. Different from those models studied by \cite{Shao2010}, \cite{Ma2015}, \cite{Ma2020} etc,  the model that we considered has heteroscedastic errors and heteroscedastic unobserved covariates with linear or nonlinear structures. It will be shown that for such a heteroscedasticity model under a  CAR procedure, the classical test based on the least square estimators usually has not a precise type I error rate.   Consistent tests are proposed to recover the type I error rate and increase the power. The consistent tests are   shown to be most powerful under certain conditions, especially, under  the
stratified randomization procedure,   \cite{Hu2012}'s procedure or \cite{Pocock1975}'s procedure for a large class of real models with discrete covariates.

\section{ CAR Procedure for Two Treatments}\label{sec:two-treatment}
\setcounter{equation}{0}

\subsection{Framework}\label{sec:Procedure_Framework}

For comparing our proposed framework and theories with existing ones, especially those of \cite{Ma2022}, we first consider the two-treatment case in this section.
Suppose that $n$ experimental units are to be assigned to two treatment groups.
Let $T_i$ be the assignment of the $i$-th unit, i.e., $T_i=1$ for treatment 1 and $T_i=0$
for treatment 0.
Consider $p$ covariates available for each unit and use $\boldsymbol{X}_i=(X_{i,1}, ...,X_{i,p}) \in \mathbb{R}^p$ to denote the covariates of the $i$-th unit.
We assume that the covariates $\{\bm{X}_i; i=1,2,\ldots,n\}$ are independent and identically distributed as a random vector $\bm{X}$.
 We consider a procedure to balance general covariate features $\bm\phi(\bm X_i)$, defined by a feature map $\bm\phi(\bm x): \mathbb R^p\to \mathbb R^q$ that maps $\bm X_i$ into a $q$-dimensional feature space. Here, $q$ is usually larger than $p$ so
that $\bm \phi(\bm X_i)$   has more features than the original covariates. As in \cite{Ma2022}, we define the numerical imbalance measure $Imb_n$
  as the squared Euclidean norm of the imbalance vector $\sum_{i=1}^n (2T_i-1)\bm \phi(\bm X_i)$,
  $$Imb_n=\Big\|\sum_{i=1}^n (2T_i-1)\bm \phi(\bm X_i)\Big\|^2. $$
The $\bm\phi$-based CAR procedure to minimize the imbalance measure $Imb_n$ is defined as follows:
\begin{enumerate}
	
	\item[(1)]
	Randomly assign the first unit with equal probability to the treatment or   the control.
	
	\item[(2)]
	 Suppose that $(n-1)$ units have been assigned to a treatment $(n > 1)$ and the $n$-th unit
is to be assigned, and the results of assignments $T_1,\ldots, T_{n-1}$ of previous stages and all covariates $\bm X_1,\ldots, \bm X_n$ up  to the $n$-th unit are observed.  Calculate the `potential' imbalance measures $Imb_n^{(1)}=Imb_n\big|_{T_n=1}$ and $Imb_n^{(2)}=Imb_n\big|_{T_n=0}$,
corresponding to $T_n = 1$ and $T_n = 0$, respectively.
	
	\item[(3)]Assign the $n$-th unit to the treatment with the probability
 	\begin{align}\label{eq:allocation-2}
	\Prob(T_n=1|\boldsymbol{X}_n,...,\boldsymbol{X}_{1},T_{n-1},...,T_{1}) =\ell \left( Imb_n^{(1)} - Imb_n^{(2)}\right),
	\end{align}
	where $\ell(x):\mathbb R\to (0,1)$  is a non-increasing function with
\begin{align}
  \ell(0)&=0.5, \;
  \ell(-x)=1-\ell(x), \label{eq:allocationfunction1} \\
& \lim_{x\to + \infty} \ell(x)<0.5, \label{eq:allocationfunction2}\\
& \lim_{x\to +\infty}\ell(x)> 0.\label{eq:allocationfunction3}
\end{align}
	\item[(4)]
	Repeat the last two steps until all units are assigned.
\end{enumerate}
The function $\ell(x)$ in \eqref{eq:allocation-2} is called the allocation function. For \cite{Efron1971}'s biased coin type design, $\ell(x)=0.5$ if $x=0$, $\rho$ if $x<0$ and $1-\rho$ if $x>0$, where $0.5<\rho< 1$. In this case, the allocation probability in \eqref{eq:allocation-2} is
	\begin{align}\label{eq:BCallocation}
	\Prob(T_n=1|\boldsymbol{X}_n,...,\boldsymbol{X}_{1},T_{n-1},...,T_{1})&=
	\begin{cases}
	\rho & \text{if } Imb_n^{(1)} < Imb_n^{(2)},\\
	1-\rho & \text{if } Imb_n^{(1)} > Imb_n^{(2)},\\
	0.5 & \text{if } Imb_n^{(1)} = Imb_n^{(2)}.
	\end{cases}
\end{align}
The properties of $\bm\phi$-based CAR procedure with Efron's biased coin allocation are studied by \cite{Ma2022}. But, to obtain $\sum_{i=1}^n(2T_i-1)\phi(\bm X_i)=O_P(1)$ and the central limit theorem for $\sum_{i=1}^{n}(2T_i-1)m(\bm X_i)$ for a general function $m(\bm X)$, it is needed to assume that $\bm\phi(\bm X)$ is  a  purely continuous type random vector. The results do not conclude those for \cite{Pocock1975}'s procedure, and can not be used to obtain $\sum_{i=1}^n(2T_i-1)=O_P(1)$ and  $\sum_{i=1}^n(2T_i-1)\bm X_i=O_P(1)$ simultaneously, which is popularly needed in analysis after designs, especially when the covariates have non-zero means.       In this paper, we consider general covariate vectors which contain both discrete covariates and continuous covariates.
To make the procedures and results widely practicable, we consider a general allocation function $\ell(x)$.
The assumption \eqref{eq:allocationfunction1} for $\ell(x)$ means that the allocation probability defined as in \eqref{eq:allocation-2} is symmetric about treatments. The assumption \eqref{eq:allocationfunction2} is used to avoid the complete randomization in which the imbalance measure $Imb_n$ is increasing in $n$ with order $n$. The assumption \eqref{eq:allocationfunction1} is used to avoid the deterministic allocation in which the unit may be assigned to treatment or control with probability $1$ or $0$.
A continuous allocation function can be chosen as
\begin{equation}\label{eq:normallocation}
\ell (x)= 1-\Phi((-D)\vee x\wedge D) ,
\end{equation}
 where   $\Phi$ is the standard normal distribution and $D$ is a positive constant.
In practice, $D$ can be chosen to be $3$.

\subsection{Theoretical Properties}


In this subsection we shall study the theoretical properties of the CAR  procedures with a feature map $\bm\phi(\bm{x})=(\phi_1(\bm{x}),\ldots,\phi_q(\bm{x})):\mathbb{R}^p\mapsto\mathbb{R}^q$.
Our goal is to prove the convergence rate of covariate imbalance measured by
$\bm{\Lambda}_n=\sum_{i=1}^n(2T_i-1)\bm\phi(\bm{X}_i)$.

To obtain the asymptotic properties, we require the following assumptions.

\begin{assumption}\label{asm:iid}
The covariate sequence $\{\bm{X}_i=(X_{i,1}, \ldots, X_{i,p})\}$ is a sequence of independent and identically distributed random vectors.
\end{assumption}

\begin{assumption}\label{asm:moment}
  The feature map $\phi(\bm{X})$ satisfies that $\Ex[\|\phi(\bm{X})\|^\gamma]$ is finite, $\gamma\ge 2$.
\end{assumption}

Assumption \ref{asm:iid} is the usual condition required on the sequence of covariates in a randomized experiment.
Under the proposed procedure, it ensures that $\{\bm{\Lambda}_n\}$ is a Markov chain on the space $\mathbb{R}^{q}$.

  To study the properties of  the imbalance $\bm \Lambda_n$ of the specified covarite feature $\bm \phi(\bm X)$ and  
   the imbalance $\bm \Lambda_n$ of the unspecified covarite feature $m(\bm X)$, we need a more assumption on the distribution of $\bm\phi(\bm X)$. In the sequel, for a $d_1$-dimensional discrete random vector $\bm Y_1$ taking values in  $\mathbb{Y}_1=\{\bm y_1^{(1)},\bm y_2^{(1)},\ldots\}$ and a $d_2$-dimensional continuous random vector $\bm Y_2$, the conditional probability $\Prob(\bm Y_2\in A|\bm Y_1=\bm y)$ is defined to be zero if $\Prob(\bm Y_1=\bm y)=0$, and $f(\bm y)$ is called a density of $\bm Y=(\bm Y_1,\bm Y_2)$ if
$$\Prob(\bm Y\in A)= \sum_{\bm y_1\in \mathbb{Y}_1}\int_{\bm y_2\in \mathbb R^{d_2}}  f(\bm y_1,\bm y_2 )I\{(\bm y_1,\bm y_2)\in A\}d\bm y_2. $$
We still write the summation-integral in the left hand by an integral $\int_{A}f(\bm y)d\bm y$.
\begin{assumption}\label{asm:recurrent} Suppose  that $\bm\phi(\bm X)$ has the form $\bm \phi(\bm X)=\big(\bm\phi^{(1)}(\bm X), \bm\phi^{(2)}(\bm X)\big)$ with distribution $\Gamma_{\phi}$, where $\bm \phi^{(1)}(\bm X)$ is a $q_1$-dimensional discrete vector with distribution $\Gamma_{\phi^{(1)}}$, $\bm\phi^{(2)}(\bm X)$ is a $(q-q_1)$-dimensional continuous vector with distribution $\Gamma_{\phi^{(2)}}$, and $\Gamma_{\phi^{(2)}}(\cdot|\bm x^{(1)})$ is the conditional distribution of $\bm\phi^{(2)}(\bm X)$ for given $\bm\phi^{(1)}(\bm X)=\bm x^{(1)}$ .

(i) For the discrete  part  $\bm\phi^{(1)}(\bm X)$, let $\mathscr{A}=\{\bm a_1,\bm a_2,\ldots,\}$ be all possible values of $\bm\phi^{(1)}(\bm X)$, and denote $\mathscr{X}=\{\sum_{i=1}^r n_i\bm y_i: \bm y_i\in \mathscr{A}, n_i\in \mathbb Z, r\ge 1\}$ be the linear space spanned by $\mathscr{A}$ with integer coefficients. We assume that for any bounded set $B\subset  \mathbb R^{q_1}$, $B\cap \mathscr{X}$ has only a finite number of different points.

Without loss of generality, we assume that $\bm\phi^{(1)}(\bm X)$ has a positive probability at each $\bm a_i$ for otherwise we can remove those $\bm a_i$s with mull  probabilities.  

(ii) For the continuous part  $\bm\phi^{(2)}(\bm X)$, we assume that  there is an $n_0$, a $\bm y_0\in \mathscr{X}$, a density $\nu(\bm x)$ and a constant $c_{\nu}>0$ such that
\begin{equation}\label{eq:asm:rencurrent1} \Prob\Big(\sum_{i=1}^{n_0}\bm\phi^{(2)}(\bm X_i)\in A\big|\sum_{i=1}^{n_0}\bm\phi^{(1)}(\bm X_i)=\bm  y_0\Big)\ge c_{\nu}\int_A \nu(\bm x) d\bm x, \text{ for all } A.
\end{equation}

\end{assumption}

\begin{remark} Assumption \ref{asm:recurrent} will be used to show the irreducibility of the Markov Chain $\{\bm \Lambda_n\}$.  Assumption \ref{asm:recurrent} (i) is satisfied if each element of $\bm\phi^{(1)}(\bm X)$ takes possible values in lattices. In particular, if each element of $\bm\phi^{(1)}(\bm X)$ takes a finite number of rational values, then Assumption \ref{asm:recurrent} (i) is satisfied.

For Assumption  \ref{asm:recurrent} (ii), it is sufficient that \eqref{eq:asm:rencurrent1} holds for all subsets $A$ of a set $B$ with $\nu(B)=\int_B \nu(\bm x)>0$, because $\nu(\bm x)$ can be replaced by $\widetilde{\nu}(\bm x)=\frac{\nu(\bm x)I\{\bm x\in B\}}{\nu(B)}$.
When the $n_0$-th convolution of $\Gamma_{\phi}$ has a density $\mu(\bm x^{(1)},\bm x^{(2)})$ on $\mathscr{X}\times \mathbb R^{q-q_1}$, then \eqref{eq:asm:rencurrent1} is satisfied with $\nu(\bm x)=\frac{\mu(\bm y_0,\bm x)}{\int \mu(\bm y_0,\bm x^{(2)})d x^{(2)}}$. Further, if there exist $\bm y^{(1)},\ldots,\bm y^{(n_0)}\in \mathscr{A}$ such that the convolution $\Gamma_{\phi^{(2)}}(\cdot|\bm y^{(1)})\ast\Gamma_{\phi^{(2)}}(\cdot|\bm y^{(2)})\ast\cdots\ast\Gamma_{\phi^{(2)}}(\cdot|\bm y^{(n_0)})$ of the conditional distributions has a density $\nu(\bm x)$, or more generally, the convolution satisfies the inequality \eqref{eq:asm:rencurrent1} for a density $\nu(\bm x)$, then  \eqref{eq:asm:rencurrent1} is satisfied with $\bm y_0=\bm y^{(1)}+\cdots+\bm y^{(n_0)}$. Hence, Assumption \ref{asm:recurrent} is widely satisfied in practices.

 When $\bm\phi^{(1)}(\bm X)$ and $\bm \phi^{(2)}(\bm X)$ are independent, the probability in the left hand of \eqref{eq:asm:rencurrent1} is $\Gamma_{\phi^{(2)}}^{ \ast n_0 }(A)$,
where $\Gamma_{\phi^{(2)}}^{\ast k}$ is the $k$-th convolution of $\Gamma_{\phi^{(2)}}$.  In such case, Assumption \ref{asm:recurrent} (ii) is used in \cite{Ma2022}. Here, the total covariate feature $\bm\phi(\bm X)$  can be a combination  of both discrete and continuous covariates.

\end{remark}

\begin{theorem} \label{thm:recurrent} Suppose that  Assumptions \ref{asm:iid}, \ref{asm:moment} (with $\gamma \ge 2$) and \ref{asm:recurrent} hold.
Then $\bm{\Lambda}_n$ is a positive Harris recurrent Markov chain with an      invariance probability measure   $\pi$  and  $\Ex[\|\bm{\Lambda}_n\|^{\gamma-1}]=O(1)$.
\end{theorem}

 Theorem  \ref{thm:recurrent} tells us that $\sum_{i=1}^n (2T_i-1)\bm\phi(\bm X_i)$ is bounded in probability. Besides the covariate feature $\bm\phi(\bm X)$ which is considered in the randomization procedure, there are  many other  known or unknown covariate features possibly needed to be balanced at the same time. Such covariate feature can be written as a known or unknown function $m(\bm X)$ of $\bm X$. The following theorem tells  us that for any other function $m(\bm X)$ of $\bm X$, $\sum_{i=1}^n (2T_i-1)m(\bm X_i)$ can be bounded by the rate $\sqrt{n}$ in probability.

\begin{theorem} \label{thm:clt} Suppose that Assumptions \ref{asm:iid}, \ref{asm:moment} (with $\gamma \ge 3$) and \ref{asm:recurrent} hold. Assume that $m(\bm X)$ is a function of $\bm X$ with $\Ex[m^2(\bm X)]<\infty$.
Then there is a $\vec{\sigma}_m\ge 0$ such that
\begin{equation}\label{eq:clt} \frac{\sum_{i=1}^n (2T_i-1)m(\bm X_i)}{\sqrt{n}}\overset{D}\to N(0,\vec{\sigma}_m^2) \text{ and }
\frac{\Ex\left(\sum_{i=1}^n (2T_i-1)m(\bm X_i)\right)^2}{n} \to  \vec{\sigma}_m^2.
\end{equation}
Here and in the sequel, we use the superscript $\to$  to distinguish the above asymptotic variance  from other variances. 
\end{theorem}

For Theorem \ref{thm:clt}, we have a more general result which  show that if $\sum_{i=1}^n (2T_i-1)m(\bm X_i)=o_P(\sqrt{n})$, then $m(\bm X)$ must be a linear function of $\bm\phi(\bm X)$.

\begin{theorem}\label{thm:generalclt}
Assume that  $(\bm X_i, Z_i,W_i)$, $i=1,2,\ldots$, are i.i.d. random vectors with the same distribution as $(\bm X,Z,W)$, and $\Ex[Z^2]<\infty$, $\Ex[W^2]<\infty$, $\Ex[W]=0$.  Suppose that Assumptions \ref{asm:iid}, \ref{asm:moment} (with $\gamma \ge 3$) and \ref{asm:recurrent} hold. Then there is a $\sigma\ge 0$ such that
\begin{equation}\label{eq:corclt} \frac{\sum_{i=1}^n \left\{(2T_i-1)Z_i+W_i\right\}}{\sqrt{n}}\overset{D}\to N(0,\sigma^2) \text{ and }
\frac{\Ex\left(\sum_{i=1}^n \left\{(2T_i-1)Z_i+W_i\right\}\right)^2}{n} \to  \sigma^2,
\end{equation}
where $\sigma^2=\Ex(Z-\Ex[Z|\bm X])^2+\Ex[W^2]+\vec{\sigma}_m^2$ and $\vec{\sigma}_m^2$ is the same as that in \eqref{eq:clt} with $m(\bm X)=\Ex[Z|\bm X]$. Also
\begin{equation}\label{eq:LIL}
\sum_{i=1}^n \left\{(2T_i-1)Z_i+W_i\right\}=O(\sqrt{n\log\log n})\; a.s.
\end{equation}

Furthermore,  
$\sigma^2=\Ex[W^2]$ if and only if $Z=\langle\bm x_0,\bm\phi(\bm X)\rangle$ a.s. for some    $\bm x_0\in \mathbb R^{q}$. Here and in the sequel, $\langle\bm x,\bm y\rangle=\bm x\bm y^{\prime}$ is the inner product of vectors $\bm x$ and $\bm y$.
 \end{theorem}
\begin{remark} With $\bm\phi(\bm X)$ and $\Ex[m(\bm X)|\bm\phi(\bm X)]$ taking the place of $\bm X$ and $m(\bm X)$, respectively,  applying \eqref{eq:corclt} with $Z=m(\bm X)$ yields
\begin{equation}\label{eq:variance-m-conditionM}\vec{\sigma}_m^2=\Var\big(m(\bm X)-\Ex[m(\bm X)|\bm\phi(\bm X)]\big) +\vec{\sigma}_{\Ex[m(\bm X)|\bm\phi(\bm X)]}^2.
\end{equation}
\end{remark}

 Theorems \ref{thm:clt} and \ref{thm:generalclt} are important because they show that, when the chosen feature $\bm\phi(\bm X)$ satisfies Assumption  \ref{asm:recurrent}, the proposed CAR procedure not only can balance the prescribed covariate features $\bm\phi(\bm X_i)$ in the best rate $O_P(1)$, but also can balance the unknown function $m(\bm X_i)$ or unobserved covariates $Z_i$ in a rate  $O_P(\sqrt{n})$ despite of the covariates have zero means or not when the variance of $m(\bm X_i)$ or $Z_i$ is finite. \eqref{eq:corclt} is important also because it can be a basic tool in statistical inference after randomization.
  The following corollary tells us that, if only the mean is finite, then the convergence rate is $o_P(n)$.

 \begin{corollary}\label{cor:LLN}
Assume that  $(\bm X_i, Z_i)$, $i=1,2,\ldots$, are i.i.d. random vectors with the same distribution as $(\bm X,Z)$, and $\Ex[|Z|]<\infty$.  Suppose that Assumptions \ref{asm:iid}, \ref{asm:moment} (with $\gamma \ge 3$) and \ref{asm:recurrent} hold. Then
\begin{align*}
& \frac{\Ex\left|\sum_{i=1}^n (2T_i-1)Z_i\right|}{n} \to  0, \;\;  \Ex\left|\frac{\sum_{i=1}^n T_iZ_i}{n} -\frac{1}{2}\Ex Z\right|\to 0\\
\text{ and } &\frac{ \sum_{i=1}^n (2T_i-1)Z_i }{n} \to  0 \;a.s., \;\; \frac{\sum_{i=1}^n T_iZ_i}{n} \to \frac{1}{2}\Ex Z\; a.s.
\end{align*}
 \end{corollary}
The results follow from \eqref{eq:corclt} and \eqref{eq:LIL} immediately by noting
$$
\frac{1}{n}\Big|\sum_{i=1}^n (2T_i-1)Z_i\Big|\le   \frac{\Big|\sum_{i=1}^n (2T_i-1)Z_iI\{|Z_i|\le C\}\Big|}{n}+\frac{\sum_{i=1}^n |Z_i|I\{|Z_i|> C\}}{n}
$$
and $\sum_{i=1}^n T_iZ_i=\frac{1}{2}\sum_{i=1}^n (2T_i-1)Z_i+\frac{1}{2}\sum_{i=1}^nZ_i$.

Assumption \ref{asm:recurrent} is a key condition to obtain that $\sum_{i=1}^n(2T_i-1)\bm\phi(\bm X_i)$ is a  positive Harris recurrent Markov chain so that it is bounded in probability.  If Assumption \ref{asm:recurrent} is not satisfied, we still have 
$\sum_{i=1}^n(2T_i-1)\bm\phi(\bm X_i)=O_P(1)$ if $\|\bm \phi(\bm X)\|$has moments of order higher than two.

\begin{theorem}\label{thm:DesignPropertyGeneral}
Suppose that Assumptions  \ref{asm:iid} and \ref{asm:moment} with $\gamma>2$ hold and the assumption \eqref{eq:allocationfunction3} can be removed. Then  $\Ex[\|\bm{\Lambda}_n\|^{r}]=O(1)$ for all $r\in (0,\gamma-1)$.  Furthermore,    $\|\bm \Lambda_n\|=o(n^{\frac{1}{\gamma}+\epsilon})$ a.s. for any $\epsilon>0$.
\end{theorem}
\begin{remark}
The convergence rate $o_P(\sqrt{n})$ of $\sum_{i=1}^n(2T_i-1)\bm\phi(\bm X_i)$ is usually needed in   the analysis under the covariate-adaptive randomization. In some situations, especially  when the response is heteroscedastic, we also need the convergence rate $o_P(n)$ of $\sum_{i=1}^n(2T_i-1)m(\bm X_i)$ for unknown $m(\bm X_i)$ which is usually not contained in $\bm\phi(\bm X_i)$. We will give more  discussions in Section \ref{sec:Inference}.  Theorem  \ref{thm:DesignPropertyGeneral} has not given such a convergence rate. Whether $\sum_{i=1}^n(2T_i-1)m(\bm X_i)=o_P(n)$ or not without Assumption \ref{asm:recurrent} is an open problem.

In Theorem \ref{thm:DesignPropertyGeneral}, the assumption  \eqref{eq:allocationfunction3} is not needed, which means that the results are valid for the minimization procedure where the $n$-th unit is allocated deterministically to the treatment when $Imb_n^{(1)}<Imb_n^{(2)}$, and to the control when $Imb_n^{(1)}>Imb_n^{(2)}$. 
\end{remark}

All the above theorems are special cases of the results on general multi-treatment CAR procedures which is stated in Section \ref{sec:multi}. \cite{Ma2022} has obtained similar results as  those in Theorems \ref{thm:recurrent}, \ref{thm:clt}, \ref{thm:DesignPropertyGeneral} and \eqref{eq:corclt}.  Comparing,  our proposed framework with theories for the two-treatment case has four main contributions:
\begin{itemize}
  \item[(1)]  First, the allocation function for defining the allocation properties is no longer limited to the type of  Efron's biased coin allocation. The flexibility of $\ell(\bm x)$ can generate a  large class of allocation procedures. For Efron's biased coin allocation, the allocation probability only considers the signum of $Imb_n^{(1)}-Imb_n^{(2)}$ regardless of the size no matter how large it is. If one wants to allocate the unit with larger $|Imb_n^{(1)}-Imb_n^{(2)}|$ to treatments by a more biased probability, he/she can choose a continuous allocation function which has involved the size of $|Imb_n^{(1)}-Imb_n^{(2)}|$. For more discussion, one can refer to \cite{Antognini2004}.
  \item[(2)] Secondly, the rate $O(1)$ of the convergence in Theorem \ref{thm:DesignPropertyGeneral} is much faster than $O(n^{1/(2\gamma-2)})$ that of \cite{Ma2022} under the same moment condition. 
  \item[(3)] Thirdly and most importantly, in Theorems \ref{thm:recurrent}-\ref{thm:generalclt},  the covariate feature $\bm\phi(\bm X)$ can be a combination of discrete and continuous covariates instead of a purely continuous vector as in \cite{Ma2022}. This leads to a much broader range of CAR procedures that can be generated by the proposed  CAR procedures.  For example, if balancing a continuous covariate vector $(X_1,\ldots, X_p)$ and balancing treatment allocations are required simultaneously, a CAR procedure can be defined with $\phi(\bm X)=(1, X_1,\ldots, X_d)$  which is a combination of discrete covariate $1$ and continuous covariates $X_1,\ldots, X_p$. More discussions will be given in subsection \ref{sec:Examples} by examples.

      Also, because of the flexibility to define various feature maps $\bm \phi(\bm X_i)$, the proposed CAR procedure with the theoretical properties has unified many existing models such as the stratified randomization procedure, \cite{Pocock1975}'s procedure and  \cite{Hu2012}'s procedure for discrete covariates and several randomization procedures for balancing continuous covariates proposed by recently by \cite{Li2018},   \cite{Qin2016} etc.
  \item[(4)] At last, Theorem \ref{thm:generalclt} gives the result that $\sum_{i=1}^n(2T_i-1)Z_i=o_P(\sqrt{n})$ if and only if $Z$ is a linear function of the covariate feature $\bm \phi(\bm X)$. This provides a new insight into the properties of the CAR procedures including the \cite{Pocock1975}'s procedure.
\end{itemize}

\subsection{Examples}\label{sec:Examples}
 In this subsection, we show by examples how to balance discrete, continuous, and mixed covariates by applying the $\bm\phi$-based CAR procedure.

 We first consider the cases with discrete covariate $\bm X=(X_1,\ldots, X_p)$.
 \begin{example}\label{ex:discrete1}
 Suppose that each $X_t$ takes     finite many values $x_t^{(l_t)}$, $l_t=1\ldots, L_t$.
 \begin{description}
   \item[\rm (i)] Let $\bm \phi=\bm \phi_{S}(\bm X)=\big(I\{\bm X=(x_1^{(l_1)},\ldots, x_p^{(l_p)})\}; l_t=1,\ldots, L_t,t=1,\ldots, p\big)$.  Then the $\bm \phi$-based CAR procedure is the stratified randomization procedure.
   \item[\rm (ii)]  Let  $\bm \phi=\bm \phi_{PS}(\bm X)=\big(\sqrt{w_t}I\{X_t=x_t^{(l_t)}\}; l_t=1,\ldots, L_t,t=1,\ldots, p\big)$, where $w_t$s are positive weights. Then the $\bm \phi$-based CAR procedure is \cite{Pocock1975}'s marginal procedure.
   \item[\rm (iii)]  Let  $\bm \phi=\bm \phi_{HH}(\bm X)=\big(\sqrt{w_o}; \sqrt{w_{m,t}}I\{X_t=x_t^{(l_t)}\};\sqrt{w_s}I\{\bm X=(x_1^{(l_1)},\ldots, x_p^{(l_p)})\}; l_t=1,\ldots, L_t,t=1,\ldots, p\big)$, where $w_0,w_{m,t}, w_s$ are non-negative weights with $w_0+\sum_tw_{m,t}+w_s\ne 0$. Then the $\bm \phi$-based CAR procedure is \cite{Hu2012}'s procedure.
 \end{description}
 \end{example}

Note that $wI\{X=c\}$ takes possible values in lattices $wk$, $k=0,1$.  $\bm \phi$ in the above example satisfies Assumptions  \ref{asm:moment}  and \ref{asm:recurrent}. By  Theorem \ref{thm:recurrent}, $\sum_{i=1}^n (2T_i-1)\bm\phi(\bm X_i)=O_P(1)$.
 \cite{Hu2012} obtain the theoretical properties for their proposed design under a strict condition (condition (C) of Theorem 3.2 in \cite{Hu2012}). Under this condition, comparing with $w_s$, other weighs $w_0,w_{m,t}$ should be very small and the procedure is very close to the stratified randomization procedure.
Now, since each function
$$m(\bm X)=\sum_{l_t=1,\ldots,L_t,t=1,\ldots,p} m(x_1^{(l_1)},\ldots, x_p^{(l_p)}) I\{\bm X=(x_1^{(l_1)},\ldots, x_p^{(l_p)})\} $$
is a linear function of $\bm\phi_S(\bm X )$ and $\bm \phi_{HH}(\bm X)$ with $w_s\ne 0$,   under the stratified randomization procedure or \cite{Hu2012}'s procedure, we always have
$ \sum_{i=1}^n (2T_i-1)m(\bm X_i)=O_P(1)$ for any function $m(\bm X)$. In particular,  the overall imbalance  $D_n=\sum_{i=1}^n (2T_i-1)$, the marginal imbalance $D_n(t,l_t)=\sum_{i=1}^n(2T_i-1)I\{X_t=x_t^{(l_t)}\}$ and the within-stratum imbalance $D_n(l_1,\ldots,l_p)=\sum_{i=1}^n (2T_i-1)I\{\bm X=(x_1^{(l_1)},\ldots, x_p^{(l_p)})\}$ are all bounded in probability.

For  Pocock and Simon's procedure, since each function of $X_t$ can be written as
$$f(X_t)=\sum_{l_t=1}^{L_t}f(x_t^{(l_t)})I\{X_t=x_t^{(l_t)}\} $$
and is a linear function of $\bm\phi_{PS}(\bm X)$,  we have
$ \sum_{i=1}^n (2T_i-1)f(X_{i,t})=O_P(1)$ for any function $f(X_t)$. In particular,  the overall imbalance  $D_n$ and the marginal imbalance $D_n(t,l_t)$  are bounded in probability.  Furthermore, for any function $m(\bm X)$ of $\bm X$, by Theorems \ref{thm:clt} and \ref{thm:generalclt}, \eqref{eq:clt} holds  with
a $\vec{\sigma}_m\ge 0$, and $\vec{\sigma}_m=0$ if and only if $m(\bm X)$ is a linear function of $\bm\phi_{PS}(\bm X)$ which is equivalent to $m(\bm X)=\sum_{t=1}^p f_t(X_t)$ for some  functions $f_1,\ldots, f_p$.
In particular, if $m(\bm X)=I\{\bm X=(x_1^{(l_1)},\ldots, x_p^{(l_p)})\}$, then $\vec{\sigma}_m^2>0$  when $\bm X$ takes each $(x_1^{(l_1)},\ldots, x_p^{(l_p)})$  with a positive probability. That is, for the within-stratum imbalance $D_n(l_1,\ldots,l_p)$,
$$ \lim_{n\to\infty}\frac{\Ex\big[\big(D_n(l_1,\ldots, l_p)\big)^2\big]}{n}=\sigma_{l_1,\ldots,l_p}^2>0. $$
Hence, under   Pocock and Simon's procedure, the within-stratum imbalances increase with rate $\sqrt{n}$ as the sample increases.

The theoretical properties of Pocock and Simon's procedure and the procedure in Example \ref{ex:discrete1} (iii) are studied by  \cite{Hu2012} (for special weights), \cite{Ma2015} and \cite{HuZhang2020}. All the results shown are corollaries of Theorems \ref{thm:recurrent}-\ref{thm:generalclt}. For example, in (iii) if $w_s=w_{m,t}=0$, then $I\{X_t=x_t^{(l_t)}\}$ is not  a linear function of $\sqrt{w_o}$, $\sqrt{w_{m,s}}I\{X_s=x_s^{(l_s)}\}$;  $ l_s=1,\ldots, L_s, s\ne t$.
And so, when $\bm X$ takes each $(x_1^{(l_1)},\ldots, x_p^{(l_p)})$  with a positive probability, the marginal imbalance $D_n(t,l_t)$  for treatment $t$  satisfies
$$ \lim_{n\to\infty} \frac{\Ex[(D_n(t,l_t))^2]}{n}=\sigma^2_{t,l_t}>0, $$
which is shown in \cite{HuZhang2020}.

Summarily, we have the following corollary on the stratified randomization procedure, Pocock and Simon's procedure, and Hu and Hu's procedure.
\begin{corollary}\label{cor:HuHu} Suppose $\bm X=(X_1,\ldots,X_p)$ with  each $X_t$ taking    finite many values $x_t^{(l_t)}$, $l_t=1\ldots, L_t$. Consider a $\bm\phi$-based CAR procedure with $\bm\phi(\bm X)=\bm\phi_{HH}(\bm X)$, i.e., Hu and Hu's procedure is to used for randomizing units.
\begin{itemize}
  \item[(i)] Suppose $w_s\ne 0$. Then   we have $\sum_{i=1}^n(2T_i-1)m(\bm X_i)=O_P(1)$ for any function $m(\bm X)$ of $\bm X$.
  \item[(ii)] Suppose $w_s=0$. Then for the covariate $X_t$ with $w_{m,t}\ne 0$ we have $\sum_{i=1}^n(2T_i-1)f_t(X_{i,t})=O_P(1)$ for any function $f_t(X_t)$ of $X_t$.  Furthermore, for  any function $m(\bm X)$ of $\bm X$, there exists a $\vec{\sigma}_m\ge 0$ such that
      $$ \lim\frac{\Ex(\sum_{i=1}^n(2T_i-1)m(\bm X_i))^2}{n}=\vec{\sigma}_m^2\; \text{ and }\; \frac{\sum_{i=1}^n(2T_i-1)m(\bm X_i)}{\sqrt{n}}\overset{D}\to N(0,\vec{\sigma}_m^2). $$
  Morever, $\vec{\sigma}_m=0$ if and only if $m(\bm X)= \sum_{s: w_{m,s}\ne 0}f_s(X_s)$ for some functions $f_1(X_1)$, $\ldots$, $f_p(X_p)$.
  \item[(iii)] For any case, $\sum_{i=1}^n(2T_i-1)=O_P(1)$.
\end{itemize}

\end{corollary}
By this corollary, \cite{Pocock1975}'s procedure can and only can balance the functions of the covariates with the additive form $\sum_{s}f_s(X_s)$ in a convergence rate $O_P(1)$.

Next, we consider the continuous covariates.
\begin{example} \label{ex:continuous1}Suppose that the   covariates $ X_1,\ldots, X_p$ have a  joint density with the finite  $2$-nd moments. For balancing covariates $X_1,\ldots, X_p$ and treatment allocations,  we choose $\bm\phi(\bm x) =(1,x_1,x_2,\ldots,x_p)=(\phi^{(1)}(\bm x), \bm\phi^{(2)}(\bm x))$  with $\phi^{(1)}(\bm x)\equiv 1$ and  $\bm\phi^{(2)}(\bm x)=(x_1,x_2,\ldots,x_p)$. Hence, Assumptions \ref{asm:moment} (with $\gamma=2$) and \ref{asm:recurrent} are satisfied. By Theorem  \ref{thm:recurrent}, we have $\sum_{i=1}^n (2T_i-1)=O_P(1)$ and $\sum_{i=1}^n (2T_i-1)X_{i,t} =O_P(1)$, $t=1,\ldots, p$.

Furthermore, if the $3$-rd moments of $X_t$s are finite and $m(x_1,\ldots,x_p)$ is a function of $x_1,\ldots,x_p$ with $\Ex[m^2(X_1,\ldots,X_p)]<\infty$, then, by   Theorem \ref{thm:clt}, \eqref{eq:clt} holds  with
a $\vec{\sigma}_m\ge 0$, and $\vec{\sigma}_m=0$ if and only if $m(x_1,\ldots,x_p)=\alpha_0+\sum_{j=1}^p \alpha_jx_j$.
\end{example}
This example is considered in \cite{Ma2022}. But, in their Theorems 3.4 and 3.6, the constant $1$ can not be put into the feature $\bm \phi(\bm X)$  and so the overall imbalance $\sum_{i=1}^n (2T_i-1)$ is not bounded in probability and increases with rate $\sqrt{n}$ as the sample increases so that the treatment effect analysis can not deal with the case that the observed covariates have no-zero means.

\begin{example} \label{ex:continuous3} Suppose that the two-dimensional covariate $\bm X=(X_1,X_2)$ has a continuous joint density with the finite $4$-th moment. For balancing the covariates $X_1$, $X_2$ and the interaction $X_1X_2$, we choose
$\bm\phi(\bm x)=(1,x_1,x_2,x_1x_2) =(\phi^{(1)}(\bm x), \bm\phi^{(2)}(\bm x))$ with $\phi^{(1)}(\bm x)\equiv 1$ and  $\bm\phi^{(2)}(\bm x)=(x_1,x_2,x_1x_2)$. Let $y_j=\sum_{i=1}^2 x_{i,j}$, $j=1,2$, $y_3=\sum_{i=1}^2x_{i,1}x_{i,2}$, $y_4=\sum_{i=1}^2x_{i,1}^2$. Then,   $(X_{1,1}, X_{1,2}, X_{2,1}, X_{2,2})$ has a continuous joint density, and the Jacobian determinant of the map $(x_{1,1},x_{1,2},x_{2,1},x_{2,2})\to (y_1,y_2,y_3, y_4)$ is
$$ \frac{\partial (y_1,y_2,y_3,y_4)}{\partial (x_{1,1}, x_{1,2},x_{2,1},x_{2,2})}=\left|\begin{matrix} 1 & 0 & 1 & 0\\ 0 & 1 & 0 &1 \\ x_{1,2} & x_{1,1} & x_{2,2} & x_{2,1} \\ 2x_{1,1} & 0 & 2x_{2,1} &0 \end{matrix}\right|
=-2(x_{2,1}-x_{1,1})^2\ne 0\;\; a.e.  $$
Thus, $(Y_1,Y_2,Y_3,Y_4)$ has a positive joint density $f(\bm y)$ on some subset $K_1\times K_2\times K_3\times K_4 \subset \mathbb R^4$, i.e.,
$$ \Prob\left(Y_1,Y_2,Y_3,Y_4)\in A\right) =\int_{\bm y\in A}f(\bm y)d\bm y, \;\; A\subset K_1\times K_2\times K_3\times K_4. $$
It follows that
$$ \Prob\left(Y_1,Y_2,Y_3)\in A\right) \ge\int_{(y_1,y_2,y_3)\in A}\left[\int_{y_4\in K_4}f(y_1,y_2,y_3,y_4)dy_4\right]d y_1d y_2 d y_3, \;\; A\subset K_1\times K_2\times K_3. $$
  Hence, Assumptions \ref{asm:moment} (with $\gamma=2$) and \ref{asm:recurrent} are satisfied. By Theorem  \ref{thm:recurrent}, we have $\sum_{i=1}^n (2T_i-1)=O_P(1)$, $\sum_{i=1}^n (2T_i-1)X_{i,j} =O_P(1)$, $j=1, 2$, and $\sum_{i=1}^n(2T_i-1) X_{i,1}X_{i,2} =O_P(1)$.

Furthermore, if the $6$-th moments of $X_1$ and $X_2$ are finite and $m(\bm x)$ is a function of $\bm x$ with $\Ex[m^2(\bm X)]<\infty$, then, by   Theorem \ref{thm:clt}, \eqref{eq:clt} holds  with
a $\vec{\sigma}_m\ge 0$, and $\vec{\sigma}_m=0$ if and only if  $m(\bm x)=\alpha_0+\alpha_1x_1+\alpha_2x_2+\alpha_{1,2}x_1x_2$.

For general multi-dimensional covariates $\bm X=(X_1,\ldots, X_p)$, we may have similar results with a much more complex proof. For example,
  for balancing a three-dimensional continuous covariate $\bm X=(X_1,X_2,X_3)$ and the   interactions $X_1X_2$, $X_2X_3$, $X_3X_1$,  we choose
$\bm\phi(\bm x) =(\phi^{(1)}(\bm x), \bm\phi^{(2)}(\bm x))$ with $\phi^{(1)}(\bm x)\equiv 1$ and  $\bm\phi^{(2)}(\bm x)=(x_1,x_2,x_3,x_1x_2,x_2x_3,x_3x_1)$. Let $y_j=\sum_{i=1}^2 x_{i,j}$, $j=1,2,3$, $y_4=\sum_{i=1}^2x_{i,1}x_{i,2}$, $y_5=\sum_{i=1}^2x_{i,2}x_{i,3}$, $y_6=\sum_{i=1}^2x_{i,3}x_{i,1}$. Then,   $(\bm X_1,\bm X_2)$ has a continuous joint density and  the map $(x_{1,1},x_{1,2},x_{1,3},x_{2,1},x_{2,2},x_{2,3})\to (y_1,\ldots, y_6)$  has the Jacobian determinant $2(x_{2,1}-x_{1,1})(x_{2,2}-x_{1,2})(x_{2,3}-x_{1,3})\ne 0\;\; a.e. $ Hence,  $(Y_1,\ldots,Y_6)=\sum_{i=1}^2\bm\phi^{(2)}(\bm X_i)$  has a positive joint density $f(\bm y)$ on some subset $K\subset \mathbb R^6$  and Assumption  \ref{asm:recurrent} is satisfied.
\end{example}

\begin{example} \label{ex:mixing}Suppose that $\bm X=(X_1,\ldots,X_p)$ can be written as $(\bm X_{dis},\bm X_{con})$, where the discrete part $\bm X_{dis}$   takes finite many values, and, for some $n_0$ and $\bm x_{dis}$, the $n_0$-th convolution   of the conditional distribution of   $\bm X_{con}$ for given $ \bm X_{dis}=\bm x_{dis}$ has a density. For balancing covariates $X_1,\ldots, X_p$ and treatment allocations,  we choose $\bm\phi(\bm x) =(\phi^{(1)}(\bm x), \bm\phi^{(2)}(\bm x))$ with  $\bm\phi^{(2)}(\bm x)=\bm x_{con}$, $\bm\phi^{(1)}(\bm x)=\bm\phi_{PS}(\bm x_{dist})$ or $\bm\phi_{HH}(\bm x_{dist})$.  Then by Theorem  \ref{thm:recurrent}, we have $\sum_{i=1}^n (2T_i-1)=O_P(1)$, $\sum_{i=1}^n (2T_i-1)X_{i,j} =O_P(1)$, $j=1, \cdots, p$, when each element of $\bm X_{con}$ has finite variance.  And by Theorem \ref{thm:clt}, $\sum_{i=1}^n (2T_i-1)m(\bm X_i)=O_P(\sqrt{n})$ when each element of $\bm X_{con}$ has the finite third moment and $\Ex[m^2(\bm X)]<\infty$.
\end{example}

\section{Test for Treatment Effect with  Heteroscedasticity Models }\label{sec:Inference}
\setcounter{equation}{0}

Because of the high prevalence of covariate-adaptive randomization, it is important to verify the validity of conventional tests under covariate-adaptive randomization.  Here, the validity of a test procedure refers to the type I error
rate of the test being no larger than a given level of significance, at least in the limiting sense. In this section, we study hypothesis testing of treatment effects under
covariate-adaptive designs. Suppose that two treatments are studied under a
covariate-adaptive randomized clinical trial.
   For each unit $i$, denote $Y_i(1)$ and $Y_i(0)$ as the potential outcomes under the treatment and control, respectively.  The observed outcome is
$
Y_i = T_i Y_i(1) + ( 1 - T_i ) Y_i(0).
$
The treatment effect is defined as $\tau =\Ex \{ Y_i(1)\}-\Ex \{ Y_i(0)\}$.
Also, denote $\bm X_i=(X_{i,1},\ldots, X_{i,p+q})$   as the covariate of unit $i$, where only the first
$p$ observed covariates $\bm X_{i, obs}=(X_{i,1},\ldots, X_{i,p})$ are used both in statistical analysis and randomization,
and the last $q$  covariates $\bm X_{i, rd}=(X_{i,p+1},\ldots, X_{i,p+q})$ are only used in randomization. Assume that $(Y_i(0), Y_i(1), \bm X_i)$, $i=1,\ldots, n$, are  independent and identically distributed as $(Y(0), Y(1), \bm X)$.  We suppose that the units are randomized sequentially by a CAR procedure with feature $\bm\phi(\bm X)$.
We consider the test of
\begin{equation}\label{htest} H_0:\tau=:\Ex \{ Y_i(1)\}-\Ex \{ Y_i(0)\} = 0\;\;  \text{versus} \;\; H_A : \tau\ne  0.
\end{equation}

\subsection{The Model for Responses}\label{sec:Inference-additive}
The test will depend on the real model and working model. \cite{Shao2010}, \cite{Ma2015}, \cite{Ma2020} etc consider  the test of \eqref{htest}
 basing on the following underlying linear model:
\begin{equation}\label{eq:linearmodel} Y_i=T_i\mu_1+(1-T_i)\mu_0+\sum_{j=1}^p\beta_jX_{i,j}+\sum_{j=1}^q \gamma_j X_{i,p+j}+\epsilon_i, \;\; i=1,\ldots, n,
\end{equation}
and the observed data $(Y_i,T_i,\bm X_{i,obs})$, $i=1,\ldots,n$, where all covariates  $X_{i,1},\ldots,X_{i,p+q}$ are assumed to be independent with zero means. The working model for analysis is
\begin{equation}\label{eq:working} Y_i=T_i\mu_1+(1-T_i)\mu_0+\sum_{j=1}^p\beta_jX_{i,j}+e_i, \;\; i=1,\ldots, n.
\end{equation}
It is shown that the test for \eqref{htest} basing the least square method is valid under a large class of CAR procedures but is usually conservative. \cite{LiuHu2023} considers  the linear heteroscedasticity model
\begin{equation}\label{eq:linearmodel2} Y_i=T_i\mu_1+(1-T_i)\mu_0+\sum_{j=1}^p\beta_jX_{i,j}+\sum_{j=1}^q \gamma_j X_{i,p+j}+\gamma_1z_{i,1}T_i+\gamma_2(1-T_i)z_{i,2}+\epsilon_i, \;\; i=1,\ldots, n,
\end{equation}
with unobserved $z_{i,1}$, $z_{i,2}$, and obtains  the theory on the test  under a CAR procedure after discretization. In this paper, we still consider the working model \eqref{eq:working}. But   the real model is   a general heteroscedasticity model as follows in which  $X_{i,1},\ldots, X_{i,p+q}$ are not necessarily to be independent and have not necessarily zero means:
\begin{align}\label{eq:nonlinearmodel}
 Y_i= & T_iY_i(1)+(1-T_i)Y_i(0)\nonumber\\
 =& T_i\{\mu_1+m_1(\bm X_i)+\epsilon_{i,1}^{(n)}\}+(1-T_i)\{\mu_0+m_0(\bm X_i)+\epsilon_{i,0}^{(n)}\}, \;\; i=1,\ldots, n,
\end{align}
 where
  \begin{enumerate}
  \item[(L1)] $(\bm X_i,\ \epsilon_{i,0}^{(n)},\epsilon_{i,1}^{(n)})$, $i=1,\ldots, n$, are  independent and identically distributed  with $\Ex[\epsilon_{i,a}^{(n)}]=0$, $\Ex[(\epsilon_{i,a}^{(n)}-\epsilon_{i,a})^2]\to 0$ as $n\to \infty$, $a=0,1$, and $(\bm X_i,\ \epsilon_{i,0},\epsilon_{i,1})$, $i=1,2,\ldots$  are  independent and identically distributed as $(\bm X, \epsilon_0,\epsilon_1)$, $\Ex[|m_a(\bm X)|^{2}]<\infty$, $\Ex[\epsilon_a^2]<\infty$, $a=0,1$,
  \item[(L2)]   $\bm X_{i,obs}$  has a non-singular   variance-covariance matrix $\Var(\bm X_{obs})$ if $p\ne 0$,
  \item[(L3)]   $\Ex[\epsilon_{i,a}|\bm X_i]=0$, $a=0,1$ and $\Ex[m_1(\bm X)-m_0(\bm X)]= 0$.
\end{enumerate}
Here, the assumption  (L3)  is used to make the parameter $\mu_1-\mu_0=\Ex[Y(1)]-\Ex[Y(0)]$ to be identifiable. We assume that the model parameters $\mu_a$ and errors $\epsilon_{i, a}^{(n)}$, $a=0,1$, can be dependent on $n$ so that the model includes the case under the local alternative hypothesis.

\subsection{Traditional Test}
 Suppose that a CAR procedure with a feature $\bm \phi(\bm X)$ is applied to balance the covariates before the analysis. However, in the analysis the values of $\bm \phi(\bm X_1),\ldots,\bm\phi(\bm X_n)$ may not be observed.
Based on the working model \eqref{eq:working} with the observed data $(Y_i, T_i, \bm X_{i,obs})$, $i=1,\ldots, n$, the least square  estimator of $\theta=(\mu_1,\mu_0,\beta_1,\ldots,\beta_p)$ is
\begin{align}\label{eqLSEforBalance}
\widehat{\theta}_n=&\sum_{i=1}^{n} \underline{\bm X}_{i}Y_i\left\{ \sum_{i=1}^{n}  \underline{\bm X}_{i}^{\otimes 2}\right\} ^{-1}
= \theta +\sum_{i=1}^{n} \underline{\bm X}_{i} e_i\left\{ \sum_{i=1}^{n}  \underline{\bm X}_{i}^{\otimes 2}\right\} ^{-1},
\end{align}
where $\underline{\bm X}_i=\left(T_i,1-T_i,X_{i,1},X_{i,2}\dots X_{i,p}\right)$, $\bm x^{\otimes 2}=\bm x^{\prime}\bm x$. Correspondingly, $\widehat{\tau}_n=(1, -1, 0, \ldots, 0) \widehat{\theta}_n^{\prime}$ is the estimator of $\tau=\mu_1-\mu_0$, $\widehat{\bm \beta}_n=diag(0,0,1, \ldots, 1) \widehat{\theta}_n^{\prime}$ is the estimator of $\bm \beta=(\beta_1,\ldots,\beta_p)$.
 To compare treatment effect of $\mu_1$ and $\mu_0$, we consider the   hypothesis
testing \eqref{htest}.
The classical test statistic for  \eqref{htest}  is
\begin{equation}\label{TS}
\mathcal{T}_{LS}(n)=\frac{ \widehat{\theta}_n\bm L^{\prime}}{\widehat{\sigma}_e\big(\bm L \left\{\sum_{i=1}^{n}  \underline{\bm X}_{i}^{\otimes 2}\right\} ^{-1}\bm L^{\prime}\}^{1/2}},
\end{equation}
and reject $H_0$ when $|\mathcal{T}_{LS}(n)|\ge u_{\alpha}$, where $\boldsymbol{L}=(1,-1,0,..,0)$, $\widehat{\sigma}_e^2=\sum_{i=1}^n (Y_i-\underline{\bm X}_i\widehat{\theta}_n^{\prime})^2/(n-p-2)$, $u_{\alpha}$ is the upper $\alpha$th quantile of a  standard normal distribution. For the special case $p=0$,
$$ \mathcal{T}_{LS}(n)=\frac{\overline{Y}_{n,1}-\overline{Y}_{n,0}}{\widehat{\sigma}_e\sqrt{1/N_{n,1}+1/N_{n,0}}}, $$
$$ \widehat{\sigma}_e^2=\frac{1}{n-2}\sum_{i=1}^n \left\{T_i(Y_i-\overline{Y}_{n,1})^2+(1-T_i)(Y_i-\overline{Y}_{n,0})^2\right\}, $$
where $N_{n,1}=\sum_{i=1}^n T_i$, $N_{n,0}=n-N_{n,1}$, $\overline{Y}_{n,1}=\sum_{i=1}^n T_iY_i/N_{n,1}$ and $\overline{Y}_{n,0}=\sum_{i=1}^n (1-T_i)Y_i/N_{n,0}$.

\begin{theorem} \label{thm:test} Suppose that a CAR procedure with a feature $\bm\phi(\bm X)$ is applied to balance covariates,   and   Assumptions \ref{asm:iid}, \ref{asm:moment} (with $\gamma\ge 3)$ and \ref{asm:recurrent} are satisfied.
Write $\Sigma_{obs}=\Var(\bm X_{obs})$, $m_{ave}(\bm X)=\frac{m_1(\bm X)+m_0(\bm X)}{2}$,
$ \bm \beta_{\bot}=\Cov\{m_{ave}(\bm X), \bm X_{obs}\}\Sigma_{obs}^{-1}$ and
$\theta_{\ast}=(\mu_1+\Ex[m_1(\bm X)]-\Ex[\bm X_{obs}]\bm \beta_{\bot}^{\prime}, \mu_0+\Ex[m_0(\bm X)]-\Ex[\bm X_{obs}]\bm \beta_{\bot}^{\prime},\bm\beta_{\bot})$.
Then \begin{equation}\label{eq:estbeta} \widehat{\theta}_n-\theta_{\ast}\overset{P}\to 0,  \end{equation}
\begin{equation}\label{eq:esttau}
 \sqrt{n}(\widehat{\tau}_n-\tau)\overset{D}\to N(0,4\sigma_{\tau}^2),
  \end{equation}
where $\sigma_{\tau}^2=\sigma_{\epsilon}^2+\vec{\sigma}_m^2$,
\begin{equation}\label{eq:testvar}\sigma_{\epsilon}^2=  \frac{1}{4}\Var\big( m_1(\bm X)-m_0(\bm X)\big)+\frac{1}{2}\Var(\epsilon_1)+\frac{1}{2}\Var(\epsilon_0),
\end{equation}
and $\vec{\sigma}_m^2$ is defined as   in \eqref{eq:clt} with $m(\bm X)=\frac{m_1(\bm X)+m_0(\bm X)-\Ex[m_1(\bm X)+m_0(\bm X)]}{2}-(\bm X_{obs}-\Ex[\bm X_{obs}])\bm \beta_{\bot}^{\prime}$.
\begin{itemize}
\item[(i)] Under $H_0:\tau=0$,
\begin{equation}\label{eq:test1}
\mathcal{T}_{LS}(n)\overset{D}{\to}N(0,\sigma_{\mathcal T}^2),\; \text{with}, \;\sigma_{\mathcal T}^2=\frac{\sigma_{\tau}^2}{\sigma_e^2},
\end{equation}
 where
 $\sigma_e^2= \sigma_{\epsilon}^2+\Var(m(\bm X))$;
\item[(ii)] under $H_A:\tau\neq 0$, consider a sequence of local alternatives, i.e., $\tau=\delta/\sqrt{n}+o(1/\sqrt{n})$ for a fixed $\delta \neq 0$, then
\begin{equation}\label{eq:test2}
\mathcal{T}_{LS}(n)\overset{D}{\to}N(\Delta,\sigma_{\mathcal T}^2),\; \text{with}\; \Delta=\frac{\delta}{2\sigma_e}.
\end{equation}
\end{itemize}
\end{theorem}

The key step of the proof is showing that
\begin{align}\label{eq:linearRepofTau}
 \sqrt{n}(\widehat{\tau}_n-\tau)= &2\frac{\sum_{i=1}^n \widetilde{r}_i}{\sqrt{n}}   +o_p(1) \text{ with } \widetilde{r}_i= (2T_i-1)(Y_i-\underline{\bm X}_i\theta_{\ast}), \\
 = &\frac{\sum_{i=1}^n   \left\{(2T_i-1)( \frac{\epsilon_{i,1}+ \epsilon_{i,0}}{2}+m(\bm X_i))+\frac{ \epsilon_{i,1}- \epsilon_{i,0}+m_1(\bm X_i)-m_0(\bm X_i)}{2}\right\}}{\sqrt{n}}+o_P(1)\nonumber\\
 \widehat{\sigma}_e^2=&\frac{1}{n}\sum_{i=1}^n (Y_i-\underline{\bm X}_i\theta_{\ast}^{\prime})^2+o_P(1),\nonumber
 \end{align}
and  applying Theorem \ref{thm:generalclt} and Corollary \ref{cor:LLN} to obtain \eqref{eq:esttau} and $\widehat{\sigma}_e^2\overset{P}\to\sigma_e^2$, respectively. The details will be given in the supplementary materials.

Next, we discuss the validity of the test.  Suppose
\begin{equation}\label{eq:inSpan}
\Ex[m(\bm X)|\bm\phi(\bm X)] \in Span\{\bm \phi(\bm X)\}=\{\langle\bm\beta,\bm \phi(\bm X)\rangle|\bm \beta\in \mathbb{R}^q\}.
\end{equation}
Then $\vec{\sigma}_{\Ex[m(\bm X)|\bm\phi(\bm X)]}^2=0$, $\vec{\sigma}_m^2=\Var(m(\bm X)-\Ex[m(\bm X)|\bm\phi(\bm X)])\le \Var(m(\bm X))$ by \eqref{eq:variance-m-conditionM}, and so $\sigma_{\mathcal T}^2\le 1$. From \eqref{eq:test1}, it follows that if we  reject $H_0$ when $ |\mathcal{T}_{LS}(n)| \ge u_{\alpha/2}$, then
$$ \lim_{n\to \infty}\Prob( |\mathcal{T}_{LS}(n)| \ge u_{\alpha/2 }| H_0)\le \alpha. $$
   The test controls the type I error rate but is conservative when $\sigma_{\mathcal{T}}<1$ ($\Ex[m(\bm X)|\bm\phi(\bm X)]\not\equiv Const.$). Under the local alternatives $H_A: \tau=\delta/\sqrt{n}+o(1/\sqrt{n})$, $\delta\ne 0$, the asymptotic power is
\begin{align*}
&\lim_{n\to \infty}\Prob( |\mathcal{T}_{LS}(n)| \ge u_{\alpha }| H_A)=   \Prob\left(\left|\frac{|\delta|}{2\sigma_{\tau}}+N(0,1)\right|\ge u_{\alpha/2}\frac{\sigma_e}{\sigma_{\tau}}\right)\\
= &1-  \Prob\left(\frac{1}{\sigma_{\tau}}\big( \frac{|\delta|}{2}-u_{\alpha/2} \sigma_e\big)< N(0,1)< \frac{1}{\sigma_{\tau}}\big( \frac{|\delta|}{2}+u_{\alpha/2} \sigma_e\big)\right)\\
\le & \Prob\left(\left|\frac{|\delta|}{2\sigma_{\tau}}+N(0,1)\right|\ge u_{\alpha/2}\right)\le   \Prob\left(\left|\frac{|\delta|}{2\sigma_{\epsilon}}+N(0,1)\right|\ge u_{\alpha/2}\right).
\end{align*}
From the power function,   we have the
following conclusions: (i) The power is a decreasing function of $\sigma_e^2$. (ii) When $|\delta|/2\le  u_{\alpha/2}\sigma_e$, the power is an increasing function of $\sigma_{\tau}^2$.   When $|\delta|/2>(1+\epsilon_0)u_{\alpha/2}\sigma_e$ which means that the signal is significantly stronger than the noise, the power is a decreasing function of $\sigma_{\tau}^2\in[\sigma_{\epsilon}^2,\sigma_e^2]$, where $\epsilon_0$ is the solution of $\ln(1+2/\epsilon_0)=2(1+\epsilon_0)u_{\alpha/2}^2 $, and $\epsilon_0=9.153\times 10^{-4}$ when $\alpha=0.05$, $\epsilon_0=8.566\times 10^{-3}$ when $\alpha=0.10$. The variance $\vec{\sigma}_m^2$ of $\sum_{i=1}^n (2T_i-1)m(\bm X_i)/\sqrt{n}$ gives the loss of power. In this case,   (iii) When and only when $m(\bm X)\equiv 0$, i.e., $m_1(\bm X)+m_0(\bm X)$ is a linear function of $\bm X_{obs}$, $\sigma_e^2$, $\sigma_{\tau}^2$ and $\sigma_{\epsilon}^2$ are   equal and the power attains its largest values.

\eqref{eq:inSpan} is an important sufficient condition for the test to be valid. Basing  on the linear heteroscedasticity model
\eqref{eq:linearmodel2}, \cite{LiuHu2023} studies the  properties of the tests under  a CAR procedure after discretization. It is also found that when the variances of the within-stratum imbalances
under CAR procedures are larger than the values under the complete randomization, the type I error rate will be inflate.  However, there is non reason that a CAR procedure always generates smaller variance than the complete randomization does. A simulation study of \cite{Zhao2022} shows that there is
an increased variance of within-stratum imbalances under \cite{Pocock1975}'s procedure compared
to the complete randomization.  In general,   when \eqref{eq:inSpan} is not satisfied, there is no theory to ensure $\vec{\sigma}_m^2\le \Var(m(\bm X))$. It may happen that $\vec{\sigma}_m^2>\Var(m(\bm X))$. In such a case, the type I error rate of the test is inflated.
The following two corollaries give CAR procedures with \eqref{eq:inSpan}   so that $\sigma^2_{\mathcal{T}}\le 1$.

\begin{corollary}\label{cor:test} Assume that the third moment  of $\bm X$ is finite. We consider a CAR   procedure with a   feature $\bm\phi$. Suppose that one of the following conditions is satisfied:
\begin{description}
  \item[\rm (i)] The $n_0$-th convolution   of the distribution of the covariate $\bm X$ has a density  for some $n_0$. Let $\bm \phi=(1,X_1,\ldots,X_{p+q})$;
  \item[\rm (ii)] Each $X_t$ takes   values in lattices. Let $\bm \phi=(1,X_1,\ldots,X_{p+q})$;
  \item[\rm (iii)]  Each $X_t$ takes finite many values $x_t^{(l_t)}$, $l_t=1\ldots, L_t$. Let  $\bm \phi=\bm \phi_{PS}(\bm X)=\big(\sqrt{w_t}I\{X_t=x_t^{(l_t)}\}; l_t=1,\ldots, L_t,t=1,\ldots, p+q\big)$, where $w_t>0$; i.e., Pocock and Simon (1975)'s procedure is used for balancing covariates;
  \item[\rm (iv)] $\bm X$ can be written as $(  \bm X_{dis},\bm X_{con})$, where the discrete part $\bm X_{dis}$   takes finite many values and, for some $n_0$ and $\bm x_{dis}$, the $n_0$-th convolution   of the conditional distribution of   $\bm X_{con}$ for given $\bm X_{dis}=\bm x_{dis}$ has a density. Let $\bm\phi=\big(\bm\phi_{PS}(\bm X_{dis}),\bm X_{con}\big)$.
\end{description}
Then $1,X_1,\ldots,X_{p+q}\in Span\{\bm\phi(\bm X)\}$. Further, \eqref{eq:esttau}, \eqref{eq:test1}-\eqref{eq:inSpan} hold with $\vec{\sigma}_m^2=0$ if $m_a(\bm X)$, $a=0,1$, are linear functions of $X_1, \ldots, X_{p+q}$.
\end{corollary}

\begin{corollary}\label{cor:testdis} We consider  the CAR procedure   with respect to discretized variables $d_k(X_k)$, where $d_k$s are discrete functions with finite many values. Write $\bm X_{dis}=(d_1(X_1),\ldots,d_{p+q}(X_{p+q}))$. Suppose that one of the following conditions is satisfied:
\begin{description}
  \item[\rm (v)] Let $\bm \phi=\bm \phi_{S}(\bm X_{dis})$, i.e.,  the stratified randomization procedure with respect to the discretized covariates is applied;
  \item[\rm (vi)] Let $\bm \phi=\bm \phi_{HH}(\bm X_{dis})$ with $w_s\ne 0$, i.e.,    \cite{Hu2012}'s procedure with respect to the discretized covariates is applied;
  \item[\rm (vii)] Suppose that the elements $X_t$, $t=1,\ldots, p+q$ of $\bm X$ are independent, and $m_a(\bm X)$, $a=0,1$, are linear functions of $f_1(X_1), \ldots, f_{p+q}(X_{p+q})$.    Let  $\bm \phi=\bm \phi_{PS}(\bm X_{dis})$,  i.e., \cite{Pocock1975}'s procedure with respect to the discretized covariates is applied.
\end{description}
Then $\Ex[m(\bm X)|\bm X_{dis}]\in Span\{\bm\phi(\bm X)\}$. Hence $\Ex[m(\bm X)|\bm\phi(\bm X)]=\Ex\big[\Ex[m(\bm X)|\bm X_{dis}\big|\bm\phi(\bm X_{dis})\big]\in   Span\{\bm\phi(\bm X)\}$ and \eqref{eq:esttau}, \eqref{eq:test1} and \eqref{eq:test2} hold with $\vec{\sigma}_m^2=\Var(m(\bm X)-\Ex[m(\bm X)|\bm X_{dis}])\le \Var(m(\bm X))$.
\end{corollary}

\subsection{ Adjusted  Test}
Theorem \ref{thm:test} indicates that the test for the treatment effect based on the least square estimator usually has not a precise type I error rate. It is conservative if $\vec{\sigma}_m^2<\Var(m(\bm X))$, and has inflated type I error rate so that is not valid if $\vec{\sigma}_m^2>\Var(m(\bm X))$.  To obtain a valid test, one approach is to construct a test under a correctly specified model \eqref{eq:linearmodel}
that includes all covariates that are used in covariate-adaptive randomization (\cite{Shao2010}).  Another approach is to introduce the adjustment for the test to restore the correct Type I error rate.
For this purpose, we need to replace $\widehat{\sigma}_e^2$  by a consistent estimator $\widehat{\sigma}^2_{\tau}$ of $\sigma^2_{\tau}$, and obtain an adjusted test statistic:

\begin{equation}\label{TSadj}
\mathcal{T}_{adj}(n)=\frac{ \widehat{\theta}_n\bm L^{\prime}}{\widehat{\sigma}_{\tau}\big(\bm L \left\{\sum_{i=1}^{n}  \underline{\bm X}_{i}^{\otimes 2}\right\} ^{-1}\bm L^{\prime}\}^{1/2}} \left(\text{ or } \frac{ \sqrt{n}\widehat{\theta}_n\bm L^{\prime}}{2\widehat{\sigma}_{\tau}}\right).
\end{equation}

The following corollary follows from Theorem \ref{thm:test} immediately.

\begin{corollary}\label{cor:testadj} Suppose that the conditions in Theorem \ref{thm:test}    are satisfied. Assume  $\widehat{\sigma}_{\tau}^2\overset{P}\to \sigma_{\tau}^2$. Then
\begin{itemize}
\item[(i)] under $H_0:\tau=0$,
\begin{equation}\label{eq:testadj1}
\mathcal{T}_{adj}(n)\overset{D}{\to}N(0,1);
\end{equation}
\item[(ii)] under $H_A:\tau\neq 0$, consider a sequence of local alternatives, i.e., $\tau=\delta/\sqrt{n}+o(1/\sqrt{n})$ for a fixed $\delta \neq 0$, then
\begin{equation}\label{eq:testadj2}
\mathcal{T}_{adj}(n)\overset{D}{\to}N(\Delta,1),\; \text{with}\; \Delta=\frac{\delta}{2\sigma_{\tau}}.
\end{equation}
\end{itemize}
\end{corollary}

\begin{remark} The test \eqref{TSadj} restores  the correct Type I error rate $ \lim_{n\to \infty}\Prob( |\mathcal{T}_{adj}(n)| \ge u_{\alpha }| H_0)=\alpha$.
The asymptotic power
$$ \lim_{n\to \infty}\Prob( |\mathcal{T}_{adj}(n)| \ge u_{\alpha }| H_A)=   \Prob\left(\left|\frac{|\delta|}{2\sigma_{\tau}}+N(0,1)\right|\ge u_{\alpha/2}\right) $$
is a decreasing function of the variance $\sigma_{\tau}^2=\sigma_{\epsilon}^2+\vec{\sigma}_m^2$. When $m(\bm X)\in Span\{\bm\phi(\bm X)\}$, $\vec{\sigma}_m^2=0$ and $\sigma_{\tau}^2$ takes the smallest value $\sigma_{\epsilon}^2$ so that the test is  most powerful.
The variance $\vec{\sigma}_m^2$ of $\sum_{i=1}^n (2T_i-1)m(\bm X_i)/\sqrt{n}$ gives the loss of power.

By Corollary \ref{cor:testdis}, when the covariate $\bm X$ takes finite many values, the adjusted test is most powerful under the
stratified randomization procedure or \cite{Hu2012}'s procedure (with $w_s\ne 0$) for all kinds of real models, and,  when $X_1,\ldots, X_{p+q}$ are independent and take finitely many values,
the adjusted test is most powerful under \cite{Pocock1975}'s procedure for all kinds of real additive models of the type $m_a(\bm X)=\sum_{t=1}^{p+q}f_{t, a}(X_t)$, $a=0,1$.

In general, when $X_t$ takes an infinite number of values, we can not find a finite-dimensional covariate feature $\bm \phi(\bm X)$ such that $f(X_t)\in Span\{\bm \phi(\bm X)\}$ for all  $f$, and so there does not exist
a model free covariate-adaptive randomization procedure such that the test is always most powerful.
\end{remark}

\begin{example} {(\rm Logistic regression)} \cite{Ma2020} considers the logistic regression model
\begin{equation}\label{eq:example5.1}\Ex(Y_i=1|T_i, \bm X_i)=\Prob(Y_i=1|T_i, \bm X_i)=h\left(\mu_1T_i+(1-T_i)\mu_0+\bm X_i\bm \beta^{\prime}\right)
\end{equation}
with working model
\begin{equation}\label{eq:example5.2} Y_i=h\left(\mu_1T_i+(1-T_i)\mu_0\right)+e_i,
\end{equation}
where $h(t)=e^t/(1+e^t)$, and obtains the asymptotic properties of the test for \eqref{htest} basing on the maximum likelihood estimators.

 Now,  \eqref{eq:example5.1} can be written as the heteroscedasticity model \eqref{eq:nonlinearmodel} with $m_a(\bm x)=h(\mu_a+\bm x\bm\beta^{\prime})-h(\mu_a)$ and $\epsilon_{i,a}=Y_i-h(\mu_a+\bm X_i\bm\beta^{\prime})$, $a=0,1$. If we use the test statistic \eqref{TS}, then under $H_0$ we have
$$ \mathcal{T}_{LS}(n)=\frac{(\overline{Y}_{n,1}- \overline{Y}_{n,2})}{\widehat{\sigma}_{e}\sqrt{1/N_{n,1}+1/N_{n,0}}}\overset{D}\to N\left(0,\frac{\sigma_{\tau}^2}{\Ex[h](1-\Ex[h])}\right),$$
where $\widehat{\sigma}_e^2=\frac{1}{n-2}\sum_{j=1}^{n}\{T_i(Y_i- \overline{Y}_{n,1})^2+(1-T_i)(Y_i- \overline{Y}_{n,2})^2\}$, $\sigma_{\tau}^2=\Ex[h(1-h)]+\vec{\sigma}_m^2$, $h=h(\mu_0+\bm X\bm \beta^{\prime})$ and $\vec{\sigma}_m^2$ is defined  as in \eqref{eq:clt} with
$m(\bm X)=h-\Ex[h]$. $\mathcal{T}_{LS}(n)$   is asymptotically equivalent to $S_L$ of \cite{Ma2020}. When the stratified randomization procedure or \cite{Hu2012}'s procedure ($w_s\ne 0$) with respect to the discretized covariates is applied to randomize units, we always have $\vec{\sigma}_m^2\le\Var(h)$, and so $\sigma_{\tau}^2\le\Ex[h](1-\Ex[h])$ and the test is valid but is usually conservative and always losses power comparing to the likelihood ratio test under the underlying   model \eqref{eq:example5.1}. Under \cite{Pocock1975}'s procedure or other marginal covariate-adaptive randomization procedure, whether  $\sigma_{\tau}^2\le \Ex[h](1-\Ex[h])$ holds or not is unknown.

Under a  covariate-adaptive randomization procedure satisfying the conditions in Theorem \ref{thm:test}, if we use the test statistic with a consistent estimator  $\widehat{\sigma}_{\tau}^2$ of $\sigma_{\tau}^2$ to replace $\widehat{\sigma}_e^2$, then under $H_0$  \eqref{eq:testadj1} holds and so the test is always valid and has a precise type I error rate.

Further, under the local alternatives  $\mu_1=\mu_0+\delta/\sqrt{n}+o(1/\sqrt{n})$, we can write $Y_i(a)=I\{U_i\le h(\mu_a+\bm X_i\bm\beta^{\prime})\}$, where $U_i$, $i=1,2,\ldots$, are independent and identically distributed random variables, and are independent of all other variables. Then
\begin{align*}
Y_i(0)=& \Ex[Y(0)]+m(\bm X_i)+\epsilon_{i,0}, \\
Y_i(1)=&\Ex[Y(1)]+m(\bm X_i)+\epsilon_{i,1}^{(n)},
\end{align*}
where $m(\bm X)=h-\Ex[h]$, $\epsilon_{i,0}=Y_i(0)-h(\mu_0+\bm X_i\bm\beta^{\prime})$, $\epsilon_{i,1}^{(n)}=\epsilon_{i,0}+(Y_i(1)-Y_i(0))-\Ex[Y(1)-Y(0)]$.
We have
\begin{align*}
&\Ex[Y(1)]-\Ex[Y(0)]= \Ex\left[h\left(\mu_1+\bm X\bm \beta^{\prime}\right)-h\left(\mu_0+\bm X\bm \beta^{\prime}\right)\right]\\
=&\frac{\delta}{\sqrt{n}}\Ex\left[h^{\prime}\left(\mu_0+\bm X\bm \beta^{\prime}\right)\right]+O\left(\frac{1}{n}\right)
=\frac{\delta}{\sqrt{n}}\Ex\left[h(1-h)\right]+O\left(\frac{1}{n}\right)
\end{align*}
and
$$ \Ex[(\epsilon_{i,1}^{(n)}-\epsilon_{i,0})^2]=|\Ex[Y(1)-Y(0)]|(1-|\Ex[Y(1)-Y(0)]|)=O\left(\frac{1}{\sqrt{n}}\right). $$
Hence the conditions for the heteroscedasticity model \eqref{eq:nonlinearmodel} are satisfied. Then
\eqref{eq:testadj2} holds with
$$ \Delta=\frac{\delta\Ex\left[h(1-h)\right]}{2\sqrt{\Ex[h(1-h)]+\vec{\sigma}_m^2}}. $$
When $\vec{\sigma}_m^2=0$, the power attends its largest value which is equivalent to the one  of the likelihood ratio test under the underlying   model \eqref{eq:example5.1}.
\end{example}

\subsection{The estimator of $\sigma_{\tau}^2$}
For the estimator of  $\sigma_{\tau}^2$,  under the linear model \eqref{eq:linearmodel} with working model \eqref{eq:working} and $p=0$, \cite{Shao2010} suggests a  bootstrap estimator and shows its consistency under the null hypothesis, and \cite{Ma2022} suggests a    regression estimator of $\sigma_{\epsilon}^2$ ( $\sigma_\tau^2=\sigma_{\epsilon}^2$ under their assumptions) by fitting the linear model  \eqref{eq:linearmodel}. In this subsection, we give five kinds of estimators.

{\em Regression Estimator.} For the regression estimator under the general model, besides the data $\{Y_i, T_i,\bm X_{i, obs};i=1,\ldots,n\}$ in the working model \eqref{eq:working}, we need to assume that the balance features $\bm\phi(\bm X_i), i=1,\ldots, n$, are also observable in the analysis stage. After obtaining the   LSE $\widehat{\theta}=\widehat{\theta}_n$ of $\theta=(\mu_1,\mu_0,\beta_1\ldots,\beta_p)$, we regress the residuals $Y_i-\underline{\bm X}_i\widehat{\theta}^{\prime}$s over $\bm\phi(\bm X_i)$s by fitting the linear model
$$ Y_i-\underline{\bm X}_i\widehat{\theta}^{\prime} =\bm\phi(\bm X_i)\alpha^{\prime}+\zeta_i, \; i=1,\ldots, n, $$
and obtain the new residuals $\widehat{\zeta}_i=Y_i-\underline{\bm X}_i\widehat{\theta}^{\prime}-\bm\phi(\bm X_i)\widehat{\alpha}^{\prime}$, $i=1,\ldots, n$.  Define the estimator of $\sigma_{\tau}^2$ by
\begin{equation}\label{eq:RegEst} \widehat{\sigma}_{\tau}^2=\widehat{\sigma}_{\tau,reg}^2=\frac{1}{n-(p+2)} \sum_{i=1}^n (\widehat{\zeta}_i)^2.
\end{equation}
\begin{theorem}\label{thm:regest} Suppose that the conditions in Theorem \ref{thm:test}    are satisfied. If \eqref{eq:inSpan} holds, then
$$\widehat{\sigma}_{\tau,reg}^2\overset{P}\to \sigma_{\tau}^2. $$
\end{theorem}

If \eqref{eq:inSpan} is not satisfied, whether $\widehat{\sigma}_{\tau, reg}^2$ is consistent or not is unknown.

{\em Bootstrap Estimator.} For the bootstrap estimator as  \cite{Shao2010} suggested, we generate independent and identically random variables $I(1),\ldots,I(n)$ as a simple random sample with replacement from $\{1,\ldots, n\}$ and define the
bootstrap data $Y_i^{\ast}=Y_{I(i)}, \bm X_{i,obs}^{\ast}=\bm X_{I(i),obs},\bm\phi_i^{\ast}=\bm\phi\big(\bm X_{I(i)}\big)$, $i=1,\ldots,n$. Applying a covariate adaptive randomization with feature $\bm\phi^{\ast}$, we obtain the bootstrap analogues of treatment assignments, $T_1^{\ast},\ldots, T_n^{\ast}$. Based on the bootstrap data, we obtain the LSE $\widehat{\tau}^{\ast}$ of $\tau=\mu_1-\mu_2$ under the working model.
The  bootstrap estimator of the asymptotic variance of $\widehat{\tau}_n$  is then $v_B = \Var_{\ast}(\widehat{\tau}^{\ast})$,
where $\Var_{\ast}$ is the  variance of $\widehat{\tau}^{\ast}$
 with respect to bootstrap sampling. However, under the general model, the consistency of the bootstrap estimator is still an open problem.

\bigskip
For both the bootstrap estimator and the regression estimator, in the analysis stage,  besides the observed data in the working model, we need other additional information such as the balance features $\bm\phi(\bm X_i)$s  and the covariate-adaptive randomization procedure used in the randomization stage.
However,  models are usually subject to misspecification and so only part of covariates are included in the analysis so that the working model is \eqref{eq:working}. It is important to introduce a consistent estimator of $\sigma^2_{\tau}$ only based on the working model \eqref{eq:working}. We consider three consistent estimators.

{\em Moving Block Estimator.}  The idea is as follows. Note
\eqref{eq:esttau}. We have
$$ \Ex[n(\widehat{\tau}_n-\tau)^2]\asymp 4 \sigma_{\tau}^2. $$
Let $\widehat{\tau}_{i,l}$ be the least square  estimator of $\tau$ basing on the subgroup of the data $\{Y_j, T_j, \bm X_{j,obs}; j=i+1,\ldots, i+l\}$.
Then
$$ \Ex\Big[\big(\frac{\sum_{j=i+1}^{i+l} \widetilde{r}_i}{\sqrt{l}}\big)^2 \Big]\asymp\frac{1}{4}\Ex[l(\widehat{\tau}_{i,l}-\tau)^2]\asymp   \sigma_{\tau}^2, $$
by noting \eqref{eq:linearRepofTau}, where $\widetilde{r}_i= (2T_i-1)(Y_i-\underline{\bm X}_i\theta_{\ast})$.
Replacing the unknown parameter $\theta_{\ast}$ by its consistent estimator $\widehat{\theta}_n$ and replacing the expectation $\Ex$ by the average over $i$, we obtain
the   estimator of $\sigma^2_{\tau}$   defined as the moving block sample variance:
\begin{equation}\label{eq:mbEst} \widehat{\sigma}^2_{\tau}=\widehat{\sigma}^2_{\tau,mb}=\frac{1}{n-l+1-(p+2)}\sum_{i=0}^{n-l} \left(\frac{\sum_{j=i+1}^{i+l}r_i}{\sqrt{l}}\right)^2,
\end{equation}
where $1\le l<n$ and
$$r_i=T_i(Y_i-\underline{\bm X}_i\widehat{\theta}_n^{\prime})-(1-T_i)(Y_i-\underline{\bm X}_i\widehat{\theta}_n^{\prime}). $$
 When $l=1$, $\widehat{\sigma}^2_{\tau}=\widehat{\sigma}^2_{e}$. To make $\widehat{\sigma}^2_{\tau}$ consistent,    we assume $l\to \infty$ and $l/n\to \infty$. In practice, we can choose $l=integer(\sqrt{n})$. For the special case $p=0$,
\begin{equation}\label{eq:adjestimator3} \widehat{\sigma}^2_{\tau,mb}=\frac{1}{(n-l-1)l}\sum_{i=0}^{n-l} \left( \sum_{j=i+1}^{i+l}\big(T_i(Y_i-\overline{Y}_{n,1})-(1-T_i)(Y_i-\overline{Y}_{n,0}\big) \right)^2.
\end{equation}

\begin{theorem}\label{thm:mbest} Suppose that the conditions in Theorem \ref{thm:test}    are satisfied. Then
\begin{equation}\label{eq:mbestcon}\widehat{\sigma}_{\tau,mb}^2=\frac{1}{n-l+1}\sum_{i=0}^{n-l} \left(\frac{\sum_{j=i+1}^{i+l}\widetilde{r}_i}{\sqrt{l}}\right)^2+o_P(1)\overset{P}\to \sigma_{\tau}^2.
\end{equation}
\end{theorem}

{\em Moving Block Jackknife.} The moving block jackknife  (MBJ)  was  proposed by \cite{Kusnch1989} and \cite{LiuSingh1992} to estimate the asymptotic variance for dependent data. For each $i=0,\ldots,n-l$, we remove a block of observations $\{Y_j,T_j, \bm X_{j,obs}, j=i+1,\ldots,i+l\}$ from the data $\{Y_j,T_j,\bm X_{j,obs},j=1,\ldots,n\}$, and fit   the working model \eqref{eq:working} to obtain the LSE $\widehat{\tau}_{i,jack}$ of $\tau=\mu_1-\mu_0$. The moving block Jackknife estimator of the asymptotic variance   of $\widehat{\tau}_n-\tau$ is given by the sample variance of $\{\widehat{\tau}_{i,jack},i=0,\ldots,n-l\}$ multiplied by $(n-l)/l$:
\begin{equation}\label{eq:mbjest} \widehat{\sigma}^2_{mbj}=\frac{(n-l)}{l}\frac{1}{n-l}\sum_{i=0}^{n-l}(\widehat{\tau}_{i,jack}-\overline{\widehat{\tau}_{\cdot,jack}})^2\; \text{where}\;
\overline{\widehat{\tau}_{\cdot,jack}}=\frac{\sum_{i=0}^{n-l}\widehat{\tau}_{i,jack}}{n-l+1}.
\end{equation}

When the fourth moments of $Y_i$ and $\bm X_i$ are finite and $\Ex[\|(\epsilon_{i,a}^{(n)}-\epsilon_{i,a})(\bm X_{i,obs}-\Ex[\bm X_{obs})\|]=O(1/\sqrt{n})$, $a=0,1$,  the  convergence rate $o_P(1)$  in \eqref{eq:linearRepofTau} can be improved to $O_P(n^{-1/2})$. Similarly,
$$ \sqrt{n-l}(\widehat{\tau}_{i,jack}-\tau)=2\frac{S_n-x_{i,l}}{\sqrt{n-l}}+O_P((n-l)^{-1/2}), $$
and then
$$ \sqrt{n-l}(\widehat{\tau}_{i,jack}-\overline{\widehat{\tau}_{\cdot,jack}})=2\frac{\overline{x_{\cdot,l}}-x_{i,l}}{\sqrt{n-l}}+O_P((n-l)^{-1/2}), $$
where $S_n=\sum_{i=1}^n \widetilde{r}_i$, $x_{i,l}=\sum_{j=1}^l \widetilde{r}_{i+j}$ and  $\overline{x_{\cdot,l}}=\sum_{i=0}^{n-l}x_{i,l}/(n-l+1)$.
 Hence
\begin{align}\label{eq:mbjcon}
&n\widehat{\sigma}^2_{mbj}= 4\frac{n}{(n-l)^2l}\sum_{i=0}^{n-l}(\overline{x_{\cdot,l}}-x_{i,l})^2+o_P(1)\nonumber\\
=&4\frac{n(n-l+1)}{(n-l)^2}\left(\frac{1}{n-l+1}\sum_{i=0}^{n-l}\left(\frac{\sum_{j=1}^l  \widetilde{r}_{i+j}}{\sqrt{l}}\right)^2-4\left(\frac{\sum_{i=0}^{n-l}\sum_{j=1}^l \widetilde{r}_{i+j}}{(n-l+1)\sqrt{l}}\right)^2\right)+o_P(1)\nonumber\\
&\overset{P}\to   4\sigma_{\tau}^2,
\end{align}
 as $n\to\infty$, $l\to\infty$, $l/n\to\infty$, similarly to \eqref{eq:mbestcon}. So, $\widehat{\sigma}_{\tau,mbj}^2=:n\widehat{\sigma}_{mbj}^2/4$ is a consistent estimator of $\sigma_{\tau}^2$.

{\em Moving Block Bootstrap.} In the bootstrap as suggested by \cite{Shao2010}, the assignments $T_i^{\ast}$s are reproduced by a covariate-adaptive randomization procedure and there is not any relationship with the original ones $T_i$s. A response $Y_i^{\ast}$ from the treatment may be regarded as from the control. When the responses from the treatment and the control have different distributions, the bootstrap breaks the structure of the data.    The moving block bootstrap (MBB)  was proposed by \cite{Kusnch1989} and \cite{LiuSingh1992} in the context of bootstrapping dependent data,
in attempts to reproduce different aspects of the dependence structure of the
observed data in the "resampled data". The asymptotic properties of the moving block bootstrap mean and moving block bootstrap variance of a stationary sequence are studied by  \cite{Kusnch1989}, \cite{ShaoYu1993} and
\cite{Lahiri1999}  etc. Here we consider the moving block bootstrap estimator of $\sigma_{\tau}^2$.
 For $1\le l<n$, let $m=[n/l]$. We generate independent and identically random variables $I(0),I(1),\ldots,I(m)$ as a simple random sample with replacement from $\{0,1,\ldots, n-l\}$ and define the
bootstrap data $Y_{bl+j}^{\ast}=Y_{I(b)+j}$, $T_{bl+j}^{\ast}=T_{I(b)+j}$, $\bm X_{bl+j,obs}^{\ast}=\bm X_{I(b)+j,obs}$, $b=0,\ldots,m$, $j=1,\ldots,l$. For the moving block bootstrap data $\{Y_i^{\ast}, T_i^{\ast},\bm X_{i,,obs}^{\ast};i=1,\ldots, (m+1)l\}$, each block $\{Y_{bl+j}^{\ast}, T_{bl+j}^{\ast},\bm X_{bl+j,obs}^{\ast};j=1,\ldots, l\}$ keeps the same dependent structure of the original data $\{Y_{j}, T_{j},\bm X_{j,obs};j=1,\ldots, l\}$. We choose the first $n$ bootstrap observations to estimate the parameters.
With   $\{Y_i^{\ast}, T_i^{\ast},\bm X_{i,obs}^{\ast};i=1,\ldots, n\}$ taking the place of $\{Y_i,T_i,\bm X_{i,obs};i=1,\ldots, n\}$, we fit the working model \eqref{eq:working} and obtain the LSE $\tau^{\ast}$ of $\tau=\mu_1-\mu_0$.
The  moving block bootstrap estimator of the asymptotic variance  of $\widehat{\tau}_n-\tau$  is then $\widehat{\sigma}^2_{mbb} =  \Var_{\ast}(\widehat{\tau}^{\ast})$,
where $\Var_{\ast}$ is the  variance
 with respect to bootstrap sampling. In applications, $\Var_{\ast}(\widehat{\tau}^{\ast})$
can be approximated by Monte Carlo simulation. We independently generate $B$ bootstrap samples
to obtain $\widehat{\tau}_b^{\ast}$ $(b = 1, \ldots, B)$, and approximate $\Var_{\ast}(\widehat{\tau}^{\ast})$  by the sample variance of $\widehat{\tau}_b^{\ast}$ $(b = 1, \ldots, B)$.

Similarly to \eqref{eq:linearRepofTau}, we have
\begin{align*}\sqrt{n}(\widehat{\tau}^{\ast}-\tau)= 2\frac{\sum_{i=1}^{n} \widetilde{r}_i^{\ast}}{\sqrt{n}}  +o_p(1) \text{ with } \widetilde{r}_i^{\ast}= (2T_i^{\ast}-1)(Y_i^{\ast}-\underline{\bm X}_i^{\ast}\theta_{\ast}).
 \end{align*}
 Hence,
\begin{align}\label{eq:mbbcon}
  n\Var_{\ast}  (\widehat{\tau}^{\ast})= & 4\frac{\Var_{\ast}\big(\sum_{i=1}^n \widetilde{r}_i^{\ast}\big)}{n}+o(1)=4\frac{\Var_{\ast}\big(\sum_{i=1}^{ml} \widetilde{r}_i^{\ast}\big)}{ml}+o(1)\nonumber\\
= & 4\frac{\Var_{\ast}\big(\sum_{i=0}^{m-1}\sum_{j=1}^l \widetilde{r}_{I(i)+j}\big)}{ml}+o(1)
=   4\frac{\Var_{\ast}\big( \sum_{j=1}^l \widetilde{r}_{I(0)+j}\big)}{l}+o(1)\nonumber\\
=&4\frac{1}{n-l+1}\sum_{i=0}^{n-l}\left(\frac{\sum_{j=1}^l  \widetilde{r}_{i+j}}{\sqrt{l}}\right)^2-4\left(\frac{\sum_{i=0}^{n-l}\sum_{j=1}^l \widetilde{r}_{i+j}}{(n-l+1)\sqrt{l}}\right)^2+o_P(1)\nonumber\\
\overset{P}\to & 4\sigma_{\tau}^2
\end{align} 
 as $n\to\infty$, $l\to\infty$, $l/n\to\infty$, similarly to \eqref{eq:mbestcon}.
So, $\widehat{\sigma}^2_{\tau,mbb}=:n\widehat{\sigma}^2_{mbb}/4$ is a consistent estimator of $\sigma_{\tau}^2$.

\begin{remark} It should be noted that the estimators $\widehat{\sigma}_{\tau,mb}^2$, $\sigma_{\tau,mbj}^2$  and $\widehat{\sigma}^2_{\tau,mbb}$ are constructed basing only on the working model \eqref{eq:working} and the observed data $(Y_i,T_i,\bm X_{i,obs})$, $i=1,\ldots, n$.  They are consistent estimators of $\sigma_{\tau}^2$ regardless of whether condition   \eqref{eq:inSpan} is satisfied or not. However, simulation studies found that the moving block estimator and the moving block bootstrap estimator converge very slowly.  To obtain a precise estimate, large samples are needed (c.f. \cite{Jones2006}, \cite{Chakraborty2022}).  The moving block jackknife estimator behaves similarly because it is also based on the same block sums $\sum_{j=1}^l \widetilde{r}_{i+j}$. $l$ is usually chosen to be $integer(\sqrt{n})$ or $integer(n^{1/3})$. which is much smaller than $n$. For example, for $n=500$, we have only $22$ or $9$ samples in each block. The variance  of $\sum_{j=1}^l \widetilde{r}_{i+j}/\sqrt{l}$  may be quite different from the variance
of $\sum_{j=1}^n \widetilde{r}_{i}/\sqrt{n}$.  To find a "good" estimator of $\sigma_{\tau}^2$ only under the working model is still a problem.
\end{remark}

The proofs of Theorems \ref{thm:test}, \ref{thm:regest}, \ref{thm:mbest} and Corollaries \ref{cor:test} and \ref{cor:testdis} are given in the Supplementary Materials of the paper. In the Supplementary Materials, simulation studies are also given for comparing the properties of the classical test and adjusted tests.  It is founded that the test adjusted by the regression estimator or the bootstrap estimator performs satisfactorily in restoring the type I error rate and increasing power, and the test adjusted by the moving block estimator, the moving block jackknife estimator or the moving block bootstrap estimator can restore part of the type I error rate, but the type I error rate under the full model is usually inflated a little.

\section{General  CAR Procedure for Multi-Treatments }\label{sec:multi}
\setcounter{equation}{0}

We consider a general procedure to balance covariate features $\bm\phi(\bm X_i)$  among   $T(T \ge 2)$ treatments.    Let
$\bm T_n = (T_n^{(1)},\ldots,T_n^{(T)})$ be the result of the $n$th assignment; that is, if the $n$-th
unit is assigned to treatment $t$, then the $t$-th component $T_n^{(t)}$ is $1$ and
other components are $0$. Then $\sum_{i=1}^n T_i^{(t)} \bm\phi(\bm X_i)$ is total  covariate features in treatment $t$, and $\frac{1}{T}\sum_{i=1}^n \bm\phi(\bm X_i)$ is the average covariate features for each treatment. So, $\bm \Lambda_n^{(t)}=\sum_{i=1}^n \big(T_i^{(t)}-\frac{1}{T}\big) \bm\phi(\bm X_i)$ is the imbalance of covariate features in treatment $t$,  and the vector $\bm\Lambda_n=(\bm \Lambda_n^{(1)}, \ldots, \bm \Lambda_n^{(T)})$ represents all the imbalances. We define the numerical imbalance measure $Imb_n$
  as the squared Euclidean norm of the imbalance vector $\bm \Lambda_n$,
  $$Imb_n=\|\bm \Lambda_n\|^2=\sum_{t=1}^T\Big\|\sum_{i=1}^n \big(T_i^{(t)}-\frac{1}{T}\big)\bm \phi(\bm X_i)\Big\|^2. $$
The $\bm\phi$-based CAR procedure to minimize the imbalance measure $Imb_n$ is defined as follows:

\begin{enumerate}
\item[1)\label{step1}] \emph{Start.} The first unit is assigned to treatment $t=1,\dots,T$ with equal  probability $p=1/T$.

\item[2)\label{step2}]    \emph{Measurement.} Suppose that $(n-1)$ units have been assigned to  treatments $(n > 1)$ and the $n$-th unit
is to be assigned, and the results of assignments $\bm T_1,\ldots, \bm T_{n-1}$ of previous stages and all covariates $\bm X_1,\ldots, \bm X_n$ up  to the $n$-th unit are observed.   Calculate the `potential' imbalance measures $Imb_n^{(t)}=Imb_n\big|_{T_n^{(t)}=1}$,
corresponding to $T_n^{(t)} = 1$, $t=1,\ldots, T$.

\item[3)\label{step3}] \emph{Probability generator.} Assign the $n$-th unit to  treatment $t$ with the probability
 	\begin{align}\label{eq:allocation}
	\Prob(T_n^{(t)}=1|\boldsymbol{X}_n,...,\boldsymbol{X}_{1},\bm T_{n-1},...,\bm T_{1}) =\ell_{n,t}=\ell_t\left( Imb_n^{(s)} - \overline{Imb}_n;s=1,\ldots,T\right),
	\end{align}
	where $\overline{Imb}_n=\frac{1}{T}\sum_{s=1}^TImb_n^{(s)}$, and $\ell_t(\bm x):\mathbb R^T\to (0,1)$ are allocation functions with $\sum_{t=1}^T\ell_{t}(\bm x)=1$.
\item[4)]
	Repeat the last two steps until all units are assigned.
\end{enumerate}

In the definition of $Imb_n$, the norm $\|\bm x\|^2$ can be replaced by $\|\bm x\|_{\Sigma}^2=\bm x\Sigma\bm x^{\prime}$, where $\Sigma$ is a positive definite symmetric matrix. For the allocation probability in \eqref{eq:allocation}, we can define the allocation functions following \cite{Pocock1975}. Let $\kappa_{1}\ge \kappa_2\ge\ldots \kappa_T$ be $T$ ordered positive probabilities  with $\sum_{t^\prime=1}^T \kappa_{t^\prime} = 1$ and $\kappa_{1}>\kappa_{T}>0$.
Suppose that the $Imb_j^{\left((t)\right)}$ is the $t$-th order statistic that is the $t$-th smallest value of the random sample $Imb_j^{(1)},\ldots, Imb_j^{(T)}$. One thus can rank the treatments according to the values of $Imb_j^{(t)}$, $t = 1,\ldots, T$, in a non-decreasing order so that
$$
 Imb_j^{((1))}\le Imb_j^{((2))}\le\ldots\le Imb_j^{\left((T)\right)}.\nonumber
$$
 Define $\ell_{n,t}$ to be a function with respect to the ranking $R_t$ of $Imb_j^{(t)}$ among the random sample $Imb_j^{(t^\prime)},~t^\prime=1,\ldots,T$, and
 $\ell_{n,t}=\kappa_{R_t}$, or equivalently,
\begin{equation}\label{eq:PS-allocat}
 \ell_{n,(t)}=\kappa_{t}, \;\;,~t=1,\ldots,T.
\end{equation}
The choice of $\kappa_t$s can be found in \cite{Pocock1975}. A $\bm\phi$-based CAR procedure with Pocock and Simon's allocation assigns the $n$-th unit to a treatment having the smallest  `potential' imbalance with the largest probability, to a treatment having the second smallest  `potential' imbalance with the second largest probability, etc.

 Another   simple way to define allocation probabilities is by defining
\begin{equation}\label{eq:con-allocat}
\ell_t(\bm x)=\frac{1-\Phi((-K_0)\vee x_t\wedge K_0)}{ \sum_{s=1}^T\big(1-\Phi((-K_0)\vee x_s\wedge K_0)\big)}, \;\; t=1,\ldots, T,
\end{equation}
 where $\Phi(\cdot)$ is the standard normal distribution and $K_0$ is a constant. In practice, we can choose $K_0=3$ because $1-\Phi(3)$ is a small probability.  The function $1-\Phi(x)$ can be replaced by any other positive and decreasing function, for example, $e^{-x}$.

To obtain the general  results for the general CAR procedures, we require the following
assumptions on the allocation  functions $\ell_t(\bm x)$:
\begin{assumption}\label{asm:allocation}
\begin{enumerate}
  \item[A1)]  $\ell_t(\bm x)$ is nondecreasing in $x_t$, i.e.,  $\ell_t(\bm x)\le \ell_s(\bm x)$ whenever $x_t>x_s$;
  \item[A2)] The randomization is not the complete randomization, i.e., there exist constants $\delta_0>0$ and $b_0>0$ such that $\max_{t,s}|\ell_t(\bm x)-\ell_s(\bm x)|\ge \delta_0$ whenever $\max_{t,s}|x_t-x_s|\ge b_0$;
  \item[A3)] The allocation keeps aloof from the  deterministic allocation, i.e., there exists a constant $0<\rho<1/T$ such that $\ell_k(\bm x)\ge \rho$ for all $\bm x$;
  \item[A4)]  The allocation probabilities are symmetric about treatments, i.e., $\ell_{\Pi(t)}(x_{\Pi(1)},\ldots,x_{\Pi(T)})=\ell_{t}(x_1,\ldots,x_T)$, $t=1,\ldots, T$,  for any permutation $\Pi=(\Pi(1),\ldots,\Pi(T))$ of treatments $(1,\ldots, T)$.
\end{enumerate}
\end{assumption}

\begin{theorem}\label{thm:multiCAR}
Suppose that Assumptions  \ref{asm:iid} and \ref{asm:moment} with $\gamma>2$    hold and A1)-A2) of Assumption \ref{asm:allocation}   for the allocation function $l_t(\bm x)$s are satisfied.
Then $\Ex[\|\bm{\Lambda}_n\|^r]=O(1)$ for all $r\in (0,\gamma-1)$ and $\|\bm \Lambda_n\|=o(n^{\frac{1}{\gamma}+\epsilon})$ a.s. for any $\epsilon>0$.
\end{theorem}

\begin{theorem}\label{thm:multiCARrecurrent} Suppose  that   Assumptions \ref{asm:iid}, \ref{asm:moment}   and \ref{asm:recurrent} hold and A1)-A3) of Assumption \ref{asm:allocation}   for the allocation function $l_t(\bm x)$s are satisfied.
Then $\bm{\Lambda}_n$ is a positive Harris recurrent Markov chain and  $\Ex[\|\bm{\Lambda}_n\|^{\gamma-1}]=O(1)$.
\end{theorem}

\begin{theorem} \label{thm:multiclt} Suppose  that Assumptions \ref{asm:iid}, \ref{asm:moment} (with $\gamma \ge 3$), \ref{asm:recurrent} and  \ref{asm:allocation} hold. Moreover,
assume that  $(\bm X_i, Z_i,\bm W_i)$, $i=1,2,\ldots$, are i.i.d. random vectors with the same distribution as $(\bm X,Z,\bm W)$, $\bm W=(W^{(1)}, \ldots, W^{(T)})$, $\Ex[Z^2]<\infty$, $\Ex[\|\bm W\|^2]<\infty$, $\Ex[\bm W]=\bm 0$.
Then 
\begin{description}
  \item[\rm (i)]  for any function  $m(\bm X)$ of $\bm X$ with $\Ex[m^2(\bm X)]<\infty$,
  there is a $\vec{\sigma}_m\ge 0$ such that
\begin{equation}\label{eq:multiclt} \frac{\sum_{i=1}^n (T\bm T_i-1)m(\bm X_i)}{\sqrt{n}}\overset{D}\to N(\bm 0, \Vec{\bm \Sigma}_m) \text{ and }
\frac{\Ex\big[\big(\sum_{i=1}^n (T\bm T_i-1)m(\bm X_i)\big)^{\otimes 2}\big]}{n} \to   \Vec{\bm \Sigma}_m,
\end{equation}
with
\begin{equation}\label{eq:multivarianceM}\Vec{\bm \Sigma}_m=\vec{\sigma}_m^2\Big(T\bm I-(1,\ldots,1)(1,\ldots,1)^{\prime}\Big),
\end{equation}
where $\bm I$ is an identity matrix;
\item[\rm (ii)]  there is a $\vec{\sigma}_Z\ge 0$ such that
\begin{equation}\label{eq:multiclt2}
\begin{aligned}
  \frac{\Ex\left[\big(\sum_{i=1}^n  (T T_i^{t}-1)Z_i\big)^{2}\right]}{n} \to & (T-1)\vec{\sigma}_Z^2,\; t=1,\ldots,T, \\
   \frac{\Ex\left[\big(  \sum_{i=1}^n(T\bm T_i-1)Z_i, \sum_{i=1}^n\bm W_i \big)^{\otimes 2}\right]}{n}
  & \to   \begin{pmatrix} \vec{\bm\Sigma}_Z & \bm 0 \\ \bm 0 & \Var(\bm W)\end{pmatrix} \; \text{ and} \\
 \left(\frac{\sum_{i=1}^n (T\bm T_i-1)Z_i}{\sqrt{n}}, \frac{\sum_{i=1}^n\bm W_i}{\sqrt{n}}\right)\overset{D}\to & N(\bm 0,\vec{\bm \Sigma}_Z)\times N(\bm 0,\Var(\bm W)),
 \end{aligned}
\end{equation}
where
\begin{equation}\label{eq:multivarianceZ}\vec{\bm\Sigma}_Z=\vec{\sigma}_Z^2\Big(T\bm I-(1,\ldots,1)(1,\ldots,1)^{\prime}\Big),\;\;\vec{\sigma}_Z^2=\Ex(Z-\Ex[Z|\bm X])^2+\vec{\sigma}_m^2
 \end{equation}
 and   $\vec{\sigma}_m^2$ is the same as that in \eqref{eq:multivarianceM} with $m(\bm X)=\Ex[Z|\bm X]$;
 \item[\rm (iii)] 
 \begin{equation}\label{eq:multiLIL}\sum_{i=1}^n \left\{(T\bm T_i-1)Z_i+\bm W_i\right\}=O(\sqrt{n\log\log n}) \; a.s.;
 \end{equation}
\item[\rm (iv)] 
$\vec{\sigma}_Z^2=0$ if and only if $Z=\langle\bm x_0,\bm\phi(\bm X)\rangle$ a.s. for some   $\bm x_0\in \mathbb R^{q}$.
\end{description}
 \end{theorem}

\begin{remark} $\vec{\sigma}_Z^2$ is determined by $Z$, the feature $\bm\phi$  and the allocation functions $\ell_t$ so that it be written as $\vec{\sigma}_Z^2(Z;\bm\phi,\ell)$. With $\bm\phi(\bm X)$ taking the place of $\bm X$,  applying \eqref{eq:multiclt2}  yields
\begin{equation}\label{eq:variance-m-conditionM}\vec{\sigma}_Z^2=\Var\big(Z-\Ex[Z|\bm\phi(\bm X)]\big) +\vec{\sigma}_{\Ex[Z|\bm\phi(\bm X)]}^2.
\end{equation}
For the complete randomization,  \eqref{eq:multiclt2} holds with $\vec{\sigma}_m^2=\Ex[Z^2]$.  It is expected that $\vec{\sigma}_Z^2\le \Ex[Z^2]$. It is obvious that, if $\Ex[Z|\bm\phi(\bm X)]\in Span\{\bm\phi(\bm X)\}$, then $\vec{\sigma}_{\Ex[Z|\bm\phi(\bm X)]}^2=0$ and $\vec{\sigma}_Z^2\le \Ex[Z^2]$. 
We don't know when this equality will hold in general.
 \end{remark}

 \begin{remark} The assumption A4) for the allocation function $\ell_t(\bm x)$ is only used in the proof of Theorem \ref{thm:multiclt}. If this condition is not satisfied, it may happen that
 $$ \frac{\sum_{i=1}^n (TT_i^{(t)}-1)m(\bm X_i)}{n}\overset{P}\to c_t\ne 0. $$
 \end{remark}

 Under the Assumptions in Theorem \ref{thm:multiclt}, we have the following corollary on the moment inequality.

\begin{corollary}\label{cor:moments}
 Suppose  that Assumptions \ref{asm:iid}, \ref{asm:moment} (with $\gamma \ge 3$), \ref{asm:recurrent} and  \ref{asm:allocation} hold.
 Assume that $\{Z_i,W_i, i=1,2,\ldots,n\}$ is an array of random variables such that $(\bm X_i, Z_i,W_i)$, $i=1,2,\ldots,n$, are i.i.d. random vectors with the same distribution as $(\bm X,Z,W)$, and $\Ex[Z^2]<\infty$, $\Ex[W^2]<\infty$, $\Ex[W]=0$.
 Then there exists a positive constant $C_{\Lambda}$ which depends only on the Markov Chain $\{\bm \Lambda_n\}$,  such that
\begin{equation} \Ex\left[\max_{l\le m}\big(\sum_{j=i+1}^{i+l} \left\{(TT_j^{(t)}-1)Z_j+W_j\right\}\big)^2\right]\le C_{\Lambda} m (\Ex[Z^2]+\Ex[W^2]), \;\; i+m\le n.
\end{equation}
\end{corollary}

The proofs of the theorems and corollary \ref{cor:moments} will be given in supplementary materials of the paper.  Examples \ref{ex:discrete1}-\ref{ex:mixing} can be extended to the multi-treatment case. In particular,
 we have the following corollary on the stratified randomization, Pocock and Simon's procedure, and Hu and Hu's procedure for balancing discrete covariates among $T$ treatments.
\begin{corollary}\label{cor:multiHuHu} Suppose $\bm X=(X_1,\ldots,X_p)$ with  each $X_t$ taking    finite many values $x_t^{(l_t)}$, $l_t=1\ldots, L_t$. Consider a $\bm\phi$-based CAR procedure with $\bm\phi(\bm X)=\bm\phi_{HH}(\bm X)$ and allocation functions $\ell_k(\bm x)$ satisfying Assumption \ref{asm:allocation} , i.e., Hu and Hu's procedure is to used for randomizing units to $T$ treatments.
\begin{itemize}
  \item[(i)] Suppose $w_s\ne 0$. Then for any function $m(\bm X)$ of $\bm X$ we have $\sum_{i=1}^n(T_i^{(t)}-1/T)m(\bm X_i)=O_P(1)$, $t=1,\ldots,T$.
  \item[(ii)] Suppose $w_s=0$. Then for the covariate $X_s$ with $w_{m,s}\ne 0$ we have $\sum_{i=1}^n(T_i^{(t)}-1/T)f_s(X_{i,s})=O_P(1)$, $t=1,\ldots,T$, for any function $f_s(X_s)$ of $X_s$.  Furthermore, for  any function $m(\bm X)$ of $\bm X$, there exists a $\vec{\sigma}_m\ge 0$ such that
      $$ \lim\frac{\Ex\left[(\sum_{i=1}^n(T_i^{(t)}-1/T)m(\bm X_i))^2\right]}{n}=\vec{\sigma}_m^2\; \text{ and }\; \frac{\sum_{i=1}^n(T_i^{(t)}-1/T)m(\bm X_i)}{\sqrt{n}}\overset{D}\to N(0,\vec{\sigma}_m^2), $$
    $t=1,\ldots,T$. Also, $\vec{\sigma}_m=0$ if and only if $m(\bm X)=\sum_{s: w_{m,s}\ne 0}f_s(X_s)$ for some functions $f_1(X_1),\ldots,f_p(X_p)$.
  \item[(iii)] For any case, $\sum_{i=1}^n(T_i^{(t)}-1/T)=O_P(1)$, $t=1,\ldots,T$.
\end{itemize}

\end{corollary}

Corollary \ref{cor:multiHuHu} gives a general and complete picture of the properties of the stratified randomization, Pocock and Simon's procedure, and Hu and Hu's procedure.


\bigskip
{\bf Acknowledgements:} {The research was partly   supported by  grants from National Key R\&D Program of China (No. 2024YFA1013502), NSF of China (No. U23A2064) and the Summit Advancement Disciplines of Zhejiang Province (Zhejiang Gongshang University - Statistics).}

\bibliographystyle{apalike}

\newpage

\appendix
\setcounter{equation}{0}
\renewcommand{\theequation}{A.\arabic{equation}}
\setcounter{table}{0}
\renewcommand{\thetable}{C.\arabic{table}}
\setcounter{table}{0}
\renewcommand{\thefigure}{C.\arabic{table}}

\begin{center} {\Large \bf Supplementary Materials}
\end{center}

We give the proofs of the main results of the paper and some simulation studies.

\section{Proofs of Main Results}

\subsection{Proofs of the Properties of the  CARs}

 Theorems \ref{thm:DesignPropertyGeneral} and \ref{thm:recurrent} are special cases of Theorem \ref{thm:multiCAR} and \ref{thm:multiCARrecurrent}, respectively, and Theorems \ref{thm:clt} and   \ref{thm:generalclt} are special cases of Theorem \ref{thm:multiclt}. We only need to give the proofs for the general multi-treatment case.

Recall $\bm \Lambda_n^{(t)}=\sum_{i=1}^n (T_i^{(t)}-\frac{1}{T})\bm \phi(\bm X_i)$, $\bm\Lambda_n=(\bm \Lambda_n^{(1)},\ldots,\bm \Lambda_n^{(T)})$, $\|\bm \Lambda_n\|^2=\sum_{t=1}^T \|\bm \Lambda_n^{(t)}\|^2$, and   for   two row vectors $\bm u$ and $\bm v$,  $\langle\bm u,\bm v\rangle=\bm u\bm v^{\prime}$, $\|\bm u\|=\langle\bm u,\bm u\rangle^{1/2}$. Write $\mathscr{F}_n=\sigma(\bm T_1,\bm X_1,\cdots,\bm T_n,\bm X_n)$ be the history sigma field.  To prove the results, we need some properties of Markov Chain $\{\bm\Lambda_n\}$. The first is about the drift condition.
\begin{lemma}\label{lem:drift} Suppose that Assumptions  \ref{asm:iid} and \ref{asm:moment}    hold and A1)-A2) of Assumption \ref{asm:allocation}   for the allocation function $l_t(\bm x)$s are satisfied. Then for any $\gamma_0\in [0,\gamma]$ there exist positive constants $b_{\gamma_0}$,    $c_{\gamma_0}$  and $d_{\gamma_0}$ such that
 \begin{equation}\label{eq:lem:drift1}
  P_{\lambda}\left[\|\bm\Lambda\|^{\gamma_0}\right]-\|\bm\Lambda\|^{\gamma_0}
  \le   -c_{\gamma_0}   \|\bm   \Lambda\|^{\gamma_0 -1}
  +b_{\gamma_0}\mathbb{I}\{\|\bm \Lambda\|\le d_{\gamma_0}\},
\end{equation}
where $P_{\lambda}\left[f(\bm\Lambda)\right]=\int f(\bm y) P_{\lambda}(\bm \Lambda,d\bm y)$ and $P_{\lambda}(\bm x,A)=\Prob(\bm \Lambda_n\in A|\bm\Lambda_{n-1}=\bm x)$ is the transition probability of $\bm\Lambda$.
\end{lemma}

\begin{proof}[Proof] It is sufficient to show that
 \begin{equation}\label{eq:multi-drift-p-2}
  \Ex\left[\|\bm\Lambda_{n+1}\|^{\gamma_0}\big |\mathscr{F}_n\right]-\|\bm\Lambda_n\|^{\gamma_0}
  \le   -c_{\gamma_0}   \|\bm   \Lambda_n\|^{\gamma_0 -1}
  +b_{\gamma_0}\mathbb{I}\{\|\bm \Lambda_n\|\le d_{\gamma_0}\},
\end{equation}

 Recall  $Imb_n^{(t)}=\|\bm \Lambda_n\|^2\big|_{T_n^{(t)}=1}$, $\sum_{t=1}^T \bm \Lambda_{n}^{(t)}=0$. We have
\begin{align}
&\|\bm \Lambda_{n+1}\|^2-\|\bm \Lambda_n\|^2=  2\sum_{t=1}^T (T_{n+1}^{(t)}-\frac{1}{T})\langle\bm \Lambda_{n}^{(t)},\bm\phi(\bm X_{n+1})\rangle+(1-\frac{1}{T})\|\bm\phi(\bm X_{n+1})\|^2,\nonumber
\\
&Imb_{n+1}^{(t)}= \|\bm \Lambda_{n}\|^2+2\langle\bm \Lambda_{n}^{(t)},\bm\phi(\bm X_{n+1})\rangle+(1-\frac{1}{T})\|\bm\phi(\bm X_{n+1})\|^2,\nonumber\\
 \label{eq:Imbdifference}
&Imb_{n+1}^{(t)}- Imb_{n+1}^{(t^{\prime})}= 2\langle\bm \Lambda_{n}^{(t)}-\bm \Lambda_{n}^{(t^{\prime})},\bm\phi(\bm X_{n+1})\rangle.
\end{align}
Note that $\bm\Lambda_{n+1}=\bm\Lambda_n+\bm \psi_{n+1}$ where $\bm\psi_{n+1}=\big((T_{n+1}^{(1)}-\frac{1}{T})\bm \phi(\bm X_{n+1}),\ldots, (T_{n+1}^{(T)}-\frac{1}{T})\bm \phi(\bm X_{n+1})\big)$, and
 $\|\bm\psi_{n+1}\|^2=\sum_{t=1}^T (T_{n+1}^{(1)}-\frac{1}{T})^2\|\bm\phi(\bm X_{n+1})\|^2=(1-\frac{1}{T})\|\bm\phi(\bm X_{n+1})\|^2$.
From the elementary inequality that
\begin{equation}\label{eq:element1}\|\bm u+\bm v\|^{\gamma_0}-\|\bm u\|^{\gamma_0}\le \gamma_0 \langle\bm u,\bm v\rangle\|\bm u\|^{\gamma_0 -2}+c_{\gamma_0}   \left(\|\bm v\|^{\gamma_0}+\|\bm v\|^2 \|\bm u\|^{\gamma_0 -2}\right), \; \gamma_0 \ge 2,
\end{equation}
where $c_{\gamma_0}   $ is a constant which depends only on $\gamma_0$,  we have
\begin{align*}
  &\Ex\left[\|\bm\Lambda_{n+1}\|^{\gamma_0}|\mathscr F_n,\bm X_{n+1}\right]-\|\bm\Lambda_n\|^{\gamma_0}\\
  \le &    \Ex\left[\gamma_0 \langle\bm \Lambda_n,\bm \psi_{n+1}\rangle\|\bm   \Lambda_n\|^{\gamma_0 -2}
  +c_{\gamma_0}   \|\bm \psi_{n+1}\|^{\gamma_0}
 + c_{\gamma_0} \|\bm \psi_{n+1}\|^2\|\bm   \Lambda_n\|^{\gamma_0 -2} \big|\mathscr F_n,\bm X_{n+1}\right]\\
 \le & \Ex\left[ \gamma_0 \sum_{t=1}^T (T_{n+1}^{(t)}-\frac{1}{T})\langle\bm \Lambda_n^{(t)},\bm \phi(\bm X_{n+1})\rangle\|\bm   \Lambda_n\|^{\gamma_0 -2}\big|\mathscr F_n,\bm X_{n+1}\right]\\
  &+c_{\gamma_0}\left(\|\bm \phi(\bm X_{n+1})\|^{\gamma_0}
 +   \|\bm   \phi(\bm X_{n+1})\|^2\|\bm   \Lambda_n\|^{\gamma_0 -2}\right)\\
= & \gamma_0 \sum_{t=1}^T (\ell_{n+1,t}-\frac{1}{T})\langle\bm \Lambda_n^{(t)},\bm \phi(\bm X_{n+1})\rangle\|\bm   \Lambda_n\|^{\gamma_0 -2}
   +  c_{\gamma_0}\left(\|\bm \phi(\bm X_{n+1})\|^{\gamma_0}+ \|\bm   \phi(\bm X_{n+1})\|^2\|\bm   \Lambda_n\|^{\gamma_0 -2}\right)
\end{align*}
by noting
$
	\Prob(T_{n+1}^{(t)}=1|\mathscr{F}_n, \boldsymbol{X}_{n+1}) =\ell_{n+1,t}
	$.
Let $t_{\max}$ (resp. $t_{\min}$) be the $t$ for which $Imb_{n+1}^{(t)}$ is the largest (resp. the smallest) and $\xi_n=\max_{t, s}\big|\langle\bm \Lambda_n^{(t)}-\bm\Lambda_n^{(s)},\bm \phi(\bm X_{n+1})\rangle\big|$. Notice that  $\sum_{t=1}^T(\ell_{n+1,t}-1/T)=0$,  and $\ell_{n+1,t}\ge \ell_{n+1,t^{\prime}}$ if  $Imb_{n+1}^{(t)}< Imb_{n+1}^{(t^{\prime})}$ which is equivalent to $\langle\bm \Lambda_n^{(t)}-\bm\Lambda_n^{(t^{\prime})},\bm \phi(\bm X_{n+1})\rangle<0$ by \eqref{eq:Imbdifference}. We have
\begin{align*}
   2T\sum_{t=1}^T &(\ell_{n+1,t}-\frac{1}{T})\langle\bm \Lambda_n^{(t)},\bm \phi(\bm X_{n+1})\rangle
 =    \sum_{t, t^{\prime}=1}^T (\ell_{n+1,t}-\ell_{n+1,t^{\prime}} )\langle\bm \Lambda_n^{(t)}-\bm\Lambda_n^{(t^{\prime})},\bm \phi(\bm X_{n+1})\rangle\\
 = & -  \sum_{t, t^{\prime}=1}^T |\ell_{n+1,t}-\ell_{n+1,t^{\prime}} |\cdot \big|\langle\bm \Lambda_n^{(t)}-\bm\Lambda_n^{(t^{\prime})},\bm \phi(\bm X_{n+1})\rangle\big|\\
 \le & -  |\ell_{n+1,t_{\max}}-\ell_{n+1,t_{\min}} |\cdot \big|\langle\bm \Lambda_n^{(t_{\max})}-\bm\Lambda_n^{(t_{\min})},\bm \phi(\bm X_{n+1})\rangle\big|  \\
 = &- \max_{t, s} |\ell_{n+1,t}-\ell_{n+1,s} |\cdot \max_{t, s} \big|\langle\bm \Lambda_n^{(t)}-\bm\Lambda_n^{(s)},\bm \phi(\bm X_{n+1})\rangle\big|\\
 \le & -  \delta_0 \xi_nI\{\xi_n\ge b_0\}= -  \delta_0 (\xi_n- b_0),
 \end{align*}
 by \eqref{eq:Imbdifference} and A2) of Assumption \ref{asm:allocation}  for the allocation functions $l_t(\bm x)$.
  We  will show that,  there exist a  positive constant   $c_1$   such that
\begin{eqnarray}\label{eq:bounded-multi}
  \Ex\left[\xi_n \big |\mathscr{F}_n\right] =\Ex\left[\xi_n \big |\bm \Lambda_n\right] \ge c_1\|\bm{\Lambda}_n\|.
\end{eqnarray}
If \eqref{eq:bounded-multi} is proved, then
\begin{align}\label{eq:multi-drift-p-3}
& \Ex\left[\|\bm\Lambda_{n+1}\|^{\gamma_0}\big |\mathscr{F}_n\right]-\|\bm\Lambda_n\|^{\gamma_0}\nonumber\\
\le &\begin{cases} c_{\gamma_0}\left(\beta_{\gamma_0}+\beta_2\|\bm   \Lambda_n\|^{\gamma_0 -2}\right), & \\
 -\frac{\gamma_0\delta_0c_1}{4T}\|\bm   \Lambda_n\|^{\gamma_0 -1}
  +c_{\gamma_0}\left( \beta_{\gamma_0}+\beta_2\|\bm   \Lambda_n\|^{\gamma_0 -2}\right), & \text{ if } \|\bm\Lambda_n\|\ge 2b_0/c_1,
  \end{cases}
\end{align}
 where $\beta_{\gamma_0}=\Ex[\|\bm \phi(\bm X)\|^{\gamma_0}]$ and $\beta_2=\Ex[\|\bm \phi(\bm X)\|^2]$.
By applying the following   elementary inequality instead of \eqref{eq:element1}:
\begin{equation}\label{eq:element2} \|\bm u+\bm v\|^{\gamma_0}-\|\bm u\|^{\gamma_0}\le  \gamma_0\langle\bm u,\bm v\rangle\|\bm u\|^{\gamma_0-2}+\frac{\gamma_0}{2}\|\bm v\|^2 \|\bm u\|^{\gamma_0-2},\;\; 0< \gamma_0\le 2,
\end{equation}
 and repeating the proof of  \eqref{eq:multi-drift-p-3}, it follows that
$$ \Ex\left[\|\bm\Lambda_{n+1}\|^{\gamma_0}\big |\mathscr{F}_n\right]-\|\bm\Lambda_n\|^{\gamma_0} \le
 -\frac{\gamma_0\delta_0c_1}{4T}\|\bm   \Lambda_n\|^{\gamma_0 -1}+\frac{\gamma_0}{2}\beta_2 \|\bm   \Lambda_n\|^{\gamma_0 -2},\;\; \text{ if } \|\bm\Lambda_n\|\ge 2b_0/c_1, $$
by \eqref{eq:bounded-multi}.   It follows that \eqref{eq:multi-drift-p-2} holds for all $1\le \gamma_0\le \gamma$.

For showing \eqref{eq:bounded-multi}, without loss of generality we can assume that
\begin{equation}\label{eq:zeromeanIdent} \Sigma=\Ex\left[(\bm\phi(\bm X))^{\prime}\bm \phi(\bm X)\right] \; \text{ is   non-singular}.
\end{equation}
In fact,  if $\Sigma=\bm 0$, then $\bm\phi(\bm X)=\bm 0$ a.s. and \eqref{eq:bounded-multi} is obvious.  When $\Sigma\ne\bm 0$, there is an orthogonal matrix $U$ such that $\Sigma=U diag(\sigma_1^2,\ldots,\sigma_{q_{\ast}}^2,0,\ldots) U^{\prime}$, where $\sigma_j^2$s are positive numbers. Let $\widetilde{\bm \phi}=\bm \phi(\bm X) U^{\prime}$ and $\widetilde{\bm \phi}_i=\bm \phi(\bm X_i) U^{\prime}$. Then $\Ex[\widetilde{\bm \phi}^{\prime}\widetilde{\bm \phi}]=diag(\sigma_1^2,\ldots,\sigma_{q_{\ast}}^2,0,\ldots)$, $\bm\Lambda_n^{(t)}=\sum_{i=1}^n(T_i^{(t)}-\frac{1}{T})\widetilde{\bm \phi}_iU$,
$\|\bm \Lambda_n^{(t)}\|^2=\|\sum_{i=1}^n(T_i^{(t)}-\frac{1}{T})\widetilde{\bm \phi}_i\|^2$, and $\langle\bm \Lambda_n^{(t)},\bm \phi_{n+1}\rangle
=\langle\sum_{i=1}^n(T_i^{(t)}-\frac{1}{T})\widetilde{\bm \phi}_i,\widetilde{\bm \phi}_{n+1}\rangle$. So, it is sufficient to consider $\widetilde{\bm \phi}=(\widetilde{\phi}^{(1)},\ldots,\widetilde{\phi}^{(q)})$. Since $\Prob(\widetilde{\phi}^{(t)}=0, t=q_{\ast}+1,\ldots q)=1$,   it is sufficient to consider the subverter $\widetilde{\bm\psi} =(\widetilde{\phi}^{(1)},\ldots,\widetilde{\phi}^{(q_{\ast})})$
with $\Ex[\widetilde{\bm\psi}^{\prime}\widetilde{\bm\psi}]=diag(\sigma_1^2,\ldots,\sigma_{q_{\ast}}^2)$ being non-singular. 

Now under \eqref{eq:zeromeanIdent}, for \eqref{eq:bounded-multi} it is sufficient to show that 
\begin{align}\label{eq:infmulti2}  \inf\Big\{  \Ex\big[ \max_{t, s} \big|\langle\bm a^{(t)}-\bm a^{(s)},\bm \phi(\bm X)\rangle\big| \big]:
 \; \sum_{t=1}^T\|\bm a^{(t)}\|^2=1, \sum_{t=1}^T\bm a^{(t)}=\bm 0\Big\}>0.
\end{align}
If \eqref{eq:infmulti2} does not hold, then there exists an $\bm a=(\bm a^{(1)},\ldots, \bm a^{(T)})$ with $\|\bm a\|^2=\sum_{t=1}^T\|\bm a^{(t)}\|^2=1$ and $\sum_{t=1}^T\bm a^{(t)}=\bm 0$ such that
$\Ex\big[ \max_{t, s} \big|\langle\bm a^{(t)}-\bm a^{(s)},\bm \phi(\bm X)\rangle\big|\big]=0. $
Thus, $\langle\bm a^{(t)}-\bm a^{(s)},\bm \phi(\bm X)\rangle=0$ a.s. for all $t$ and $s$. Hence
$$ (\bm a^{(t)}-\bm a^{(s)})\Sigma (\bm a^{(t)}-\bm a^{(s)})^{\prime}=\Ex\langle\bm a^{(t)}-\bm a^{(s)},\bm \phi(\bm X)\rangle^2=0, \;\; t,s=1,\ldots, T, $$ 
which implies $\bm a^{(t)}-\bm a^{(s)}=\bm 0$ for all $t$ and $s$ since $\Sigma$ is non-singular. It follows that
$$ \bm a^{(t)}=\frac{1}{T}\sum_{s=1}^T\bm a^{(s)}=\bm 0,\;\; t=1,\ldots, T. $$
Hence $\bm a=\bm 0$.  We obtain a contradiction and so  \eqref{eq:infmulti2} is proved.

At last, we show the elementary inequalities \eqref{eq:element1} and \eqref{eq:element2}. Let $g(t)=t^{\gamma_0/2}$ ($t\ge 0$). Then  $g^{\prime}(t)=\frac{\gamma_0}{2}t^{\gamma_0/2-1}$ and $g^{\prime \prime}(t)=\frac{\gamma_0}{2}\big(\frac{\gamma_0}{2}-1\big)t^{\gamma_0/2-2}$. When $0\le \gamma_0\le 2$, $g^{\prime\prime}(t)\le 0$, and so
\begin{align*}
&\|\bm u+\bm v\|^{\gamma_0}-\|\bm u\|^{\gamma_0}=g(\|\bm u+\bm v\|^2)-g(\|\bm u\|^2)\\
\le &(\|\bm u+\bm v\|^2-\|\bm u\|^2)g^{\prime}(\|\bm u\|^2)
= \big(2\langle\bm u,\bm v\rangle+\|\bm v\|^2\big)\frac{\gamma_0}{2}\|\bm u\|^{\gamma_0-2}.
\end{align*}
\eqref{eq:element2} is proved. When $\gamma_0\ge 2$, let $x=\|\bm u\|$ and $y=\|\bm u+\bm v\|-\|\bm u\|$. Then
$$
\|\bm u+\bm v\|^{\gamma_0}-\|\bm v\|^{\gamma_0}=  |x+y|^{\gamma_0}-|x|^{\gamma_0}=y \gamma_0 x^{\gamma_0-1}+\frac{1}{2}y^2\gamma_0(\gamma_0-1)z^{\gamma_0-2},
$$
where $z$ is between $x$ and $x+y$. Notice that $|y|\le \|\bm v\|$ and $|z|\le \|\bm u\|+\|\bm v\|$. It follows that
\begin{align*}
y^2z^{\gamma_0-2}\le \|\bm v\|^2 \cdot c_{\gamma_0}(\|\bm u\|^{\gamma_0-2}+\|\bm v\|^{\gamma_0-2})=c_{\gamma_0}(\|\bm v\|^2\|\bm u\|^{\gamma_0-2}+\|\bm v\|^{\gamma_0}).
\end{align*}
On the other hand,
\begin{align*}
y=&\|\bm u+\bm v\|-\|\bm u\|=\frac{\|\bm u+\bm v\|^2-\|\bm u\|^2}{\|\bm u+\bm v\|+\|\bm u\|}=\frac{2\langle\bm u,\bm v\rangle+\|\bm v\|^2}{\|\bm u+\bm v\|+\|\bm u\|} \\
=&\frac{\langle\bm u,\bm v\rangle}{\|\bm u\|}+\langle\bm u,\bm v\rangle\frac{\|\bm u\|-\|\bm u+\bm v\|}{\|\bm u\|(\|\bm u+\bm v\|+\|\bm u\|)}+\frac{\|\bm v\|^2}{\|\bm u+\bm v\|+\|\bm u\|}\\
\le & \frac{\langle\bm u,\bm v\rangle}{\|\bm u\|}+\frac{(\|\bm u\|\cdot\|\bm v\|)\|\bm v\|}{\|\bm u\|(\|\bm u+\bm v\|+\|\bm u\|)}+\frac{\|\bm v\|^2}{\|\bm u+\bm v\|+\|\bm u\|}
\le  \frac{\langle\bm u,\bm v\rangle}{\|\bm u\|}+\frac{2\|\bm v\|^2}{\|\bm u\|}.
\end{align*}
\eqref{eq:element1} is proved.
\end{proof}

The next lemma is about the irreducibility of the Markov Chain $\{\bm\Lambda_n\}$.  For the concepts of $\psi$-irreducibility, $T$-Chain and petite set, we refer to \cite{Meyn2009}.
Note $\sum_{t=1}^T\bm \Lambda_n^{(t)}=0$. The true dimension of the vector $\bm \Lambda_n$ is at most $q(T-1)$. We need to represent  it by a $q(T-1)$-dimensional vector.
Recall $\bm T_i=(T_1^{(1)},\ldots, T_i^{(T)})$. Write $\bm T_i-T_i^{(T)}=(T_1^{(1)}-T_i^{(T)},\ldots, T_i^{(T-1)}-T_i^{(T)})$, and
$$\bm\Upsilon_n=\sum_{i=1}^n(\bm T_i-T_i^{(T)})\odot\phi(\bm X_i) =\sum_{i=1}^n\big((T_i^{(1)}-T_i^{(T)})\bm\phi(\bm X_i),\ldots,(T_i^{(T-1)}-T_i^{(T)})\bm\phi(\bm X_i)\big),$$
where in $\bm T_i-T_i^{(T)}$ we have removed the last element which is always zero so that $\bm T_i-T_i^{(T)}$ is a $(T-1)$-dimensional vector.
Then $\bm\Lambda_n\to \bm\Upsilon_n$ is a one to one linear map. It is sufficient to consider the irreducibility of the $q(T-1)$-dimensional Markov Chain $\{\bm\Upsilon_n\}$.

Let $\bm\Upsilon_{n,a}=\sum_{i=1}^n(\bm T_i-T_i^{(T)})\odot\bm \phi^{(a)}(\bm X_i)$, $a=1,2$.
Let $\mathscr{Y}=\{\sum_{i=1}^n(\bm t_i-t_i^{(T)})\odot \bm a_i: \bm a_i \in \mathscr{A}, \bm t_i=(1,0,\ldots, 0), \ldots, (0,\ldots,0,1), i=1,\ldots,n, n=1,2,\ldots\}$ be space of all   values of $\{\bm\Upsilon_{n,1}\}$. It can be check that if $\bm x,\bm y\in \mathscr{Y}$, then $\bm x+\bm y\in \mathscr{Y}$, $-\bm x\in \mathscr{Y}$, and   $\bm\Upsilon_n=(\bm\Upsilon_{n,1},\bm\Upsilon_{n,2})$ is a Markov Chain takes values on $\mathscr{Y}\times\mathbb R^{(q-q_1)(T-1)}$.

Denote $\widetilde{\bm T}=(\widetilde{T}^{(1)}, \ldots,\widetilde{T}^{(T)})$, $\widetilde{\bm T}_i=(\widetilde{T}_i^{(1)}, \ldots,\widetilde{T}_i^{(T)})$, $i=1,2\ldots,$  be a sequence of i.i.d. random vectors which are the results of complete randomization, i.e. $\Prob(\widetilde T^{(t)}=1)=\frac{1}{T}$, $t=1,\ldots,T$, $\bm T$ and $\widetilde{\bm T}_i$s are independent of all other random variables.  With   $\widetilde{\bm T}_i$ taking the place of $\bm T_i$, define $\widetilde{\bm \Lambda}_n$,  $\widetilde{\bm\Upsilon}_n$, $\widetilde{\bm\Upsilon}_{n,1}$ and  $\widetilde{\bm\Upsilon}_{n,2}$   the same as  $\bm\Lambda_n$,  $\bm\Upsilon_n$, $\bm\Upsilon_{n,1}$ and $\bm\Upsilon_{n,2}$, respectively.

\begin{lemma}\label{lem:density} Let
$$\widetilde{K}_{\delta}(A)= (1-\delta)\sum_{n=0}^{\infty} \Prob(\widetilde{\bm\Upsilon}_n\in A)\delta^n. $$
Suppose that Assumptions   \ref{asm:iid}, \ref{asm:recurrent} and A3) of Assumption \ref{asm:allocation}.  Then there exist  a bounded, continuous and positive density $f_{\delta}(\bm y)$ on state space $\mathscr{Y}\times\mathbb R^{(q-q_1)(T-1)}$ and a  $c_{\delta}>0$ such that
\begin{equation} \label{eq:lemdensity}\widetilde{K}_{\delta}(A)
\ge c_{\delta}\int_A f_{\delta}(\bm y) d\bm y.
\end{equation}
Here, $\int_{A}f(\bm y)d\bm y$ means that
$$\sum_{\bm y^{(1)}\in \mathscr{Y}}\int_{\bm y^{(2)}\in \mathbb R^{(q-q_1)(T-1)}}f(\bm y^{(1)},\bm y^{(2)})I\{(\bm y^{(1)},\bm y^{(2)})\in A\}d\bm y^{(2)}. $$
Furthermore, let the resolvent kernel of $\{\bm\Upsilon_n \}$ be
$$ K_{\Upsilon, \epsilon}(\bm x,A)=(1-\epsilon)\sum_{n=0}^{\infty}P_{\lambda}^n(\bm x,A)\epsilon^{n}. $$
where $0<\epsilon<1$. Then,
\begin{align} \label{eq-drift-kenerUp}K_{\Upsilon, \epsilon}(\bm x,A)\ge  \frac{1-\epsilon}{1-\delta}c_{\delta}\int_{A} f_{\delta}(\bm y-\bm x)d\bm y
\end{align}
with $\delta=\epsilon T\rho$, and, $\{\bm \Upsilon_n\}$ is a $\psi$-irreducible $T$-chain for which every bounded closed subset of $\mathscr{Y}\times\mathbb R^{(q-q_1)(T-1)}$ is petite.

\end{lemma}

\begin{proof}[Proof] In Assumption \ref{asm:recurrent}, we can assume that $\nu$ is bounded and continuous, and then $\inf_{\bm u\in O} \nu(\bm u)>0$ for an open set $O$. In fact, let $\Gamma_{\phi^{(2)}}^{\ast r}(\cdot|\bm y)$ be the conditional distribution of $\sum_{i=1}^r\bm\phi^{(2)}(\bm X_i)$ for given $\sum_{i=1}^r\bm\phi^{(1)}(\bm X_i)=\bm\ y$. Then
\begin{align*}
&\Gamma_{\phi^{(2)}}^{\ast 2n_0}(A|\bm 2\bm y_0)=\Prob\Big(\sum_{i=1}^{2n_0}\bm\phi^{(2)}(\bm X_i)\in A\big|\sum_{i=1}^{2n_0}\bm\phi^{(1)}(\bm X_i)=2\bm y_0\Big)\\
\ge &
\Prob\Big(\sum_{i=1}^{2n_0}\bm\phi^{(2)}(\bm X_i)\in A, \sum_{i=1}^{n_0}\bm\phi^{(1)}(\bm X_i)=\bm y_0,\sum_{i=n_0+1}^{2n_0}\bm\phi^{(1)}(\bm X_i)=\bm y_0\Big)\\
=&\Prob^2\Big(\sum_{i=1}^{n_0}\bm\phi^{(1)}(\bm X_i)=\bm y_0\Big) \Gamma_{\phi^{(2)}}^{\ast n_0}(\cdot|\bm  y_0)\ast\Gamma_{\phi^{(2)}}^{\ast n_0}(\cdot| \bm y_0)(A)\\
\ge & \Prob^2\Big(\sum_{i=1}^{n_0}\bm\phi^{(1)}(\bm X_i)=\bm y_0\Big)c_{\nu}^2 \int_A \nu^{ \ast 2 }(\bm u)d\bm u\ge \widetilde{c}_{\nu}^2 \int_A (\nu\wedge 1)^{\ast 2}(\bm u)d\bm u,
\end{align*}
  and $  (\nu\wedge 1)^{\ast 2}(\bm u)/\int  (\nu\wedge 1)^{\ast 2}(\bm u)d\bm u$ is a bounded and continuous density, which can take the place of $\nu$ in Assumption \ref{asm:recurrent}.

Let
$\Gamma_{n_1,\ldots,n_T}$, $\Gamma_{n_1,\ldots,n_T}^{(a)}$ be the distribution functions of $\sum_{i=1}^{n_1+\ldots+n_T} (\bm t_i-t_i^{(T)})\odot \bm\phi(\bm X_i)$, $\sum_{i=1}^{n_1+\ldots+n_T} (\bm t_i-t_i^{(T)})\odot \bm\phi^{(a)}(\bm X_i)$ respectively,   $\nu_{n_1,\ldots,n_T}^{(1)}(\bm y^{(1)})$ be the probability mass function of $\Gamma_{n_1,\ldots,n_T}^{(1)}$, and $\Gamma_{n_1,\ldots,n_T}^{(2)}(\cdot|\bm y^{(1)})$ be the conditional distribution of
 $\sum_{i=1}^{n_1+\ldots+n_T} (\bm t_i-t_i^{(T)})\odot \bm\phi^{(2)}(\bm X_i)$ for given $\sum_{i=1}^{n_1+\ldots+n_T} (\bm t_i-t_i^{(T)})\odot \bm\phi^{(1)}(\bm X_i)=\bm y^{(1)}$,  where $\bm t_i$ takes values $(1,0,\ldots,0)$, $\ldots$, $(0,\ldots,0,1)$ with $\sum_{i=1}^n \bm t_i=(n_1,\ldots,n_T)$, $a=1,2$. Then
 \begin{equation}\label{eq:distribution-convolution0}\Gamma_{n_1,\ldots,n_T}(\bm y)=\int \Gamma_{\phi}^{\ast n_1}(\bm y_1+\bm u)\cdots \Gamma_{\phi}^{\ast n_{T-1}}(\bm y_{T-1}+\bm u)  \Gamma_{\phi}^{\ast n_T}(d\,\bm u),
\end{equation}
\begin{equation}\label{eq:distribution-convolution1}\Gamma_{n_1,\ldots,n_T}^{(a)}(\bm y^{(a)})=\int \Gamma_{\phi^{(a)}}^{\ast n_1}(\bm y_1^{(a)}+\bm u)\cdots \Gamma_{\phi^{(a)}}^{\ast n_{T-1}}(\bm y_{T-1}^{(a)}+\bm u)  \Gamma_{\phi^{(a)}}^{\ast n_T}(d\,\bm u), \; a=1,2,
\end{equation}
$\Gamma_{n_1,\ldots,n_T}(\bm y)=\Gamma_{n_1,\ldots,n_T}^{(1)}(\bm y^{(1)})\times\Gamma_{n_1,\ldots,n_T}^{(2)}(\bm y^{(2)}|\bm y^{(1)})$, and
\begin{equation}\label{eq:distribution-convolution2} \Prob(\widetilde{\bm \Upsilon}_n\in A)=\frac{1}{T^n}\sum_{n_1+\ldots+n_d=n}\frac{n!}{n_1!\cdots n_T!}\int_A\Gamma_{n_1,\ldots,n_T}( d \,\bm y),
\end{equation}
where $\bm y^{(1)}\in \mathscr{Y}$,  $\bm y^{(2)}=(\bm y^{(2)}_1,\ldots,\bm y_{T-1}^{(2)})\in \mathbb R^{(q-q_1)(T-1)}$.
Denote $\widetilde{\Gamma}$ by the distribution on $\mathbb R^{(q-q_1)(T-1)}$ with density
\begin{equation}\label{eq:distribution-convolution3}\beta(\bm y^{(2)})=\int \nu(\bm y_1^{(2)}+\bm u)\cdots \nu(\bm y_{T-1}^{(2)}+\bm u)\nu(\bm u) d\bm u.
\end{equation}
Let
 $q_0=\Prob(\sum_{i=1}^{n_0}\bm\phi^{(1)}(\bm X_i)=\bm y_0)>0$ and
 $$\beta_{n_1,\ldots,n_d}(\bm y^{(2)}|\bm y^{(1)})=\beta\ast \Gamma_{n_1,\ldots,n_d}(\bm y^{(2)}|\bm y^{(1)})=\int \beta(\bm y^{(2)}-\bm u) \Gamma_{n_1,\ldots,n_d}^{(2)}( d \,\bm u|\bm y^{(1)}). $$
 Then $\beta(\bm y^{(2)})$ and $\beta_{n_1,\ldots,n_d}(\bm y^{(2)}|\bm y^{(1)})$ are bounded and continuous density.
  Note that for a non-negative function $g(\bm y^{(1)},\bm y^{(2)})$, by \eqref{eq:distribution-convolution0} and  Assumption \ref{asm:recurrent} (ii) we have
 \begin{align*}
 &\int g(\bm y)\Gamma_{n_0,\ldots,n_0}(d\bm y)
 =  \int g(\bm y)\Gamma_{\phi}^{\ast n_0}(\bm u+d\bm y_1)\cdots \Gamma_{\phi}^{\ast n_0}(\bm u+d\bm y_{T-1})  \Gamma_{\phi}^{\ast n_0}(d\,\bm u)\\
\ge &\int g(\bm 0,\bm y^{(2)}) \Gamma_{\phi}^{\ast n_0}(\bm y_0, \bm u^{(2)}+d\bm y_1^{(2)})\cdots \Gamma_{\phi}^{\ast n_0}(\bm y_0, \bm u^{(2)}+d\bm y_{T-1}^{(2)})  \Gamma_{\phi}^{\ast n_0}(\bm y_0,d\bm u^{(2)}) \\
=&q_0^T\int g(\bm 0,\bm y^{(2)}) \Gamma_{\phi^{(2)}}^{\ast n_0}(\bm u^{(2)}+ d\bm y_1^{(2)}|\bm y_0)\cdots \Gamma_{\phi^{(2)}}^{\ast n_0}(\bm u^{(2)}+d \bm y_{T-1}^{(2)}|\bm y_0)  \Gamma_{\phi^{(2)}}^{\ast n_0}(d\bm u^{(2)}|\bm y_0) \\
\ge & q_0^Tc_{\nu}^T\int g(\bm 0,\bm y^{(2)}) \nu(\bm y_1^{(2)}+\bm u^{(2)})\cdots \nu(\bm y_{T-1}^{(2)}+\bm u^{(2)})\nu(\bm u^{(2)}) d\bm u^{(2)}d\bm y^{(2)}\\
=&  q_0^Tc_{\nu}^T\int g(\bm 0,\bm y^{(2)})\beta(\bm y^{(2)})d\bm y^{(2)}.
 \end{align*}
It follows that
\begin{align*} 
\Prob(\widetilde{\bm \Upsilon}_{Tn_0+n}\in A)\ge & T^{-Tn_0}\Gamma_{n_0,\ldots,n_0}\ast \Prob(\widetilde{\bm \Upsilon}_n\in A)
\ge   T^{-Tn_0} q_0^{T}c_{\nu}^{T} (I\times \widetilde{\Gamma})\ast \Prob(\widetilde{\bm \Upsilon}_n\in A)\\
=:&c(I\times \widetilde{\Gamma})\ast \Prob(\widetilde{\bm \Upsilon}_n\in A)=c\int_{A}\mu_n(\bm y) d\bm y\\
 \text{ with }   \mu_n(\bm y)=  & \frac{1}{T^n}\sum_{n_1+\ldots+n_d=n}\frac{n!}{n_1!\cdots n_T!} \nu_{n_1,\ldots,n_d}^{(1)}(\bm y^{(1)})\cdot \beta_{n_1,\ldots,n_d}(\bm y^{(2)}|\bm y^{(1)}),
\end{align*}
where $(I\times \widetilde{\Gamma})\ast \Prob(\widetilde{\bm \Upsilon}_n\in A)=\int \Prob((\widetilde{\bm \Upsilon}_n^{(1)},\widetilde{\bm \Upsilon}_n^{(2)}+\bm u)\in A) \widetilde{\Gamma}(d\,\bm u)$.
It is easily seen that for each $\bm y^{(1)}\in \mathscr{Y}$, $\mu_n(\bm y^{(1)},\bm y^{(2)})$ is a bounded and continuous function of $\bm y^{(2)}$, and the points in $\mathscr{Y}$ are isolated by the Assumption \ref{asm:recurrent} (i). Hence $\mu_n(\bm y)$  is a bounded and continuous density on $\mathscr{Y}\times \mathbb R^{(q-q_1)(T-1)}$.
Let
$$f_{\delta}(\bm y) =(1-\delta)\sum_{n=0}^{\infty} \mu_n(\bm y)\delta^n.$$
Then $f_{\delta}(\bm y)$ is   a bounded and continuous density on $\mathscr{Y}\times \mathbb R^{(q-q_1)(T-1)}$. Furthermore,  for each $\bm y^{(1)}\in \mathscr{Y}$, there exist $n_1,\ldots, n_T$ such that $\nu_{n_1,\ldots,n_d}^{(1)}(\bm y^{(1)})>0$ and $\beta_{n_1,\ldots,n_d}(\cdot|\bm y^{(1)})\not\equiv 0$. It follows that $f_{\delta}(\bm y^{(1)},\cdot)\not\equiv 0$.
Now,
\begin{align*}
&\widetilde{K}_{\delta}(A)= (1-\delta)\sum_{n=0}^{\infty}\Prob(\widetilde{\bm\Upsilon}_n\in A)\delta^n
\ge  (1-\delta)\delta^{Tn_0}\sum_{n=0}^{\infty}\Prob(\widetilde{\bm \Upsilon}_{Tn_0+n}\in A)\delta^n \\
\ge& c_{\delta}(1-\delta)\sum_{n=0}^{\infty}(I\times \widetilde{\Gamma})\ast \Prob(\widetilde{\bm \Upsilon}_{n}\in A)\delta^n
= c_{\delta} (I\times \widetilde{\Gamma})\ast \widetilde{K}_{\delta}(A)
=  c_{\delta}\int_Af_{\delta}(\bm y)d\bm y.
\end{align*}
\eqref{eq:lemdensity} is proved. Next, we show that $f_{\delta}(\bm y)$ is positive. Repeating the about arguments yields
\begin{align*}
& \int_Af_{\delta}(\bm y)d\bm y=(I\times \widetilde{\Gamma})\ast \widetilde{K}_{\delta}(A) \ge c_{\delta}^m (I\times \widetilde{\Gamma})^{\ast(m+1)}\ast \widetilde{K}_{\delta}(A)\\
\ge & c_{\delta}^m (I\times \widetilde{\Gamma})^{\ast m}\ast \int_Af_{\delta}(\bm y)d\bm y=c_{\delta}^m  \int_A(I\times \beta^{\ast m})\ast f_{\delta}(\bm y)d\bm y,
\end{align*}
where
$$(I\times \beta^{\ast m})\ast f_{\delta}(\bm y)=\int \beta^{\ast m}(  \bm y^{(2)}-\bm u)f_{\delta}(\bm y^{(1)}, \bm u)d\bm u. $$
Note that the densities are continuous. We have $f_{\delta}(\bm y)\ge c_{\delta}^m   (I\times \beta^{\ast m})\ast f_{\delta}(\bm y)$. By the fact that $\beta(\bm y^{(2)})$ is a bounded continuous function with $\beta(\bm 0)=\int \nu^T(\bm u)d\bm u>0$, it follows that  $\beta(\bm u)$ is  positive in a neighborhood of $\bm 0$.
 Thus, for any bounded area, there is an $m$ such that $\beta^{\ast m}(\bm u)$ is positive in this area, which implies that for any given $\bm y^{(1)},  \bm y^{(2)}$, there exists $m$ such that $(I\times \beta^{\ast m})\ast f_{\delta}(\bm y)>0$ due to the fact that $f_{\delta}(\bm y^{(1)}, \cdot)\not \equiv 0$. Thus, $f_{\delta}(\bm y)>0$.

 Finally, we verify \eqref{eq-drift-kenerUp}.  Note that by A3) of Assumption \ref{asm:allocation}  for the allocation functions,  there exists a constant $0<\rho<1/T$ such that $\ell_t(\bm x)\ge \rho$, $t=1,\ldots,T$. So, for each $\bm t_i=(1,0,\ldots,0),\ldots, (0,\ldots,0,1)$,
$$ \Prob\big(\bm\Upsilon_n\in A|\bm\Upsilon_{n-1},\bm\phi(\bm X_n)\big)\ge \rho\Prob\big(\bm \Upsilon_{n-1}+(\bm t_i-t_i^{(T)})\odot\bm\phi\big(\bm X_n)\in A|\bm \Lambda_{n-1},\bm\phi(\bm X_n)\big).$$
Hence
$$ \Prob(\bm\Upsilon_n\in A|\bm\Upsilon_{n-1})\ge \rho\Prob(\bm \Upsilon_{n-1}+(\bm t_i-t_i^{(T)})\odot\bm\phi(\bm X_n)\in A|\bm \Upsilon_{n-1}).
$$
It follows that
\begin{equation} \label{eq:nultilowerboundofP} \Prob(\bm\Upsilon_n\in A|\bm\Upsilon_{n-1})\ge (T\rho)\Prob(\bm \Upsilon_{n-1}+(\widetilde{\bm T}_n-\widetilde{T}_i^{(T-1)})\odot \bm\phi(\bm X_n)\in A|\bm \Upsilon_{n-1}).
\end{equation}
By the induction, we have
\begin{equation} \label{eq:multilowerboundofP2}
P_{\Upsilon}^{n}(\bm x,A)=:\Prob(\bm\Upsilon_n\in A|\bm\Upsilon_0=\bm x)\ge  (T\rho)^{n} \Prob(\widetilde{\bm\Upsilon}_n\in A-\bm x).
\end{equation}
Then,  by \eqref{eq:multilowerboundofP2},
\begin{align} \label{eq-drift-kenerUp2}K_{\Upsilon, \epsilon}(\bm x,A)\ge \frac{1-\epsilon}{1-\delta}\widetilde{K}_{\delta}(A-\bm x),
\end{align}
where $\delta=\epsilon T\rho$. By \eqref{eq:lemdensity}, \eqref{eq-drift-kenerUp} is proved.

 Under \eqref{eq-drift-kenerUp}, $\{\bm \Upsilon_n\}$ is a $\psi$-irreducible $T$-chain for which every compact set is  a petite set. Notice that every bounded set contains only finite number of points in $\mathscr{Y}$ by Assumption \ref{asm:recurrent}. Hence, every bounded closed subset of $\mathscr{Y}\times\mathbb R^{(q-q_1)(T-1)}$ is compact. The proof is completed.
\end{proof}

 The following lemma is   Theorem 1 of  \cite{PemantleRosenthal1999}.

\begin{lemma}\label{lem:PR1999} \citep{PemantleRosenthal1999}
Let $\{X_n;n\ge 0\}$  be random variables and the filtration $\{\mathscr{F}_n;n\ge 0\}$ of sigma fields refer to any filtration to which  $\{X_n;n\ge 0\}$  is adapted.   Suppose that there exist constants $a > 0$, $J, V < \infty$, and $p > 2$, such that $X_0 \leqslant J$, and for all $n$,
$$
\mathbf{E}(X_{n+1} - X_n \mid \mathscr{F}_n) \leqslant -a \quad \text{on the event } \{X_n > J\} $$
and
$$
\mathbf{E}(|X_{n+1} - X_n|^p \mid X_0, \dots, X_n) \leqslant V.  
$$
Then for any $r \in (0, p-1)$ there is a $c = c(p, a, V, J, r) > 0$ such that $\mathbf{E}(X_n^+)^r < c$ for all $n$.
\end{lemma}

\begin{proof}[Proof of Theorem \ref{thm:multiCAR}]  
By \eqref{eq:lem:drift1} (or equivalently,\eqref{eq:multi-drift-p-2}), it follows that
 \begin{equation}\label{eq:multi-drift-L1}
  \Ex\left[\|\bm\Lambda_{n+1}\| \big |\mathscr{F}_n\right]-\|\bm\Lambda_n\| 
  \le   -c  
  +b \mathbb{I}\{\|\bm \Lambda_n\|\le d\}
\end{equation} 
for some positive constants $c, b$ and $d$.
On the other hand, $\left|\|\bm \Lambda_{n+1}\|-\|\bm\Lambda_n\|\right|\le \|\bm\psi_{n+1}\|\le \|\bm \phi(\bm X_{n+1})\|$. Hence, 
$$ \Ex\left[\left|\|\bm \Lambda_{n+1}\|-\|\bm\Lambda_n\|\right|^{\gamma}\right]\le \Ex\left[\|\bm \phi(\bm X)\|^{\gamma}\right]. $$ Notice $\gamma>2$. By Lemma \ref{lem:PR1999}, $\Ex[\|\bm\Lambda_n\|^r]=O(1)$ for all $r\in (0,\gamma-1)$.

 For the a.s. convergence, we let $1<\beta<2$. Choose a constant $C_0>0$ such that $(\beta-1)C_0>2 b_{\gamma}$. By \eqref{eq:lem:drift1},
 \begin{align*} 
 &\Ex\left[\frac{\|\Lambda_{n+1}\|^{\gamma}+C_0(n+1)}{(n+1)^{\beta}}\Big|\mathscr{F}_{n}\right]\le \frac{\|\Lambda_n\|^{\gamma}+b_{\gamma}+C_0(n+1)}{(n+1)^{\beta}} \\
 \le &\frac{\|\Lambda_n\|^{\gamma}+C_0n}{n^{\beta}}+\frac{b_{\gamma}+C_0 +C_0 n\big(1-(1+1/n)^{\beta}\big)}{(n+1)^{\beta}} \\
 \le & \frac{\|\Lambda_n\|^{\gamma}+C_0n}{n^{\beta}}+\frac{b_{\gamma}+C_0 +C_0 n\big(-\beta/n+1/n^2\big)}{(n+1)^{\beta}}\\
 \le & \frac{\|\Lambda_n\|^{\gamma}+C_0n}{n^{\beta}}+\frac{-b_{\gamma}+C_0/n}{(n+1)^{\beta}}\le \frac{\|\Lambda_n\|^{\gamma}+C_0n}{n^{\beta}} \text{ when } n\ge C_0/b_{\gamma}.
 \end{align*}
It following that $\left\{\frac{\|\Lambda_n\|^{\gamma}+C_0n}{n^{\beta}}\right\}$ is a positive super-martingale and so it is almost surely convergent. Hence $\|\Lambda_n\|=O(n^{\beta/\gamma})$ a.s. for any $\beta>1$. The proof is completed.
\end{proof}

\begin{proof}[Proof of Theorem \ref{thm:multiCARrecurrent}] We want to show that $\{\bm \Lambda_n\}$  is a positive Harris recurrent Markov chain and $\max_n\Ex[\|\bm\Lambda_n\|^{\gamma-1}]<\infty$.

By lemma \ref{lem:density}, $\{\bm\Upsilon_n\}$ is an irreducible $T$-Chain and every bounded closed subset of $\mathscr{Y}\times\mathbb R^{(q-q_1)(T-1)}$ is petite. On the other hand,
  by letting  $V(\bm\Upsilon_n)=\|\bm \Lambda_n\|$ and \eqref{eq:lem:drift1}, the Markov Chain $\{\bm\Upsilon_n\}$ satisfies the drift-condition:
\begin{align}\label{eq:driftforUpsilon}
 P_{\Upsilon}\left[V^{\gamma}(\bm\Upsilon)\right]-V^{\gamma}(\bm\Upsilon)
  \le -1 -  c V^{\gamma-1}(\bm\Upsilon)
  +b\mathbb{I}\{V(\bm\Upsilon)\le d\},
\end{align}
where $P_{\Upsilon}\left[f(\bm\Upsilon)\right]=\int f(\bm y) P_{\Upsilon}(\bm \Upsilon,d\bm y)$ and $P_{\Upsilon}(\bm x,A)=\Prob(\bm \Upsilon_n\in A|\bm\Upsilon_{n-1}=\bm x)$ is the transition probability. Notice that $V(\bm \Upsilon)$ is a norm-like function of $\bm\Upsilon$ and $\{V(\bm\Upsilon)\le d\}$ is a bounded closed set of $\bm\Upsilon$. By Theorems 11.3.4 of \cite{Meyn2009},  $\{\bm\Upsilon_n\}$ (and equivalently, $\{\bm \Lambda_n\}$) is a positive Harris recurrent Markov chain with an      invariance probability measure $\pi_{\Upsilon}$ (resp. $\pi_{\lambda}$).
By Theorem  14.2.3 (i), Proposition 14.1.1 and Theorem  14.3.6 of \cite{Meyn2009},  $\pi_{\lambda}[\|\bm\Lambda\|^{\gamma-1}]= \pi_{\Upsilon}[V^{\gamma-1}(\bm\Upsilon)]]<\infty$ and  $\max_n\Ex[\|\bm\Lambda_n\|^{\gamma-1}]=\max_n\Ex[V^{\gamma-1}(\bm\Upsilon_n)]]<\infty$.   The proof is completed.
\end{proof}

\begin{proof}[Proof of Theorem \ref{thm:multiclt}] We will complete the proof via four steps.

 {\sl Step 1.} We show
\begin{equation}\label{eq:proofofclt1}
\begin{aligned}
&\frac{\sum_{i=1}^n \left\{(T\bm T_i-1)\alpha Z_i+\beta \bm W_i\right\}}{\sqrt{n}}\overset{D}\to N\big(\bm 0,\alpha^2\vec{\bm\Sigma}_Z+\beta^2\Var(\bm W)\big) \text{ and } \\
& \frac{\Ex\left[\left(\sum_{i=1}^n \left\{(T\bm T_i-1)\alpha Z_i+\beta \bm W_i\right\}\right)^{\otimes 2}\right]}{n} \to  \alpha^2\vec{\bm\Sigma}_Z+\beta^2\Var(\bm W),
\end{aligned}
\end{equation}
where
$$\vec{\bm\Sigma}_Z=\vec{\sigma}_Z^2\Big(T\bm I-(1,\ldots,1)(1,\ldots,1)^{\prime}\Big)
$$
 and  $\vec{\sigma}_Z^2=\Ex(Z-\Ex[Z|\bm X])^2+\vec{\sigma}_m^2$.
 
 \eqref{eq:multiclt2}  follows  from \eqref{eq:proofofclt1} immediately. 

Without loss of generality, we assume $\alpha=\beta=1$.   Recall $m(\bm X)=\Ex[Z|\bm X]$. Let $\widetilde{\mathscr F}_i=\sigma(\mathscr F_i, Z_1,\bm W_1,\ldots,Z_i,\bm W_i)$,   $ \eta_i^{(t)}=(TT_i^{(t)}-1)Z_i+W_i^{(t)}$. Then
	\begin{equation*}
	\Ex( \eta_i^{(t)}|\widetilde{\mathscr{F}}_{i-1})=\Ex[(TT_i^{(t)}-1)m( \bm X_i)|\widetilde{\mathscr{F}}_{i-1}]=\Ex[(TT_i^{(t)}-1)m( \bm X_i)|\bm\Lambda_{i-1}]=:g_t(\bm\Lambda_{i-1}).
	\end{equation*}
Notice that $g_t(\bm\Lambda_{i-1})$ is determined by $m(\bm X_i)$.  It is obvious that $|g_t(\bm \Lambda)|\le (T-1)\Ex[|m(\bm X)|]$ is bounded. By Theorem \ref{thm:multiCARrecurrent},  $\{\bm \Upsilon_n\}$,  and equivalently, $\{\bm\Lambda_n\}$ is  a positive Harris recurrent Markov chain. Let  $\pi$ (resp. $\pi_{\Upsilon}$) be the   invariance probability measure of $\{\bm\Lambda_n\}$ (resp. $\{\bm\Upsilon_n\}$).
Consider the Poisson equation
\begin{equation}\label{eq:Poissonmulti}
\widehat{g}(\bm\Lambda)-P_{\lambda}[\widehat{g}](\bm\Lambda)=g(\bm\Lambda)-\pi[g].
\end{equation}
 with $g(\bm\Lambda)=\frac{g_t(\bm\Lambda)}{(T-1)\Ex[|m(\bm X)|]}$, where and $P_{\lambda}(\bm x,A)$ is the transition probability of $\bm\Lambda$. First,  applying  \eqref{eq:lem:drift1} with $\gamma_0=1$ yields  that
 $  P_{\lambda}\left[\|\bm \Lambda\|\right]-\|\bm\Lambda\|
  \le   -c  +b \mathbb{I}\{\|\bm \Lambda\|\le d \}.
 $
 That is
 \begin{align}\label{eq:multi-drift-recurrent-1}
 P_{\lambda}\left[\frac{\|\bm\Lambda\|}{c}\right]-\frac{\|\bm\Lambda\|}{c}
  \le    -1
  +\frac{b}{c}\mathbb{I}\{\|\bm \Lambda\|\le d\}.
\end{align}
Note $|g(\bm \Lambda)|\le 1$. By \eqref{eq:multi-drift-recurrent-1}  and Theorem 17.7.1 of \cite{Meyn2009} with $V=\|\bm\Lambda\|/c$ and $f=1$, the Poisson equation \eqref{eq:multi-drift-recurrent-1} admits a solution satisfying the bound $|\widehat{g}(\bm\Lambda)|\le R(\|\bm\Lambda\|/c+1)$, where the positive constant $R$ depends only on the Markov Chain $\bm \Lambda$.
In the proof of Theorem \ref{thm:multiCARrecurrent}, it is shown that  $\pi[\|\bm \Lambda\|^{\gamma -1}]<\infty$ and   $\max_n\Ex[\|\bm\Lambda_n\|^{\gamma -1}]<\infty$. So, $\pi[|\widehat{g}(\bm\Lambda)|^{\gamma -1}]<\infty$ and $\max_n\Ex[|\widehat{g}(\bm\Lambda_n)|^{\gamma -1}]<\infty$. Since $\widehat{g}+C$ is also a solution, without loss of generality, we assume $\widehat{g}(\bm 0)=0$. Denote  $\widehat{g}_t(\bm \Lambda)=(T-1)\Ex[|m(\bm X)|]\widehat{g}(\bm\Lambda)$ with $\widehat{g}(\bm 0)=0$. Then $\widehat{g}_t(\bm\Lambda)$ is a solution of the equation \eqref{eq:Poissonmulti} with $g(\bm \Lambda)=g_t(\bm\Lambda)$, $\widehat{g}_t(\bm 0)=0$ and
\begin{equation}\label{eq:estofghat}|\widehat{g}_t(\bm \Lambda)|\le R(T-1)\Ex[|m(\bm X)|]\big(\|\bm \Lambda\|/c+1\big).
\end{equation}
Let $\Delta M_i^{(t)}=\eta_i^{(t)}-\pi[g_t(\bm \Lambda)]+\widehat{g}_t(\bm\Lambda_i)-\widehat{g}_t(\bm\Lambda_{i-1})$. Then $\sup_n|\Ex[|\widehat{g}_t (\bm\Lambda_{n})|^2]\le C\sup_n(\Ex[\|\bm \Lambda_n\|^2]+1)<\infty$,
\begin{equation}\label{eq:martingaleRep}\sum_{i=1}^n\left[ (TT_i^{(t)}-1)Z_i+W_i^{(t)}-\pi[g_t(\bm \Lambda)]\right]=M_n^{(t)}+\widehat{g}_t(\bm\Lambda_{0})- \widehat{g}_t (\bm\Lambda_{n}),
\end{equation}
and  $\{(\Delta M_n^{(1)},\ldots,\Delta M_n^{(T)}); n=1,2,\ldots\}$ is a sequence of martingale difference vectors with
\begin{equation}\label{eq:conditionalCovar1}\sigma_{t,s}(\bm\Lambda_{i-1})=\Ex[\Delta M_i^{(t)}\Delta M_i^{(s)} |\bm\Lambda_{i-1}]=\Ex[\Delta M_i^{(t)}\Delta M_i^{(s)} |\widetilde{\mathscr{F}}_{i-1}].
\end{equation}
It is easily seen that   $|\sigma_{t,s}( \bm y)|\le C(\|\bm y\|^2+1)$.   By the ergodic theorem (c.f. Theorems 17.1.7 and 17.1.6 of  \cite{Meyn2009}),
\begin{equation}\label{eq:conditionalCovar3}
\frac{\sum_{i=1}^n\Ex[\Delta M_i^{(t)}\Delta M_i^{(s)} |\widetilde{\mathscr{F}}_{i-1}]}{n}=\frac{\sum_{i=0}^{n-1}\sigma_{t,s}(\bm\Lambda_{i})}{n}
\overset{a.s.}\to  \pi[\sigma_{t,s}(\bm\Lambda)].
\end{equation}
Furthermore, let $K$ be a constant. Then $\Ex[|\Delta M_i^{(t)}|^2I\{|\Delta M_i^{(t)}|\ge   K \}|\widetilde{\mathscr{F}}_{i-1}]=\sigma_K^2(\bm\Lambda_{i-1})$ is a function of $\bm \Lambda_{i-1}$. It is obvious that
$\pi[\sigma_K^2(\bm\Lambda_{i-1})]\le \Ex_{\pi}[|\Delta M_i^{(t)}|^2]<\infty$. For any $\epsilon>0$, 
by the ergodic theorem (c.f. Theorems 17.1.7 of  \cite{Meyn2009}),
\begin{align}\label{eq:Lindeberg}
    \limsup_{n\to\infty} & \frac{\sum_{i=1}^n \Ex[|\Delta M_i^{(t)}|^2I\{|\Delta M_i^{(t)}|\ge  \epsilon i^{1/2} \}|\widetilde{\mathscr{F}}_{i-1}]}{n}\nonumber\\
\le & \lim_{n\to\infty} \frac{\sum_{i=1}^n \Ex[|\Delta M_i^{(t)}|^2I\{|\Delta M_i^{(t)}|\ge  K \}|\widetilde{\mathscr{F}}_{i-1}]}{n}\nonumber \\
= & \lim_{n\to\infty} \frac{\sum_{i=1}^n \sigma_K^2(\bm\Lambda_{i-1}) }{n}\overset{a.s.}=\pi[  \sigma_K^2(\bm\Lambda_1)] \;\;a.s. 
\nonumber\\
=& \Ex_{\pi}[|\Delta M_1^{(t)}|^2I\{|\Delta M_1^{(t)}|\ge   K \}]\to 0 \text{ as } K\to \infty.
\end{align}
The conditional Linderberg condition is satisfied.
By the  central limit theorem for martingales (c.f. Corollary 3.1 of \cite{Hall1980}),  we have
$$\frac{\sum_{i=1}^n  [(T\bm T_i-1)Z_i+\bm W_i-\pi[g_t(\bm \Lambda)]]}{\sqrt{n}} \overset{D}\to N(\bm 0,\bm\Sigma_{\eta}) $$
with
$ \bm\Sigma_{\eta}=\Big(\pi[\sigma_{t,s}(\bm\Lambda)]; t,s=1,\ldots,T\Big). $
Furthermore,
\begin{equation}\label{eq:convergence-of-varaincemulti} \frac{\Ex\big[(\sum_{i=1}^n  ((T\bm T_i-1)Z_i+\bm W_i-\pi[g_t(\bm \Lambda)]))^{\otimes 2}\big]}{n}=\frac{\sum_{i=1}^n  \Ex\big[(\Delta\bm M_i)^{\otimes 2}\big]}{n}+o(1)\to  \bm\Sigma_{\eta},
\end{equation}
For \eqref{eq:proofofclt1},  it is remained  to check  the variance-covariance matrix and
\begin{equation}\label{eq:zeromean} \pi[g_t(\bm \Lambda)]=0. \end{equation}
To begin, we show that
 for any function $f_i=f(\bm X_i,Z_i, \bm W_i)$,
\begin{equation}\label{eq:zeromean2} \Ex_{\pi}[(T_i^{(t)}-\frac{1}{T})f(\bm X_i,Z_i,\bm W_i)] =\Ex_{\pi}[f(\bm X_i,Z_i,\bm W_i)\widehat{g}_t(\bm\Lambda_i)]=0
\end{equation}
whenever the expectation is finite. Here and in the sequel, $\Ex_{\mu}[\cdot]=\int \Ex[\cdot|\bm\Lambda_0=\bm x]\mu(d\bm x)$ and  $\Prob_{\mu}[\cdot]=\int \Prob[\cdot|\bm\Lambda_0=\bm x]\mu(d\bm x)$ denote the expectation and the probability, respectively, when the Markov Chain $\{\bm\Lambda_n\}$ has the initial distribution $\mu$.   By   A4) of Assumption \ref{asm:allocation} for the allocation functions $\ell_k(\bm x)$s,  it is easily seen that the transition probabilities of
 $$\Pi\bm\Lambda_n\widehat{=}\Big(\sum_{i=1}^n(T_i^{(\Pi(1))}-\frac{1}{T})\bm\phi(\bm X_i),\ldots, \sum_{i=1}^n(T_i^{(\Pi(T))}-\frac{1}{T})\bm\phi(\bm X_i)\Big)$$
 are the same  for any permutation $\Pi=(\Pi(1),\ldots,\Pi(T))$ of treatments $(1,\ldots,T)$.  It follows  that the invariant probability measure $\pi$ of $\{\bm\Lambda_n\}$ is symmetric about treatments. Hence, under $\pi$,
 $(\Pi\bm\Lambda_{i-1},\Pi\bm\Lambda_i, \Pi\bm T_i, Z_i,\bm W_i,\bm X_i)$ and $(\bm\Lambda_{i-1}, \bm\Lambda_i, \bm T_i, Z_i,\bm W_i,\bm X_i)$ have the same distribution for  any permutation $\Pi$.  Thus,
 $\Ex_{\pi}[(T_i^{(\Pi(t))}-\frac{1}{T})f_i]=\Ex_{\pi}[(T_i^{(t)}-\frac{1}{T})f_i]$
 and $\Ex_{\pi}[f_i\widehat{g}_t(\bm\Lambda_i)]=\Ex_{\pi}[f_i\widehat{g}_{\Pi(t)}(\Pi\bm\Lambda_i)]
 =\Ex_{\pi}[f_i\widehat{g}_{\Pi(t)}(\bm\Lambda_i)]$.  It follows that $\Ex_{\pi}[(T_i^{(s)}-\frac{1}{T})f_i]=\Ex_{\pi}[(T_i^{(t)}-\frac{1}{T})f_i]$
 and $\Ex_{\pi}[f_i\widehat{g}_s(\bm\Lambda_i)]=\Ex_{\pi}[f_i\widehat{g}_t(\bm\Lambda_i)]$ for all $t,s$. Note that $\sum_{s=1}^T(T_i^{(s)}-\frac{1}{T})=0$, and $\sum_{s=1}^T\widehat{g}_s(\bm\Lambda_i)$ is the solution of
 the Poisson equation \eqref{eq:Poissonmulti} with $g(\bm\Lambda)=\sum_{s=1}^Tg_s(\bm\Lambda)\equiv 0$ which implies that $\sum_{s=1}^T\widehat{g}_s(\bm\Lambda_i)\equiv 0$. It follows that
 $$ \Ex_{\pi}[(T_i^{(t)}-\frac{1}{T}f_i)]=\frac{1}{T}\sum_{s=1}^T\Ex_{\pi}[(T_i^{(s)}-\frac{1}{T})f_i]=0, $$
 $$ \Ex_{\pi}[f_i\widehat{g}_t(\bm\Lambda_i)]=\frac{1}{T}\sum_{s=1}^T\Ex_{\pi}[f_i\widehat{g}_s(\bm\Lambda_i)]=0. $$
 \eqref{eq:zeromean2} is proved, and so  \eqref{eq:zeromean} holds. 
  
  Let $\eta_{i,0}^{(t)}=(TT_i^{(t)}-1)Z_i+\widehat{g}_t(\bm\Lambda_i)-\widehat{g}_t(\bm\Lambda_{i-1})$. Then
 $\Delta M_i^{(t)}=\eta_i^{(t)}+\widehat{g}_t(\bm\Lambda_i)-\widehat{g}_t(\bm\Lambda_{i-1})=\eta_{i,0}^{(t)}+W_i^{(t)}$. Note
 $\Ex_{\pi}[\widehat{g}_t(\bm\Lambda_{i-1})W_i^{(s)}]=\Ex_{\pi}[\widehat{g}_t(\bm\Lambda_{i-1})\Ex[W_i^{(s)}|\widetilde{\mathscr{F}}_{i-1}]]=0$. It follows that
  $\Ex_{\pi}[\eta_{i,0}^{(t)}W_i^{(s)}]=\Ex_{\pi}\big[(TT_i^{(t)}-1)Z_iW_i^{(s)}+\widehat{g}_t(\bm\Lambda_i)W_i^{(s)}\big]=0$ by \eqref{eq:zeromean2}. Hence,
  $\pi[\sigma_{t,s}(\bm\Lambda_{i-1})]=\Ex_{\pi}[\eta_{i,0}^{(t)}\eta_{i,0}^{(s)}]+\Ex[W^{(t)}W^{(s)}]$. Thus, it is sufficient to derive $\bm\Sigma_{\eta}$ for the special case $\bm W=\bm 0$ in which $\pi[\sigma_{t,s}(\bm\Lambda_{i-1})]=\Ex_{\pi}[\eta_{i,0}^{(t)}\eta_{i,0}^{(s)}]$, and so $\vec{\bm\Sigma}_Z=\big(\Ex_{\pi}[\eta_{i,0}^{(t)}\eta_{i,0}^{(s)}];t,s=1,\ldots,T\big)$.

  By the symmetry about   treatments as we have shown,  $\Ex_{\pi}[(\eta_{i,0}^{(t)})^2]=\alpha$ and $\Ex_{\pi}[\eta_{i,0}^{(t)}\eta_{i,0}^{(s)}]=\beta$ for all $t\ne s$. On the other hand, $\sum_{t=1}^T\eta_{i,0}^{(t)}=0$. So,
  $0=\Ex_{\pi}[(\sum_{i=1}^n\eta_{i,0}^{(t)})^2]=T\alpha+T(T-1)\beta=0$. Thus, $\beta=-\alpha/(T-1)$. For $t=s$, since
 \begin{align*}
  \Ex[(TT_i^{(t)}-1)Z_i|\widetilde{\mathscr{F}}_{i-1}]=&\Ex[(TT_i^{(t)}-1)m(\bm X_i)|\bm\Lambda_{i-1}]\\
 = g_t(\bm\Lambda_{i-1})
  =&\Ex[\widehat{g}_t(\bm \Lambda_i)|\bm\Lambda_{i-1}]-\widehat{g}_t(\bm \Lambda_{i-1}),
  \end{align*}
  we have
\begin{align*}
 \Ex[(\eta_{i,0}^{(t)})^2|\widetilde{\mathscr{F}}_{i-1}]=&
  \Ex[(TT_i^{(t)}-1)^2Z_i^2|\widetilde{\mathscr{F}}_{i-1}]+2h_t(\bm\Lambda_{i-1}) +\Ex[\widehat{g}_t^2(\bm \Lambda_i)|\bm\Lambda_{i-1}]-\widehat{g}_t^2(\bm \Lambda_{i-1})\\
=&
  \Ex[(T-1)Z^2] +2h_t(\bm\Lambda_{i-1}) +\Ex[\widehat{g}_t^2(\bm \Lambda_i)|\bm\Lambda_{i-1}]-\widehat{g}_t^2(\bm \Lambda_{i-1})\\
& +\Ex[(T-2)(TT_i^{(t)}-1)Z_i^2|\bm\Lambda_{i-1}],
\end{align*}
where
$$h_t(\bm\Lambda_{i-1})=\Ex[(TT_i^{(t)}-1)m(\bm X_i)\widehat{g}_t(\bm\Lambda_i)|\bm\Lambda_{i-1}]=\Ex[(TT_i^{(t)}-1)\Ex[m(\bm X_i)|\bm\phi(\bm X_i)]\widehat{g}_t(\bm\Lambda_i)|\bm\Lambda_{i-1}].$$
  It follows that $\alpha=\Ex_{\pi}[(\eta_{i,0}^{(t)})^2]=(T-1)\Ex[Z^2]+2\pi[h_t(\bm\Lambda)]=:(T-1)\vec{\sigma}_Z^2$ by \eqref{eq:zeromean2}, and $\pi[h_t(\bm\Lambda)]$ does not depend on $t$. Thus, \eqref{eq:multivarianceZ} holds with $\vec{\sigma}_Z^2=\Ex[Z^2]+\frac{2}{T-1}\pi[h_t(\bm\Lambda)]$.

Finally, letting $Z=\Ex[Z|\bm X]=m(\bm X)$ yields $\vec{\sigma}_m^2=\Ex[m^2(\bm X)]+\frac{2}{T-1}\pi[h_t(\bm\Lambda)]$. Hence, $\vec{\sigma}_Z^2=\Ex[(Z-\Ex(Z|\bm X))^2]+\Ex[m^2(\bm X)]+\frac{2}{T-1}\pi[h_t(\bm\Lambda)]=\Ex[(Z-\Ex(Z|\bm X))^2]+\vec{\sigma}_m^2$. The proof of \eqref{eq:proofofclt1} and \eqref{eq:multiclt2} are now competed.

{\sl Step 2.} We show
$$\limsup_{n\to \infty}\frac{|\sum_{i=1}^n\{ (TT_i^{(t)}-1)Z_i+W_i^{(t)}\}|}{\sqrt{2n\log\log n}}= \widetilde{\sigma}_t  \;\; a.s.
$$
where $\widetilde{\sigma}_t^2=\pi[\sigma_{t,t}(\bm\Lambda)]=(T-1)\vec{\sigma}_Z^2+\Var(W^{(t)})$. By the same argument of showing Propositions 17.1.6 and Proposition 17.1.7 of \cite{Meyn2009}, without loss of generality we assume that the initial distribution of $\bm\Lambda_n$ is the invariant distribution $\pi$.  Notice that
$|\widehat{g}_t(\bm\Lambda_n)|\le C(\|\bm \Lambda_n\|+1)=o(n^{\frac{1}{3}+\epsilon})=o(\sqrt{n})$ a.s. by \eqref{eq:proofofmultiCAR1}. By \eqref{eq:martingaleRep}, it is sufficient to show that
 \begin{equation}\label{eq:proofofLIL1}\limsup_{n\to \infty}\frac{|M_n^{(t)}|}{\sqrt{2n\log\log n}}= \widetilde{\sigma}_t  \;\; a.s. \Prob_{\pi}.
\end{equation}
Notice that $\Ex_{\pi}[|M_n^{(t)}|^2]=n\widetilde{\sigma}_t^2$. If $\widetilde{\sigma}_t=0$, then \eqref{eq:proofofLIL1} is obvious since $\Prob_{\pi}(M_n^{(t)}=0)=1$. When $\widetilde{\sigma}_t>0$, it is sufficient to verify the conditions (4.35)-(4.37) with $W_n^2=n\widetilde{\sigma}_t^2$ and $Z_i=\sqrt{i} $ in Theorem 4.7 of \cite{Hall1980}. Notice \eqref{eq:conditionalCovar3} and \eqref{eq:Lindeberg}. It is sufficient to show that
 \begin{equation}
 \frac{\sum_{i=1}^n \{\Delta M_i^{(t)}I\{|\Delta M_i^{(t)}|\ge i^{1/2}\}-\Ex[\Delta M_i^{(t)}I\{|\Delta M_i^{(t)}|\ge i^{1/2}\}|\widetilde{\mathscr{F}}_{i-1}]\}}{\sqrt{n}}\to 0 \;\; a.s. \Prob_{\pi},
  \label{eq:proofofLIL2}
 \end{equation}
   \begin{equation}  \sum_{i=1}^{\infty} i^{-2} \Ex[|\Delta M_i^{(t)}|^4I\{|\Delta M_i^{(t)}|\le  i^{1/2} \}|\widetilde{\mathscr{F}}_{i-1}]<\infty \;\; a.s. \Prob_{\pi}.
   \label{eq:proofofLIL4}
 \end{equation}
\eqref{eq:proofofLIL2} is implied by
\begin{align*}
& \Ex_{\pi}\left[\sum_{i=1}^{\infty}i^{-1/2} \Ex[ |\Delta M_i^{(t)}|I\{|\Delta M_i^{(t)}|\ge i^{1/2}\}|\widetilde{\mathscr{F}}_{i-1}]\right]\\
=&
 \sum_{i=1}^{\infty}i^{-1/2} \Ex_{\pi}[ |\Delta M_i^{(t)}|I\{|\Delta M_i^{(t)}|\ge i^{1/2}\}]\\
=& \sum_{i=1}^{\infty}i^{-1/2} \Ex_{\pi}[ |\Delta M_1^{(t)}|I\{|\Delta M_i^{(t)}|\ge i^{1/2}\}]\le 2\Ex_{\pi}[|\Delta M_1^{(t)}|^2]<\infty.
\end{align*}
\eqref{eq:proofofLIL4} is implied by
\begin{align*}
  \Ex_{\pi} &\left[\sum_{i=1}^{\infty}i^{-2} \Ex[ |\Delta M_i^{(t)}|^4I\{|\Delta M_i^{(t)}|\le i^{1/2}\}|\widetilde{\mathscr{F}}_{i-1}]\right] 
= 
 \sum_{i=1}^{\infty}i^{-2} \Ex_{\pi}[ |\Delta M_i^{(t)}|^4I\{|\Delta M_i^{(t)}|\le i^{1/2}\}]\\
&=  \sum_{i=1}^{\infty}i^{-2} \Ex_{\pi}[ |\Delta M_1^{(t)}|^4I\{|\Delta M_1^{(t)}|\le i^{1/2}\}] 
\le   4 \Ex_{\pi}[|\Delta M_1^{(t)}|^2]<\infty.
\end{align*}

Hence,   \eqref{eq:proofofLIL1} holds.

{\sl Step 3.} We show $\vec{\sigma}_Z^2=0$ if     $Z=\langle\bm x_0,\bm\phi(\bm X)\rangle$ a.s.

 Notice $\sum_{i=1}^n(TT_i^{(t)}-1)Z_i=\langle\bm x_0,\sum_{i=1}^n(TT_i^{(t)}-1)\bm\phi(\bm X_i)\rangle=O_P(1)$ by  Theorem \ref{thm:recurrent},
which, together with \eqref{eq:proofofclt1}, yields $\vec{\bm\Sigma}_Z=\bm 0$, and hence $\vec{\sigma}_Z^2=0$.

{\sl Step 4.} We show  $Z= \langle\bm x_0,\bm\phi(\bm X)\rangle$ a.s. if $\vec{\sigma}_Z^2=0$.

This is the most difficult part of the proof. Without loss of generality, we assume $t\le T-1$. Assume $\vec{\sigma}_Z^2=0$. Then $\Ex(Z-\Ex[Z|\bm X])^2+\vec{\sigma}_m^2=0$. It follows that $Z=\Ex[Z|\bm X]=m(\bm X)$ a.s. and $\vec{\sigma}_m^2=\Ex[m^2(\bm X)]+\frac{2}{T-1}\pi[h_t(\bm \Lambda)]=0$. Hence, it is sufficient to consider the special case that $Z=m(\bm X)$. 
The proof is carried out by four steps. 

(a) We show that there exists a continuous function $\bm H$ on $\mathscr{Y}\times\mathbb{R}^{(q-q_1)(T-1)}$ such that
\begin{equation} \label{eq:findaH} \sum_{i=1}^n (T\widetilde{T}_i^{(t)}-1) m(\bm X_i)= H\big(\widetilde{\bm\Upsilon}_n\big) \; a.s.  
\end{equation}
and
 \begin{equation}\label{eq:boundnessofH1} |H(\bm x)|\le C(\|\bm x\|+1).
  \end{equation}

Let    $W^{(t)}=0$ and  $ \eta_i^{(t)}=(TT_i^{(t)}-1)m(\bm X_i)$ in \eqref{eq:martingaleRep}.   Then
\begin{align}
\label{eq:proof-multi-linear-1}
&\Ex_{\pi}\Big[\big(\sum_{i=1}^n (TT_i^{(t)}-1)m(\bm X_i)+\widehat{g}_t(\bm \Lambda_n)-\widehat{g}_t(\bm \Lambda_0)\big)^2\Big]\nonumber\\
=&\Ex_{\pi}[(M_n^{(t)})^2]
= \sum_{i=1}^n\Ex_{\pi}[(\Delta M_i^{(t)})^2]
= \sum_{i=1}^n  \Ex_{\pi}[\sigma_{t,t}(\bm\Lambda_{i-1})]\nonumber\\
=&n\pi[\sigma_{t,t}(\bm\Lambda)]=n(T-1)\big(\Ex[m^2(\bm X)]+\frac{2}{T-1}\pi[h_t(\bm \Lambda)]\big)=n(T-1)\vec{\sigma}_m^2=0,
 \end{align}
 by \eqref{eq:martingaleRep} and \eqref{eq:conditionalCovar1}.  Notice that for a function $\bm f(\bm x)$, similarly to \eqref{eq:multilowerboundofP2} we have
$$
 \Prob\Big(\sum_{i=1}^n(\bm T_i-1/T) \odot\bm f(\bm X_i)\in A|\mathscr{F}_0\Big)\ge (T \rho)^n \Prob\Big(\sum_{i=1}^n(\widetilde{\bm T}_i-1/T) \odot \bm f(\bm X_i) \in A\Big).
$$
It follows that
\begin{align*}
&\Ex_{\pi}\left[\Big(\sum_{i=1}^n (T\widetilde{T}_i^{(t)}-1) m(\bm X_i)+\widehat{g}_t\big(\bm\Lambda_0+\widetilde{\bm\Lambda}_n\big)-\widehat{g}_t(\bm\Lambda_0)\Big)^2\right]\\
\le &(T\rho)^{-n}\Ex_{\pi}\left[\Big(\sum_{i=1}^n (TT_i^{(t)}-1)m(\bm X_i)+\widehat{g}_t(\bm \Lambda_n)-\widehat{g}_t(\bm \Lambda_0)\Big)^2\right]=0,
\end{align*}
by \eqref{eq:proof-multi-linear-1}, where  $\widetilde{\bm\Lambda}_0=\bm 0$. Notice that $ \bm\Lambda \to  \bm \Upsilon $ is a one to one linear map. We can denote $H_t( \bm\Upsilon)=\widehat{g}_t\big( \bm \Lambda \big)$. Thus, under $\bm\Upsilon_0\sim \pi_{\Upsilon}$ (the  invariant probability measure of $\bm\Upsilon$), 
\begin{align}\label{eq:proof-multi-linear-ad1}
 \sum_{i=1}^n (T\widetilde{T}_i^{(t)}-1) m(\bm X_i) & =  H_t(\bm\Upsilon_0)-H_t\big(\bm\Upsilon_0+\widetilde{\bm\Upsilon}_n\big)\; a.s., \\
 \label{eq:proof-multi-linear-ad2}
|H_t(\bm\Upsilon_0)|\le & C\big(\|\bm\Upsilon_0\|+1\big) \; a.s. \text{ by } \eqref{eq:estofghat}. 
 \end{align} 
 Recall that $f_{\delta}$ and $ K_{\Upsilon,\epsilon}$ and are defined  as in Lemma \ref{lem:density}. Let
    $f_{\Upsilon}=f_{\delta}\ast \pi_{\Upsilon}$.   
  Then $f_{\Upsilon}(\bm y)$ is a continuous, bounded and positive density, $\int_Af_{\Upsilon}(\bm y) d\bm y\le c \pi_{\Upsilon}(A)$ by \eqref{eq-drift-kenerUp} and $\pi_{\Upsilon}=\pi_{\Upsilon}\ast K_{\Upsilon,\epsilon}$. Hence, by noticing that $\bm X_1, \widetilde{\bm T}_1, \bm X_2,  \widetilde{\bm T}_2,\ldots$ are independent of $\bm \Upsilon_0$, \eqref{eq:proof-multi-linear-ad1} and \eqref{eq:proof-multi-linear-ad2} also hold under $\bm\Upsilon_0\sim f_{\Upsilon}$. 
  
 Now, we define   $H(\bm x)$ by
$$  H(\bm x)=\int \left[H_t(\bm y)-H_t\big(\bm y+\bm x\big)\right]f_{\Upsilon}(\bm y)d\bm y. $$
Notice $f_{\Upsilon}(\bm y)>0$. By \eqref{eq:proof-multi-linear-ad2}, $|H_t(\bm x)|\le   C\big(\|\bm x\|+1\big)$ for almost all $\bm x=(\bm x^{(1)}, \bm x^{(2)})\in  \mathscr{Y}\times \mathbb R^{(q-q_1)(T-1)} $, i.e.,  for each  $\bm x^{(1)} \in \mathscr{Y}$  and almost all  
$\bm x^{(2)}$ under the Lebesgue measure. It follows that $\left|H_t(\bm y)\right|+\left|H_t\big(\bm y+\bm x\big)\right|\le C(\|\bm x\|+\|\bm y\|+2)$ for almost all $\bm x$ and $\bm y$. Also, 
$$ \int (\|\bm y\|+\|\bm x\|+2)f_{\Upsilon}(\bm y) d\bm y\le c\int (\|\bm y\|+\|\bm x\|+2)\pi_{\Upsilon}(d \bm y)=c(\pi_{\Upsilon}[\|\bm \Upsilon\|]+\|\bm x\|+2)<\infty, $$
by  $\pi_{\Upsilon}[\|\bm \Upsilon\|^{\gamma-1}<\infty$ as shown in Theorem \ref{thm:multiCARrecurrent}.
It follows that $H(\bm x)$ is well defined and \eqref{eq:boundnessofH1} holds for almost all $\bm x$. Also, by the continuity of $f_{\Upsilon}$, $H(\bm x)$ is a continuous function and thus \eqref{eq:boundnessofH1} holds for every $\bm x$. On the other hand, by the independency and \eqref{eq:proof-multi-linear-ad1},
$$  
 \sum_{i=1}^n (T\widetilde{T}_i^{(t)}-1) m(\bm X_i) 
 = \Ex\left[ H_t(\bm\Upsilon_0)-H_t\big(\bm\Upsilon_0+\widetilde{\bm\Upsilon}_n\big)\big|\bm  X_i,\widetilde{\bm T}_i,i=1,\ldots,n\right]=H(\widetilde{\bm\Upsilon}_n)\; a.s. $$
The proof of (a) is completed.

 (b)  We show that 
\begin{equation}\label{additivityofH}  H(\bm x+\bm y)= H(\bm x)+ H(\bm y)  \text{ for all } \bm x, \bm y\in \mathscr{Y}\times \mathbb R^{(q-q_1)(T-1)}.
\end{equation}

Let $S_n=\sum_{i=1}^n (T\widetilde{T}_i^{(t)}-1) m(\bm X_i)$. Then, $\big(S_{n+m}-S_n, \widetilde{\bm \Upsilon}_{n+m}-\widetilde{\bm\Upsilon}_n\big)$ and $\big(S_{m}, \widetilde{\bm \Upsilon}_m\big)$ have the same distribution, and so
$$\Ex\left[\big(S_{n+m}-S_n-H(\widetilde{\bm \Upsilon}_{n+m}-\widetilde{\bm \Upsilon}_n)\big)^2\right]=\Ex\left[\big(S_{m}-H(\widetilde{\bm \Upsilon}_m)\big)^2\right]=0, $$
which, together with \eqref{eq:findaH},  implies
\begin{equation}\label{eq:proof-multi-linear-4}\Ex\left[\big(H(\widetilde{\bm \Upsilon}_{n+m})- H(\widetilde{\bm \Upsilon}_{n+m}-\widetilde{\bm \Upsilon}_n)-H(\widetilde{\bm \Upsilon}_m)\big)^2\right]=0 \;\; \text{ for all } n,m.
\end{equation}
Thus, 
$$\int\int \left[\big(H(\bm x+\bm y)- H(\bm y)-H(\bm x)\big)^2\right] F_n(d\bm x)F_m(d\bm y)=0 \;\; \text{ for all } n,m,  $$
where $F_n$ is the distribution of $\widetilde{\bm \Upsilon}_n$. Notice the definition of the kernel  $\widetilde{K}_{\delta}$. It follows that
$$\int\int \left[\big(H(\bm x+\bm y)- H(\bm y)-H(\bm x)\big)^2\right] \widetilde{K}_{\delta}(d\bm x)\widetilde{K}_{\delta}(d\bm y)=0.  $$
By \eqref{eq:lemdensity}, it follows that
$$\int\int \left[\big(H(\bm x+\bm y)- H(\bm y)-H(\bm x)\big)^2\right]f_{\delta}(\bm x)f_{\delta}(\bm y)d\bm xd\bm y=0.  $$
Notice that $f_{\delta}(\bm x)>0$ for all $\bm x$, and $H(\bm x)$ and $f_{\delta}(\bm x)$ are continuous functions. It follows that \eqref{additivityofH} holds.  

(c) We show that $ H (\bm x)=\langle\bm w_0,\bm x\rangle$, $ \bm x  \in \mathscr{Y}\times \mathbb R^{(q-q_1)(T-1)}$, for some $\bm w_0\in  \mathbb R^{q(T-1)}$.   

It is obvious that $ H (\bm 0)=0$ by the definition, and then $ H (-\bm x)= H (\bm 0)- H (\bm x)=- H (\bm x)$.  Let $\mathscr{Z}=Span\{\mathscr{Y}\times \mathbb R^{(q-q_1)(T-1)}\}=\{\bm x=\sum_{j=1}^r\alpha_j\bm x_j:\alpha_j\in \mathbb R, \bm x_j\in \mathscr{Y}\times \mathbb R^{(q-q_1)(T-1)}, j=1,\cdots, r, r\ge 1\}$, and extend $ H $ to $\mathscr{Z}$ by defining 
$  H (\bm x)=\sum_{j=1}^r\alpha_j H (\bm x_j) $ for  $ \bm x=\sum_{j=1}^r\alpha_j\bm x_j\in \mathscr{Z}. $ Then, 
for $\bm x=\sum_{i=1}^r\alpha_i\bm x_i\in \mathscr{Z}, \bm y=\sum_{j=1}^s\beta_j\bm x_j\in \mathscr{Z}$, by letting
$\bm z_n=\sum_{i=1}^r[n\alpha_i]\bm x_i-\sum_{j=1}^s[n\beta_j]\bm x_j$, we have
\begin{align*}
&\left|\sum_{i=1}^r\alpha_i H (\bm x_i)-\sum_{j=1}^s\beta_j H (\bm x_j)\right| 
= \lim_{n\to\infty} \left|\frac{\sum_{i=1}^r[n\alpha_i] H (\bm x_i)-\sum_{j=1}^s[n\beta_j] H (\bm x_j)}{n}\right|\\
=& \lim_{n\to\infty}\left|\frac{ H (\bm z_n)}{n}\right|\le \lim_{n\to\infty}\frac{C(\|\bm z_n\|+1)}{n} =C\|\bm x-\bm y\|,
\end{align*}
by \eqref{additivityofH} and \eqref{eq:boundnessofH1}. It follows that the extension $ H $ is well defined on $\mathscr{Z}$, and 
$$ \left| H (\bm z_1)- H (\bm z_2)\right|\le C\|\bm z_1-\bm z_2\|, \;\; \bm z_1,\bm z_2\in \mathscr{Z}. $$
$ H $ is obviously a linear function on  $\mathscr{Z}$, i.e.,
$$  H (\alpha\bm z_1+\beta\bm z_2)= \alpha  H (\bm z_1)+\beta H (\bm z_2), \;\; \bm z_1,\bm z_2\in \mathscr{Z}, \alpha,\beta\in \mathbb R.  $$
Notice that $\mathscr{Z}$ is linear subspace of $\mathbb R^{q(T-1)}$. By the Hahn-Banach extension theorem, there exists a linear function $L$ on  $\mathbb R^{q(T-1)}$ such that
$$ |L(\bm z)|\le C\|\bm z\|, \; \bm z\in \mathbb R^{q(T-1)} \;\text{ and } L(\bm z)= H (\bm z), \bm z\in \mathscr{Z}. $$
Since $\mathbb R^{q(T-1)}$ is a Euclidean space, for the linear function $L$ there exists $\bm w_0\in \mathbb R^{q(T-1)}$ such that
$L(\bm z)=\langle\bm w_0,\bm z\rangle, \;\; \bm z\in \mathbb R^{q(T-1)}. $
  Hence, we arrive at
 $  H (\bm x)=L(\bm x)=\langle\bm w_0,\bm x\rangle$, $\bm x  \in \mathscr{Y}\times \mathbb R^{(q-q_1)(T-1)}. $

(d) We show that $m(\bm X)= \big\langle\bm x_0, \bm\phi(\bm X)\rangle$ a.s. for some $\bm x_0\in \mathbb R^{q}$. 

By \eqref{eq:findaH},    $m(\bm X_1)=\frac{1}{T\widetilde{T}_1^{(t)}-1}H(\widetilde{\bm \Upsilon}_1)=\frac{1}{T\widetilde{T}_1^{(t)}-1}H((\widetilde{\bm T}_1-\widetilde{T}_1^{(T)})\odot\phi(\bm X_1))$ a.s.  Let $\bm t$ be a $(T-1)$-dimensional  vector with the $t$-th element being $1$ and others being $0$. Then $\Prob\Big(m(\bm X)=\frac{1}{T-1} H \big(\bm t\odot \bm\phi(\bm X)\big)\Big)=1$ by noting the independence of $\widetilde{\bm T}_i$ and $\bm X_i$. 
Let  
$\bm x_0=\frac{\left(  w_{0,(t-1)q+1} , \cdots,   w_{0,tq}\right)}{T-1}.$
 Then  
\begin{align*} 
   \frac{1}{T-1}   H (\bm t\odot \bm x)=  & \frac{1}{T-1}\langle\bm w_0,\bm t\odot \bm x \rangle
 =  \langle\bm x_0, \bm x \rangle,\;\;
\bm x=(\bm x^{(1)},\bm x^{(2)})\in \mathscr{X}\times \mathbb R^{q-q_1}
\end{align*}
and, thus $\Prob\left(m(\bm X)= \big\langle\bm x_0, \bm\phi(\bm X)\big\rangle\right)=1$.
 The proof is now completed.
\end{proof}

\begin{proof}[Proof of Corollary \ref{cor:moments}] By \eqref{eq:martingaleRep} and \eqref{eq:zeromean}, we have
$$\max_{l\le m}\left|\sum_{i=j+1}^{j+l}\left[ (TT_i^{(t)}-1)Z_i+W_i^{(t)}\right]\right|
 \le \max_{l\le m} |M_{j+l}^{(t)}-M_{j}^{(t)}|+|\widehat{g}_t(\bm\Lambda_{j})|+\max_{l\le m} |\widehat{g}_t (\bm\Lambda_{j+l})|,
$$
where $\Delta M_i^{(t)}=(TT_i^{(t)}-1)Z_i+W_i^{(t)}+\widehat{g}_t(\bm\Lambda_i)-\widehat{g}_t(\bm\Lambda_{i-1})$ is a martingale difference.  Note
\begin{align*}
& \Ex[\max_{l\le m} |M_{j+l}^{(t)}-M_{j}^{(t)}|^2]\le 2\sum_{i=j+1}^{j+m}\Ex[(\Delta M_i^{(t)})^2] \\
\le &6 m (T-1)^2\Ex[Z^2]+6m  \Ex[(W^{(t)})^2]+6\sum_{i=j+1}^{j+m} \Ex[(\widehat{g}_t(\bm \Lambda_i)-\widehat{g}_t(\bm \Lambda_{i-1}))^2],
\end{align*}
and by \eqref{eq:estofghat},
\begin{align*} \Ex[(\widehat{g}_t(\bm \Lambda_i))^2]\le & \big(R(T-1)\Ex[|m(\bm X)|]\big)^2\Ex\left[\big(\|\bm \Lambda_i\|/c+1\big)^2\right]\\
\le & C_{\Lambda} \big( \Ex(|\Ex[Z|\bm X]|)\big)^2\le C_{\Lambda} \Ex[Z^2].
\end{align*}
The proof is completed.
\end{proof}

\subsection{Proofs of the Properties of the Tests}

We now consider the treatment effect tests under the $\bm \phi$-based CAR procedure.

\begin{proof}[Proof of Theorem \ref{thm:test}] Recall
\begin{align}\label{eq:proofLSE1}
\widehat{\theta}_n=&\sum_{i=1}^{n} \underline{\bm X}_{i}Y_i\left\{ \sum_{i=1}^{n}  \underline{\bm X}_{i}^{\otimes 2}\right\} ^{-1}
= \theta +\sum_{i=1}^{n} \underline{\bm X}_{i} e_i\left\{ \sum_{i=1}^{n}  \underline{\bm X}_{i}^{\otimes 2}\right\} ^{-1},
\end{align}
 $\underline{\bm X}_i=\left(T_i,1-T_i,X_{i,1},X_{i,2}\dots X_{i,p}\right)$, $\widehat{\tau}_n=(1,-1,0,\ldots,0)\widehat{\theta}_n^{\prime}$, $\widehat{\bm \beta}_n=diag(0,0,1, \ldots, 1) \widehat{\theta}_n^{\prime}$, $\widehat{\mu}_{n,1}=(1,0,0,\ldots,0)\widehat{\theta}_n^{\prime}$, $\widehat{\mu}_{n,0}=(0,1,0,\ldots,0)\widehat{\theta}_n^{\prime}$, $\widehat{\sigma}_e^2=\sum_{i=1}^n (Y_i-\underline{\bm X}_i\widehat{\theta}_n^{\prime})^2/(n-p-2)$, and
\begin{equation}\label{eq:proofLSE2}
\mathcal{T}_{LS}(n)=\frac{ \widehat{\theta}_n\bm L^{\prime}}{\widehat{\sigma}_e\big(\bm L \left\{\sum_{i=1}^{n}  \underline{\bm X}_{i}^{\otimes 2}\right\} ^{-1}\bm L^{\prime}\}^{1/2}}.
\end{equation}

The proof will be completed by tree steps.

Step 1. We   show that
 \begin{equation}\label{eq:proofthemtest.2}
   n\bm L(\sum_{i=1}^{n}  \underline{\bm X}_{i}^{\otimes 2})^{-1}\bm L^{\prime} \overset{P}\to 4,
 \end{equation}
  \begin{equation}\label{eq:proofthemtest.3}
 \widehat{\bm \beta}_n\overset{P}\to \bm \beta_{\bot} \; \text{ and } \widehat{\mu}_{n,a}= \Ex[Y(a)]-\Ex[\bm X_{obs}]\bm \beta_{\bot}^{\prime}+o_P(1),\; a=1,0.
 \end{equation}

  Let $N_{n,1}=\sum_{i=1}^n T_i$, $N_{n,0}=   \sum_{i=1}^n (1-T_i)$, $D_n=N_{n,1}-N_{n,0}$, $\bm D_n^x=\sum_{i=1}^n (2T_i-1)\bm X_{i,obs}$, and
$$
    \overline{\bm X}_{n,1}=\frac{\sum_{i=1}^nT_i \bm X_{i,obs}}{N_{n,1}},\;
    \; \overline{\bm X}_{n,0}=\frac{\sum_{i=1}^n(1-T_i) \bm X_{i,obs}}{N_{n,0}}, $$
  $$  \overline{\bm Y}_{n,1}=\frac{\sum_{i=1}^nT_i \bm Y_i}{N_{n,1}},\;
    \; \overline{\bm Y}_{n,0}=\frac{\sum_{i=1}^n(1-T_i) \bm Y_i}{N_{n,0}},
$$
\begin{align*}
    \bm S_{n,xx}=& \sum_{i=1}^n T_i(\bm X_{i,obs}-\overline{\bm X}_{n,1})^{\prime}(\bm X_{i,obs}-\overline{\bm X}_{n,1}) \\
    &+\sum_{i=1}^n (1-T_i)(\bm X_{i,obs}-\overline{\bm X}_{n,0})^{\prime}(\bm X_{i,obs}-\overline{\bm X}_{n,0}),\\
    \bm S_{n,xy}=& \sum_{i=1}^n T_iY_i(\bm X_{i,obs}-\overline{\bm X}_{n,1})+\sum_{i=1}^n (1-T_i)Y_i(\bm X_{i,obs}-\overline{\bm X}_{n,0}).
    \end{align*}
It can be verified that
$$ \widehat{\bm\beta}_n=\bm S_{n,xy}\bm S_{n,xx}^{-1}, \;\; \widehat{\mu}_{n,a}=\overline{Y}_{n,a}-\overline{\bm X}_{n,a}\widehat{\bm\beta}_n^{\prime},\; a=1,0, $$
$$\overline{\bm X}_{n,1}=\frac{n\overline{\bm X}_n+\bm D_n^x}{n+D_n}\; \text{ and } \;  \overline{\bm X}_{n,0}=\frac{n\overline{\bm X}_n-\bm D_n^x}{n-D_n}, $$
$$
   \bm L(\sum_{i=1}^{n}  \underline{\bm X}_{i}^{\otimes 2})^{-1}\bm L^{\prime}
 =  \frac{1}{N_{n,1}}+\frac{1}{N_{n,2}}+(\overline{\bm X}_{n,1}-\overline{\bm X}_{n,2})\bm S_{n,xx}^{-1}(\overline{\bm X}_{n,1}-\overline{\bm X}_{n,2})^{\prime}, $$
 where $\overline{\bm X}_n=\sum_{i=1}^n \bm X_{i,obs}/n$.  By Theorem \ref{thm:clt}, it follows that,
 $$ D_n=O_P(n^{1/2}) \text{ and } \bm D_n^x=O_P(n^{1/2}). $$
 Hence,
 $$ \overline{\bm X}_{n,a}=\overline{\bm X}_n+O_P(n^{-1/2})= \Ex[\bm X_{obs}]+O_P(n^{-1/2}), \; a=1,0, $$
 which implies that $\overline{\bm X}_{n,1}-\overline{\bm X}_{n,0}=O_P(n^{-1/2})$ and
 $$ \frac{\bm S_{n,xx}}{n}=\frac{\sum_{i=1}^n(\bm X_{i,obs}-\Ex[\bm X_{obs}])^{\prime}(\bm X_{i,obs}-\Ex[\bm X_{obs}])}{n}+o_P(1)\overset{P}\to \Var(\bm X_{obs}), $$
 $$
   n\bm L(\sum_{i=1}^{n}  \underline{\bm X}_{i}^{\otimes 2})^{-1}\bm L^{\prime}
 =  \frac{2n}{n+D_n}+\frac{2n}{n-D_n}+o_P(1)\overset{P}\to 4. $$
 \eqref{eq:proofthemtest.2} is proved.
  On the other hand,
 \begin{align*}
 &\frac{1}{n}\sum_{i=1}^n\big(T_i Y_i(1)(\bm X_{i,obs}-\overline{\bm X}_{n,1})+(1-T_i) Y_i(0)(\bm X_{i,obs}-\overline{\bm X}_{n,0})\big)\\
 =& \frac{1}{n}\sum_{i=1}^n\big\{T_i (Y_i(1)-\mu_1)(\bm X_{i,obs}-\overline{\bm X}_{n,1})+(1-T_i) (Y_i(0)-\mu_0)(\bm X_{i,obs}-\overline{\bm X}_{n,0})\big\}\\
 =&\frac{1}{n}\sum_{i=1}^n\big(T_i (m_1(\bm X_i)+\epsilon_{i,1}^{(n)})+(1-T_i)(m_0(\bm X_i)+\epsilon_{i,0}^{(n)})\big)(\bm X_{i,obs}-\Ex[\bm X_{obs}])+O_P(n^{-1/2})\\
 =&\frac{1}{n}\sum_{i=1}^n\big((2T_i-1)\frac{\epsilon_{i,1}-\epsilon_{i,0}+m_1(\bm X_i)-m_0(\bm X_i)}{2}(\bm X_{i,obs}-\Ex[\bm X_{obs}])\\
 &\; +\frac{\epsilon_{i,1}+\epsilon_{i,0}+m_1(\bm X_i)+m_0(\bm X_i)}{2}(\bm X_{i,obs}-\Ex[\bm X_{obs}]) +o_P(1)\\
 =&\frac{1}{n}\sum_{i=1}^n
  \frac{\epsilon_{i,1}+\epsilon_{i,0}+m_1(\bm X_i)+m_0(\bm X_i)}{2}(\bm X_{i,obs}-\Ex[\bm X_{obs}]) +o_P(1)\;\;\big(\text{by Corollary \ref{cor:LLN}}\big)\\
 & \overset{P}\to     \Ex\left[ \frac{\epsilon_{1}+\epsilon_{0}+m_1(\bm X)+m_0(\bm X)}{2}(\bm X_{obs}-\Ex[\bm X_{obs}])\right]=\Cov\{m_{ave}(\bm X),\bm X_{obs}\}.
 \end{align*}
 Hence we conclude that $\widehat{\bm \beta}_n\overset{P}\to   \bm\beta_{\bot}$.

Further, by   Corollary \ref{cor:LLN}  again, 
\begin{align*} &\widehat{\mu}_{n,1}= \overline{Y}_{n,1}-\overline{\bm X}_{n,1}\widehat{\bm \beta}_n^{\prime}\\
= &\frac{2\sum_{i=1}^nT_i (Y_i(1)-\Ex[Y(1)])}{n+D_n}+\Ex[Y(1)]-\overline{\bm X}_{n,1}(\widehat{\bm \beta}_n- \bm \beta_{\bot})^{\prime}-\overline{\bm X}_{n,1}\bm \beta_{\bot}^{\prime} \\
=&\frac{\sum_{i=1}^nT_i(  m_1(\bm X_i)-\Ex[m_1(\bm X_i)]+\epsilon_{i,1})}{n}
+\Ex[Y(1)]-\Ex[\bm X_{obs}]\bm \beta_{\bot}^{\prime}+o_P(1)\\
=&\Ex[Y(1)]-\Ex[\bm X_{obs}]\bm \beta_{\bot}^{\prime}+o_P(1),
\end{align*}
and, similarly,
$$ \widehat{\mu}_{n,0}=\Ex[Y(0)]-\Ex[\bm X_{obs}]\bm \beta_{\bot}^{\prime}+o_P(1). $$
\eqref{eq:proofthemtest.3} is proved.

 Step 2. We show
 \begin{equation}\label{eq:proofofesttau}
 \sqrt{n}(\widehat{\tau}_n-\tau)\overset{D}\to N(0,4\sigma_{\tau}^2).
  \end{equation}

 Note  $\Ex[Y(1)]-\Ex[Y(0)]=\mu_1+\Ex[m_1(\bm X)]-(\mu_0+\Ex[m_0(\bm X)])=\mu_1-\mu_0=\tau$,
 \begin{align*}
 \overline{\bm X}_{n,1}-\overline{\bm X}_{n,0}=&2\frac{\frac{1}{n}\sum_{i=1}^n (2T_i-1)(\bm X_{i,obs}-\Ex[\bm X_{obs}])-\frac{D_n}{n}(\overline{\bm X}_n-\Ex[\bm X_{obs}])}{1-D_n^2/n^2}\\
 =& 2\frac{\sum_{i=1}^n (2T_i-1)(\bm X_{i,obs}-\Ex[\bm X_{obs}]) }{n} +O_P(n^{-1}),
 \end{align*}
 \begin{align*}
 &\overline{Y}_{n,1}-\Ex[Y(1)]-(\overline{Y}_{n,0}-\Ex[Y(0)])\\
 =&2\frac{\sum_{i=1}^n\big(T_i (Y_i(1)-\Ex[Y(1)])- (1-T_i) (Y_i(0)-\Ex[Y(0)])\big)}{n} +o_P(n^{-1/2}),
 \end{align*}
 and $(\overline{\bm X}_{n,1}-\overline{\bm X}_{n,0})\widehat{\bm\beta}_n^{\prime}=(\overline{\bm X}_{n,1}-\overline{\bm X}_{n,0})\widehat{\bm\beta}^{\prime}+o_P(n^{-1/2})$. Recall $\underline{\bm X}_i=(T_i,1-T_i, \bm X_{i,obs})$. We have
\begin{align*}   &\widehat{\tau}_n-\tau=   \overline{Y}_{n,1}-\Ex[Y(1)]-(\overline{Y}_{n,0}-\Ex[Y(0)])-(\overline{\bm X}_{n,1}-\overline{\bm X}_{n,0})\widehat{\bm\beta}_n^{\prime}
=  2\frac{\sum_{i=1}^n \widetilde{r}_i}{n}+o_P(n^{-1/2}),
\end{align*}
where
\begin{align}\label{eq:tau}
\widetilde{r}_i=&  T_i(Y_i(1)-\underline{\bm X}_i\theta_{\ast}^{\prime})- (1-T_i)(Y_i(0)-\underline{\bm X}_i\theta_{\ast}^{\prime}) \nonumber\\
= &(2T_i-1)( \frac{\epsilon_{i,1}^{(n)}+ \epsilon_{i,0}^{(n)}}{2}+m(\bm X_i))+\frac{ \epsilon_{i,1}^{(n)}- \epsilon_{i,0}^{(n)}+m_1(\bm X_i)-m_0(\bm X_i)}{2}\nonumber\\
= &(2T_i-1)( \frac{\epsilon_{i,1}+ \epsilon_{i,0}}{2}+m(\bm X_i))+\frac{ \epsilon_{i,1}- \epsilon_{i,0}+m_1(\bm X_i)-m_0(\bm X_i)}{2}\nonumber\\
&+(2T_i-1)( \frac{\epsilon_{i,1}^{(n)}-\epsilon_{i,1}+ \epsilon_{i,0}^{(n)}-\epsilon_{i,0}}{2})+\frac{\epsilon_{i,1}^{(n)}- \epsilon_{i,1}- \epsilon_{i,0}^{(n)}+\epsilon_{i,0}}{2}\nonumber\\
=:& (2T_i-1)Z_i+W_i+(2T_i-1)\widehat{Z}_i^{(n)}+\widehat{W}_i^{(n)}.
\end{align}
By Corollary \ref{cor:moments} and the condition (L1) for the model \eqref{eq:nonlinearmodel},
$$ \frac{1}{n}\Ex\big(\sum_{i=1}^n [(2T_i-1)\widehat{Z}_i^{(n)}+\widehat{W}_i^{(n)}]\big)^2\le C_{\Lambda} \Ex[|\epsilon_{i,1}^{(n)}-\epsilon_{i,1}|^2+|\epsilon_{1,0}^{(n)}-\epsilon_{1,0}|^2]\to 0. $$
Hence, by Theorem \ref{thm:generalclt}, it follows that
$$ \frac{1}{2}\sqrt{n}(\widehat{\tau}_n-\tau)=\frac{\sum_{i=1}^n \big((2T_i-1)Z_i+W_i\big)}{\sqrt{n}}+o_P(1)\overset{D}\to N(0,\sigma_{\tau}^2)$$
with
$$\sigma_{\tau}^2=  \Var\{\frac{\epsilon_1+ \epsilon_0}{2}\}+\Var\{\frac{ \epsilon_1- \epsilon_0+m_1(\bm X)-m_0(\bm X)}{2}\}+\vec{\sigma}_m^2=\sigma_{\epsilon}^2+\vec{\sigma}_m^2. $$
 \eqref{eq:proofofesttau} is proved.

Step 3. We show that
\begin{equation}\label{eq:proofthemtest.4}\widehat{\sigma}_e^2\overset{P}\to \sigma_e^2.
\end{equation}
Write $\xi_i=(\bm X_{i,obs}-\Ex[\bm X_{obs}])\bm\beta_{\bot}^{\prime}$.  Then
$ Y_i-\underline{\bm X}_i\widehat{\theta}_n= Y_i-\underline{\bm X}_i\theta_{\ast}^{\prime}-\underline{\bm X}_i(\widehat{\theta}_n-\theta_{\ast})^{\prime},$
Note $\widehat{\theta}_n-\theta_{\ast}=o_P(1)$,  $\sum_{i=1}^n\big(Y_i-\underline{\bm X}_i\theta_{\ast}^{\prime})\underline{\bm X}_i=O_P(n)$ and $\sum_{i=1}^n \underline{\bm X}_i^{\otimes 2}=O_P(n)$. We have
\begin{align}\label{eq:proofthemtest.5}
\widehat{\sigma}_e^2=&\frac{\sum_{i=1}^n(Y_i-\underline{\bm X}_i\theta_{\ast}^{\prime})^2+o_P(n)}{n-p-2}=\frac{\sum_{i=1}^n\widetilde{r}_i^2}{n}+o_P(1)\\
=&\frac{\sum_{i=1}^n[(2T_i-1)Z_i+W_i]^2}{n}+o_P(1)\nonumber\\
=& \frac{2\sum_{i=1}^n (2T_i-1)W_iZ_i  }{n}+\frac{\sum_{i=1}^n (W_i^2+Z_i^2)  }{n}+o_P(1).\nonumber
\end{align}
The first term above will converge in probability to zero by Corollary \ref{cor:LLN}. The second term   will converge in probability to
\begin{align} \nonumber &\Ex[Z^2]+\Ex[W^2]
= \Var\{\frac{\epsilon_1+ \epsilon_0}{2}+m(\bm X)\}+\Var\{\frac{ \epsilon_1- \epsilon_0+m_1(\bm X)-m_0(\bm X)}{2}\}=:\sigma_e^2.
\label{eq:proofthemtest.6}
\end{align}
by the law of large numbers.
Hence, \eqref{eq:proofthemtest.4} is proved.
  The proof is completed.
\end{proof}

\begin{proof}[Proof of Corollaries  \ref{cor:test} and \ref{cor:testdis}] Corollary \ref{cor:test} follows from Theorem \ref{thm:test} immediately. For Corollary \ref{cor:testdis}, it is sufficient to show that $\Ex[m(\bm X)|\bm X_{dis}]\in Span\{\bm\phi(\bm X)\}$. Assume   that $\bm X_{dis}=(d_1(X_1),\ldots, d_{p+q}(X_{p+q}))$ takes   possible values $(\widetilde{x}_1^{(l_1)},\ldots, \widetilde{x}_{p+q}^{(l_{p+q})})$, $l_t=1,\ldots,L_t$, $t=1,\ldots,p+q$. Note that $\Ex[m(\bm X)|\bm X_{dis}]=f(\bm X_{dis})$ is a function of $\bm X_{dis}$. So,
$$ \Ex[m(\bm X)|\bm X_{dis}]=\sum_{l_t=1,\ldots,L_t,t=1,\ldots,p+q} f(\widetilde{x}_1^{(l_1)},\ldots, \widetilde{x}_{p+q}^{(l_{p+q})}) I\{\bm X_{dis}=(\widetilde{x}_1^{(l_1)},\ldots, \widetilde{x}_{p+q}^{(l_{p+q})})\} $$
is a linear function of $\bm\phi_{S}(\bm X_{dis})$ or $\bm\phi_{HH}(\bm X_{dis})$ with $w_s\ne 0$.

 Further,  if the condition (vii) is satisfied, then $m(\bm X)=\alpha_0+\sum_{t=1}^{p+q}\alpha_tf_t(X_t)$, and $X_1$, $\ldots$, $X_{p+q}$ are independent. So,
\begin{align*}
   \Ex[m(\bm X)|\bm X_{dis}]= & \alpha_0+\sum_{t=1}^{p+q}\alpha_t\Ex[f_t(X_t)|\bm X_{dis}]=\alpha_0+\sum_{t=1}^{p+q}\alpha_t\Ex[f_t(X_t)|d_t(X_t)]\\
 = & \alpha_0+\sum_{t=1}^{p+q}\sum_{l_t=1}^{L_t}\alpha_t\Ex[f_t(X_t)|d_t(X_t)=\widetilde{x}_t^{(l_t)}]I\{d_t(X_t)=\widetilde{x}_t^{(l_t)}\}
 \end{align*}
 is a linear function of $\bm\phi_{PS}(\bm X_{dis})$.
\end{proof}

\begin{proof}[Proof of Theorem \ref{thm:regest}] We want to prove $\widehat{\sigma}_{\tau,reg}^2\overset{P}\to \sigma_{\tau}^2$. For simplifying the proof, without loss of generality we assume that $\bm B=\Ex[\bm\phi(\bm X)^{\otimes 2}]$ is non-singular. Let $\alpha=\Ex[\bm\phi(\bm X)m(\bm X)]\left[\Ex[\bm\phi(\bm X)^{\otimes 2}]\right]^{-1}$. Then $\bm\phi(\bm X)\alpha^{\prime}$ is the projection of $m(\bm X)$ on $Span\{\bm\phi(\bm X)\}$.
By noting $\widehat{\theta}- \theta_{\ast}\overset{P}\to 0$ and $\theta_{\ast}=(\Ex[Y(1)]-\Ex[\bm X_{obs}]\bm \beta_{\bot}^{\prime},\Ex[Y(0)]-\Ex[\bm X_{obs}]\bm \beta_{\bot}^{\prime},\bm\beta_{\bot})$, we have
$\widehat{\zeta}_i=Y_i-\underline{\bm X}_i\widehat{\theta}^{\prime}-\bm\phi(\bm X_i)\widehat{\alpha}^{\prime}$ with
\begin{align*}
&\widehat{\alpha}= \sum_{i=1}^n \bm\phi(\bm X_i)(Y_i-\underline{\bm X}_i\widehat{\theta}^{\prime})\left[\sum_{i=1}^n\bm\phi(\bm X_i)^{\otimes 2}\right]^{-1}\\
=&\frac{1}{n}\sum_{i=1}^n \bm\phi(\bm X_i)(Y_i-\underline{\bm X}_i\widehat{\theta}^{\prime})\bm B^{-1}+o_P(1)
= \frac{1}{n}\sum_{i=1}^n \bm\phi(\bm X_i)(Y_i-\underline{\bm X}_i\theta^{\prime}_{\ast})\bm B^{-1}+o_P(1)\\
=&\frac{1}{n}\sum_{i=1}^n \left\{T_i\bm\phi(\bm X_i)(h_1(\bm X_i)+\epsilon_{i,1}^{(n)})+(1-T_i)\bm\phi(\bm X_i)(h_0(\bm X_i)+\epsilon_{i,0}^{(n)})\right\}\bm B^{-1}+o_P(1)\\
=&\frac{1}{n}\sum_{i=1}^n \big(T_i-\frac{1}{2}\big)\bm\phi(\bm X_i)(h_1(\bm X_i)-h_0(\bm X_i)+\epsilon_{i,1}^{(n)}-\epsilon_{i,0}^{(n)})\bm B^{-1}\\
&+\frac{1}{n}\sum_{i=1}^n \bm\phi(\bm X_i)\frac{h_1(\bm X_i)+h_0(\bm X_i)+\epsilon_{i,1}^{(n)}+\epsilon_{i,0}^{(n)}}{2}\bm B^{-1}+o_P(1)\\
  \overset{P}\to & \Ex\left[\bm\phi(\bm X)\frac{h_1(\bm X)+h_0(\bm X)}{2}\right]\bm B^{-1}=\alpha.
\end{align*}
where $h_a(\bm X)=m_a(\bm X)- \Ex[m_a(\bm X)]-(\bm X_{obs}-\Ex[\bm X_{obs}])\bm \beta_{\bot}^{\prime}$, $a=1,0$. Similarly to \eqref{eq:proofthemtest.5} and \eqref{eq:proofthemtest.4},
\begin{align}
 \widehat{\sigma}_{\tau,reg}^2= &\frac{\sum_{i=1}^n(Y_i-\underline{\bm X}_i\widehat{\theta}^{\prime}-\bm\phi(\bm X_i)\widehat{\alpha}^{\prime})^2}{n-p-2}  = \frac{\sum_{i=1}^n(Y_i-\underline{\bm X}_i\theta_{\ast}^{\prime}-\bm\phi(\bm X_i) \alpha^{\prime})^2+o_P(n)}{n-p-2} \nonumber\\
=&\frac{\sum_{i=1}^n\left((2T_i-1)( \frac{\epsilon_{i,1}+ \epsilon_{i,0}}{2}+m(\bm X_i)-\bm\phi(\bm X_i)\alpha^{\prime})+\frac{ \epsilon_{i,1}- \epsilon_{i,0}+m_1(\bm X_i)-m_0(\bm X_i)}{2}\right)^2}{n-p-2}+o_P(1)\nonumber\\
=& \Ex[(m(\bm X)-\bm\phi(\bm X)\alpha^{\prime})^2]+\frac{1}{4}\Ex[(m_1(\bm X)-m_0(\bm X))^2]+\frac{1}{2}(\Ex\epsilon_{1,1}^2+\Ex\epsilon_{1,0}^2) +o_P(1)\nonumber\\
\overset{P}\to & \Ex[(m(\bm X)-\bm\phi(\bm X)\alpha^{\prime})^2]+\sigma_{\epsilon}^2.
\end{align}
By the assumption \eqref{eq:inSpan}, $\Ex[m(\bm X)|\bm\phi(\bm X)]=\bm\phi(\bm X)\alpha^{\prime}$. Then by \eqref{eq:variance-m-conditionM}, $\vec{\sigma}_m^2=\Var(m(\bm X)-\bm\phi(\bm X)\alpha^{\prime})
+\sigma_{\bm\phi(\bm X)\alpha^{\prime}}^2=\Ex[(m(\bm X)-\bm\phi(\bm X)\alpha^{\prime})^2]$.  The proof is completed.
\end{proof}

\begin{proof}[Proof of Theorem \ref{thm:mbest}, \eqref{eq:mbjcon} and \eqref{eq:mbbcon}] We first  show that $\widehat{\sigma}_{\tau,mb}^2\overset{P}\to \sigma_{\tau}^2$.
Recall $\theta_{\ast}=(\Ex[Y(1)]-\Ex[\bm X_{obs}]\bm \beta_{\bot}^{\prime},\Ex[Y(0)]-\Ex[\bm X_{obs}]\bm \beta_{\bot}^{\prime},\bm\beta_{\bot})$,
$\xi_i=(\bm X_{i,obs}-\Ex[\bm X_{obs}])\bm\beta_{\bot}^{\prime}$ and $r_i=\widetilde{r}_i-(2T_i-1)\underline{\bm X}_i(\widehat{\theta}_n-\theta_{\ast})^{\prime}$, where $\widetilde{r}_i$ is defined as in \eqref{eq:tau} and can be written as $\widetilde{r}_i=(2T_i-1)Z_i+W_i+(2T_i-1)\widehat{Z}_i^{(n)}+\widehat{W}_i^{(n)}$.  It is sufficient to show that
\begin{equation}\label{eqproof:testadj1}
\frac{1}{nl}\sum_{i=0}^{n-l}\Big(\sum_{j=i+1}^{i+l} (2T_j-1)\underline{\bm X}_j(\widehat{\theta}_n-\theta_{\ast})^{\prime}\Big)^2
 \overset{P}\to 0,
 \end{equation}
 \begin{equation}
\label{eqproof:testadj0}
   \frac{1}{nl}\sum_{i=0}^{n-l}\Big(\sum_{j=i+1}^{i+l}[(2T_j-1)\widehat{Z}_j^{(n)}+\widehat{W}_j^{(n)}]\Big)^2
  \overset{P}\to 0,
 \end{equation}
 \begin{equation}
\label{eqproof:testadj2}
   \frac{1}{nl}\sum_{i=0}^{n-l}\Big(\sum_{j=i+1}^{i+l}[(2T_j-1)Z_j+W_j]\Big)^2
  \overset{P}\to \sigma_{\tau}^2=\sigma_{\epsilon}^2+\vec{\sigma}_m^2.
 \end{equation}
 By Corollary \ref{cor:moments},
 $$ \Ex\left[ \frac{1}{nl}\sum_{i=0}^{n-l}(\sum_{j=i+1}^{i+l}[(2T_j-1)\widehat{Z}_j^{(n)}+\widehat{W}_j^{(n)}])^2\right]\le C_{\Lambda} \sup_j \big( \Ex[\widehat{Z}_j^{(n)}]^2+
  \Ex[\widehat{W}_j^{(n)}]^2\big)\to 0. $$
\eqref{eqproof:testadj0} follows immediately.

For \eqref{eqproof:testadj2}, recall that $\bm{\Lambda}_n$ is a positive Harris recurrent Markov chain with an      invariance probability measure   $\pi$  and  $\Ex_{\pi}[\|\bm{\Lambda}_n\|^{\gamma-1}]<\infty$ by Theorem \ref{thm:recurrent}.  Note $\Ex[Z_i|\bm X_i]=m(\bm X_i)$.  Let  $\widehat{g}(\bm\Lambda)$ be the solution of the Poisson equation \eqref{eq:Poissonmulti} with $g(\bm\Lambda_{i-1})=\Ex[(2T_i-1)m(\bm X_i)|\bm \Lambda_{i-1}]$. Note $\pi[g(\bm\Lambda)]=0$.
Then as in the proof of \eqref{eq:martingaleRep}, $(2T_i-1)Z_i+W_i=\Delta M_i+\widehat{g}(\bm\Lambda_{i-1})-\widehat{g}(\bm\Lambda_{i})$ and $\{\Delta M_i;i=1,2,\ldots\}$ is a sequence of martingale differences with $\Ex[(\Delta M_i)^2|\widetilde{\mathscr{F}}_{i-1}]=\sigma^2(\bm \Lambda_{i-1})$, $\sigma^2(\bm y)\le C(\|\bm y\|^2+1)$, $|\widehat{g}(\bm y)|\le C(\|\bm y\|+1)$ and $\pi[\sigma^2(\bm \Lambda)]=\Ex[(Z-m(\bm X))^2]+\Var(W)+\vec{\sigma}_m^2=\sigma_{\tau}^2$.
It is easily seen that
$$ \Ex\left[\frac{1}{nl}\sum_{i=0}^{n-l}( \widehat{g}(\bm\Lambda_{i})-\widehat{g}(\bm\Lambda_{i+l}))^2\right]\le 4\frac{n-l}{nl}\sup_i \Ex\big[\widehat{g}^2(\bm\Lambda_{i})\big]
\le   \frac{C}{l}\sup_i \Ex\big[\|\bm\Lambda_{i}\|^2+1\big]\to 0. $$
On the other hand, let $Q_{i,l}=(M_{i+l}-M_i)^2-\sum_{j=i+1}^{i+l}\Ex[(\Delta M_j)^2|\widetilde{\mathscr{F}}_{j-1}]$. It is easily seen that $\Ex[Q_{i,l}|\widetilde{\mathscr{F}}_{i-1}]=0$. It follows that $\{Q_{il+j,l},i=0,\ldots, [n/l]-1\}$ is a sequence of martingale differences. We first  suppose $\max_i\Ex[|\Delta M_i|^{4}]<\infty$ at a moment.  Redefine $Q_{i,l}$ to be zero for $i>n-l$. By  moment inequalities for martingales (c.f. \cite{Hall1980}, Page 23),
\begin{align}\label{eqproof:testadj3}
 &\Ex\Big|\frac{1}{nl}\sum_{i=0}^{n-l}Q_{i,l}\Big|^2=\frac{1}{(nl)^2}\Ex\Big|\sum_{j=0}^{l-1}\sum_{i=0}^{[n/l]-1}Q_{il+j,l}\Big|^2
 \le  \frac{l}{(nl)^2}\sum_{j=0}^{l-1}\Ex\Big|\sum_{i=0}^{[n/l]-1}Q_{il+j,l}\Big|^2 \nonumber\\
 =  &\frac{l}{(nl)^2}\sum_{i=0}^{l-1} \sum_{i=0}^{[n/l]-1}\Ex\Big|Q_{il+j,l}\Big|^2
 =  \frac{l}{(nl)^2 }\sum_{i=0}^{n-l} \Ex[|Q_{i,l}|^2] \nonumber\\
\le  &\frac{Cl}{(nl)^2}\sum_{i=0}^{n-l} \Ex[|M_{i+l}-M_i|^4]
\le  \frac{Cl}{(nl)^2}\sum_{i=0}^{n-l} l^2\max_i\Ex[|\Delta M_i|^4]
\le    \frac{Cl}{n}\to 0,
\end{align}
which implies that
\begin{equation}\label{eqproof:testadj4} \frac{1}{nl}\sum_{i=0}^{n-l}Q_{i,l}\overset{P}\to 0.
\end{equation}
It follows that
\begin{align*}
 &\frac{1}{nl}  \sum_{i=0}^{n-l}(\sum_{j=i+1}^{i+l}   [(2T_j-1)Z_j+W_j])^2=\frac{1}{nl}  \sum_{i=0}^{n-l}(M_{i+l}-M_i)^2+o_P(1)\\
 =& \frac{1}{nl}\sum_{i=0}^{n-l} \sum_{j=i+1}^{i+l}\Ex[(\Delta M_j)^2|\widetilde{\mathscr{F}}_{j-1}]+o_P(1)=\frac{1}{n}\sum_{j=1}^{n}  \Ex[(\Delta M_j)^2|\widetilde{\mathscr{F}}_{j-1}]+o_P(1)\\
=& \frac{1}{n}\sum_{j=1}^n \sigma^2(\bm\Lambda_{j-1})+o_P(1)\overset{P}\to \pi[\sigma^2(\bm\Lambda)]=\sigma_{\tau}^2,
\end{align*}
 by the ergodic theorem (c.f. \eqref{eq:conditionalCovar3}). Hence, \eqref{eqproof:testadj1} holds.

Next, we remove the condition $\max_i\Ex[|\Delta M_i|^4]<\infty$. Let $\Delta M_i^{(K)}=\Delta M_iI\{|\Delta M_i|\le K\}-\Ex[\Delta M_iI\{|\Delta M_i|\le K\}|\widetilde{\mathscr{F}}_{i-1}]$. Then by \eqref{eqproof:testadj3},
\begin{align*}
  \Ex\Big|\frac{1}{nl}\sum_{i=0}^{n-l}\Big((\sum_{j=i+1}^{i+l}\Delta M_j^{(K)})^2-\sum_{j=i+1}^{i+l}\Ex[(\Delta M_j^{(K)})^2|\widetilde{\mathscr{F}}_{j-1}]\Big)\Big|^2
  \le  \frac{C}{(nl)^2}\sum_{i=0}^{n-l} l^2 K^4\to 0.
\end{align*}
On the other hand, note that $|\Delta M_i|\le |Z_i|+|W_i|+|\widehat{g}(\bm \Lambda_i)|+|\widehat{g}(\bm \Lambda_{i-1})|\le |Z_i|+|W_i|+C(\|\bm \Lambda_i\|+\|\bm \Lambda_{i-1}\|+2)$. We have
$\Ex_{\pi}[(\Delta M_i)^2]<\infty$, and then
\begin{align*}
& \Ex\Big|\frac{1}{nl}\sum_{i=0}^{n-l}\Big((\sum_{j=i+1}^{i+l}(\Delta M_j-\Delta M_j^{(K)})\Big)^2\Big| =\Ex\left[\frac{1}{nl}\sum_{i=0}^{n-l}\sum_{j=i+1}^{i+l}\Ex[(\Delta M_j-\Delta M_j^{(K)})^2|\widetilde{\mathscr{F}}_{j-1}]\right]  \\
= &\frac{1}{nl}\sum_{i=0}^{n-l}\sum_{j=i+1}^{i+l}\Ex[(\Delta M_j-\Delta M_j^{(K)}))^2]\le \frac{1}{n}\sum_{i=1}^{n}\Ex[(\Delta M_i)^2I\{|\Delta M_i|\ge K\}]\\
& \to \Ex_{\pi}[(\Delta M_1)^2I\{|\Delta M_1|\ge K\}]\to 0 \text{ as } K\to \infty.
\end{align*}
It follows that \eqref{eqproof:testadj4} remains true.

For \eqref{eqproof:testadj1}, note $(2T_j-1)\underline{\bm X}_j(\widehat{\theta}_n-\theta_{\ast})=\frac{1}{2}(\widehat{\tau}_n-\tau)+(2T_j-1)(\frac{1}{2},\frac{1}{2},\bm X_{j,obs})(\widehat{\theta}_n-\theta_{\ast})^{\prime} $. Similar to \eqref{eqproof:testadj2}, for $Z_j=\frac{1}{2}$ or $X_{j,t}$ we have
$$ \frac{1}{nl}\sum_{i=0}^{n-l}(\sum_{j=i+1}^{i+l}(2T_j-1)Z_j)^2\overset{P}\to \sigma^2 \text{ for some } \sigma\ge 0. $$
It follows that
\begin{align}\label{eqprrof:testadj5}
&\frac{1}{nl}\sum_{i=0}^{n-l}\Big(\sum_{j=i+1}^{i+l} (2T_j-1)\underline{\bm X}_j(\widehat{\theta}_n-\theta_{\ast})^{\prime}\Big)^2\nonumber \\
\le & \frac{2}{nl}\sum_{i=0}^{n-l}\Big(\frac{l}{2}(\widehat{\tau}_n-\tau)\Big)^2+\frac{2}{nl}\sum_{i=0}^{n-l}\Big\|\sum_{j=i+1}^{i+l} (2T_j-1)(\frac{1}{2},\frac{1}{2},\bm X_{j,obs})\big\|^2\|\widehat{\theta}_n-\theta_{\ast}\|^2\nonumber \\
=& O_P(1)\frac{n-l}{nl}(ln^{-1/2})^2+O_P(1)o_P(1)=O_P(\frac{l}{n})+o_P(1)\overset{P}\to 0,
\end{align}
by noting
$ \widehat{\theta}_n-\theta_{\ast}=o_P(1) $  and $\widehat{\tau}_n-\tau=O_P(n^{-1/2})$.
Hence \eqref{eqproof:testadj1} is proved and the proof $\widehat{\sigma}_{\tau,mb}^2\overset{P}\to \sigma_{\tau}^2$ is completed.

By noting \eqref{eqproof:testadj0} and \eqref{eqproof:testadj2}, for \eqref{eq:mbjcon} and \eqref{eq:mbbcon} it is sufficient to show that
\begin{equation}\label{eqprrof:testadj6}
\frac{\sum_{i=0}^{n-l}\sum_{j=1}^l \widetilde{r}_{i+j}}{(n-l+1)\sqrt{l}}\overset{P}\to 0.
\end{equation}
Write $\widetilde{R}_i=\sum_{j=1}^i\widetilde{r}_j$. By Corollary \ref{cor:moments}, $\Ex[\widetilde{R}_i^2]\le Ci$. Then
\begin{align*}
&\Ex\left|\frac{\sum_{i=0}^{n-l}\sum_{j=1}^l \widetilde{r}_{i+j}}{(n-l+1)\sqrt{l}}\right|
=\frac{\Ex|\sum_{i=n-l+1}^{n}\widetilde{R}_i-\sum_{i=1}^{l-1} \widetilde{R}_{i}|}{(n-l+1)\sqrt{l}}\\
\le & C\frac{ \sum_{i=n-l+1}^{n}\sqrt{i}+\sum_{i=1}^{l-1} \sqrt{i}|}{(n-l+1)\sqrt{l}}\le C\frac{ l(\sqrt{n}+  \sqrt{l})}{(n-l+1)\sqrt{l}}\to 0.
\end{align*}
\eqref{eqprrof:testadj6} is proved.
\end{proof}

\newpage

\section{Simulation studies}
\subsection{Simulation Results for Imbalance Vectors under CARs}\label{sec:simulation-CAR}

{\em Two-treatment trials.} We consider CAR procedures for balancing three covariates $X_1,X_2,X_3$ in two-treatment trials. We compare four types of randomization procedures: the complete randomization (CR), the stratified randomization (SR),  \cite{Pocock1975}'s procedure (PS) and $\bm\phi$-based CAR procedure ($\phi$-CAR). In this and subsequent sections, the weights are chosen equally in the PS procedure, and the allocation probability in SR and PS  procedures is chosen to be Efron's biased coin probability with $\rho=0.9$ the same as that in \cite{Ma2022} unless otherwise specified.
 For the $\phi$-CAR, we consider two kinds of allocation probabilities,
one ($\phi$-CAR-BC) is  Efron's biased coin allocation probability with $\rho=0.9$, and the other ($\phi$-CAR-Con) is one generated by a continuous allocation function $1-\Phi((-3)\vee x\wedge 3)$.
Since SR and PS can only deal with discrete covariates, for a continuous $X$ we discretize it to three levels $d(X)=0I\{X\le 0\}+1I\{0< X< 2\}+2I\{X\ge 2\}$. We consider the following six settings:
\begin{description}
  \item[\rm S1] (Indept)  Independent continuous covariates: $X_1,X_2,X_3$ are independent, $X_1\sim N(0,1)$, $X_2,X_3\sim N(1,1)$.  SR, PS with respect to $\widetilde{\bm X}=(d(X_1),d(X_2),d(X_3))$  and $\phi$-CARs with feature $\phi(\bm X)=(1,X_1,X_2,X_3)$ are used.
  \item[\rm S2] (Inter) Continuous covariates with interaction: $X_1,X_2$ are independent, $X_1\sim N(0,1)$, $X_2\sim N(1,1)$ and $X_3=X_1X_2$.  SR, PS with respect to $\widetilde{\bm X}=(d(X_1),d(X_2),d(X_3))$  and $\phi$-CAR with feature $\phi(\bm X)=(1,X_1,X_2,X_3)$ are used.
  \item[\rm S3] (ExpInter) Continuous covariates with exponential interaction: $X_1,X_2$ are independent, $X_1\sim N(0,1)$, $X_2\sim  N(1,1)$ and $X_3=e^{X_1-X_2}-1$.  SR, PS with respect to $\widetilde{\bm X}=(d(X_1),d(X_2),d(X_3))$  and $\phi$-CARs with feature $\phi(\bm X)=(1,X_1,X_2,X_3)$ are used.
  \item[\rm S4] (Mix) Mixing covariates: $X_1,X_2$ are independent, $X_1$ is discrete with $\Prob(X_1=0)=\Prob(X_1=1)=0.5$, $X_2$ is continuous with  $X_2\sim  N(1,1)$ and $X_3=e^{X_1-X_2}-1$.  SR, PS with respect to $\widetilde{\bm X}=(X_1,d(X_2),d(X_3))$  and $\phi$-CARs with feature $\phi(\bm X)=(1,X_1,X_2,X_3)$ are used.
   \item[\rm S5] (Unobs-C) Unobserved continuous covariates: $X_1,X_2, X_3$ are in same as in S3. But now $X_3$ is not observable and,    SR, PS with respect to $\widetilde{\bm X}=(d(X_1),d(X_2))$  and $\phi$-CARs with feature $\phi(\bm X)=(1,X_1,X_2)$ are used.
  \item[\rm S6] (Unobs-M) Unobserved mixing covariates: $X_1,X_2, X_3$ are in same as in S4. But now $X_3$ is not observable and,    SR, PS with respect to $\widetilde{\bm X}=(X_1,d(X_2))$  and $\phi$-CARs with feature $\phi(\bm X)=(1, X_1, X_2)$ are used.

\end{description}
 We also consider a $\phi$-CAR (with  Efron's biased coin allocation probability $\rho=0.9$) by removing the constant $1$ from the feature $\bm \phi$ and denote it by $\phi$-CAR-Ma. We compare four imbalance measures $Imb_0=|\sum_{i=1}^n (2T_i-1)|^2$ and $Imb_j=|\sum_{i=1}^n(2T_i-1)X_{i,j}|^2$, $j=1,2,3$. Since  $\Ex[Imb_j]=n\Ex[X_{i,j}^2]$ under the complete randomization, we normalize $Imb_j$   by $\overline{X_{j}^2}=\frac{1}{n}\sum_{i=1}^n X_{i,j}^2$ which is an estimator of $\Ex[X_{i,j}^2]$.

All of the simulations in this and subsequent sections are
based on 5000 replicates. Simulation results are given in Tables \ref{tab:balance-1-1} and \ref{tab:balance-1-2}. The main findings include, (1) In all cases, all normalized imbalances under the complete randomization  are close to the sample size; (2)
All CAR procedures perform better than the complete randomization; (3) For balancing continuous covariates, the CAR procedures with a feature including the covariates for balancing perform much better than the SR procedure and PS procedure. However, when the constant $1$ is not included in the feature $\bm\phi$, the overall imbalance $|\sum_{i=1}^n(2T_i-1)|^2$ is quarter to half of the sample size (See $Imb_0$ of $\phi$-CAR-Ma); (4)  SR, PS and $\phi$-CAR procedures perform similar for unobserved covariates which are not in $Span\{\bm\phi\}$, and the imbalance is the same order of the sample size (see $Imb_3$ for settings S5 and S6); (5) In most situations, the CAR procedure with a continuous allocation function ($\phi$-CAR-Con) performs better than the one with a discrete allocation function ($\phi$-CAR-BC).

\begin{table}[ht]

	 \centering
'	\caption{Normalized imbalance under  various setting of covariates  and randomization procedures  with  sample size $n=500$: two-treatment case.}\label{tab:balance-1-1}
	\bigskip
\small
	\begin{tabular}{cc   c  c  c  c  c  c  c  c }
		\hline\hline
		  &  	 &\multicolumn{6}{c}{ Settings of Covariates}       \\
\cline{3-8}
          &               & S1 & S2 & S3 &S4 & S5 &S6    \\
Imbalance & Randomization & Indept  & Inter & ExpInter & Mix & Unobs-C  & Unobs-M  \\
	\hline
	&	CR	&	496.2	&	510.2 	&	494.6 	&	511.4	&	489.7 	&	486.7 	\\
	&	SR	&	18.02	&	9.444	&	11.68	&	6.748	&	6.7450	&	4.790	\\
	&	PS	&	1.855 	&	1.934 	&	1.895 	&	1.618 	&	1.933	&	1.532	\\
$Imb_0$	&	$\phi$-CAR-Ma	&	168.9	&	258.8	&	251.0 	&	176.4	&	253.8 	&	178.9 	\\
	&	$\phi$-CAR-BC	&	2.966	&	2.567 	&	8.374	&	2.998	&	1.964	&	1.632	\\
	&	$\phi$-CAR-Con	&	1.896	&	1.619 	&	5.813	&	2.055 	&	1.253	&	1.038 	\\
\hline
	&	CR	&	515.6 	&	499.9	&	482.7 	&	498.2	&	508.7	&	502.9 	\\
	&	SR	&	178.6	&	160.9 	&	158.9 	&	7.651	&	165.5 	&	4.623 	\\
	&	PS	&	160.9 	&	150.4 	&	140.4	&	2.728 	&	163.5	&	2.092 	\\
${Imb_1}/{\overline{X_{1}^2}}$	
   &	$\phi$-CAR-Ma	&	3.389   &	3.068 	&	19.36 	&	11.80 	&	2.218 	&	2.873	\\
	&	$\phi$-CAR-BC	&	4.104 	&	3.810 	&	19.85	&	12.96	&	2.912	&	4.066	\\
	&	$\phi$-CAR-Con	&	2.625	&	2.457	&	14.85	&	8.926	&	1.772	&	2.809 	\\
\hline
	&	CR	&	506.9 	&	506.4	&	496.7 	&	512.3 	&	490.0 	&	483.0 	\\
	&	SR	&	88.47 	&	76.73 	&	65.46	&	51.98	&	76.48	&	70.69 	\\
	&	PS	&	70.91	&	68.59	&	61.53 	&	50.50	&	70.87	&	70.00	\\
${Imb_2}/{\overline{X_{2}^2}}$	
   &	$\phi$-CAR-Ma	&	1.831 	&	2.134	&	13.56	&	2.834 	&	1.287	&	1.121	\\
	&	$\phi$-CAR-BC	&	2.152 	&	2.328	&	10.72	&	3.254 	&	1.550	&	1.378	\\
	&	$\phi$-CAR-Con	&	1.277	&	1.444 	&	7.912 	&	2.209	&	0.9526	&	0.8659 	\\
\hline
	&	CR	&	497.7	&	494.6 	&	499.7	&	492.4	&	492.9 	&	509.5	\\
	&	SR	&	90.31 	&	178.4 	&	209.5 	&	160.3	&	278.5	&	197.3	\\
	&	PS	&	70.58 	&	186.6	&	206.7 	&	157.2	&	321.9	&	239.2 	\\
${Imb_3}/{\overline{X_{3}^2}}$	
   &	$\phi$-CAR-Ma	&	1.784 	&	3.189	&	11.68	&	6.274 	&	350.9 	&	253.3 	\\
	&	$\phi$-CAR-BC	&	2.099 	&	3.396	&	12.63	&	6.948 	&	310.7 	&	228.5	\\
	&	$\phi$-CAR-Con	&	1.308 	&	2.157	&	8.987 	&	4.741 	&	297.4	&	224.8 	\\
\hline
\hline
		\end{tabular}
 \end{table}

\newpage

\begin{table}[ht]

	 \centering
'	\caption{Normalized imbalance under  various setting of covariates  and randomization procedures  with  sample size $n=200$: two-treatment case.}\label{tab:balance-1-2}
	\bigskip
\small
	\begin{tabular}{cc   c  c  c  c  c  c  c  c }
		\hline\hline
		  &  	 &\multicolumn{6}{c}{ Settings of Covariates}       \\
\cline{3-8}
          &               & S1 & S2 & S3 &S4 & S5 &S6    \\
Imbalance & Randomization & Indept  & Inter & ExpInter & Mix & Unobs-C  & Unobs-M  \\
	\hline
	&	CR	&	199.2 	&	193.5	&	198.1 	&	200.2 	&	200.5	&	203.5	\\
	&	SR	&	15.97   &	8.867 	&	10.38 	&	7.114	&	6.214	&	4.785	\\
	&	PS	&	1.840	&	1.870 	&	1.921	&	1.581	&	1.989	&	1.569	\\
$Imb_0$	&	
$\phi$-CAR-Ma	&	67.97	&	104.4 	&	99.74	&	65.68	&	97.63 	&	68.53 	\\
	&	$\phi$-CAR-BC	&	3.014 	&	2.497	&	7.286 	&	3.099 	&	1.921	&	1.634 	\\
	&	$\phi$-CAR-Con	&	1.934 	&	1.590	&	5.445	&	1.937	&	1.294 	&	1.070 	\\
\hline															
	&	CR	&	198.28	&	194.27	&	209.5 	&	203.0	&	196.8 	&	197.8 	\\
	&	SR	&	83.33 	&	74.73	&	67.85	&	7.856 	&	74.54 	&	4.812 	\\
	&	PS	&	72.72	&	66.04	&	61.44	&	2.613 	&	70.04	&	2.105 	\\
${Imb_1}/{\overline{X_{1}^2}}$	
&	$\phi$-CAR-Ma	   &	3.458 	&	3.023 	&	16.90	&	12.46	&	2.179 	&	2.853	\\
	&	$\phi$-CAR-BC	&	4.063	&	3.804 	&	16.43	&	13.12	&	2.901	&	3.995	\\
	&	$\phi$-CAR-Con	&	2.514	&	2.480	&	12.92 	&	8.960	&	1.936	&	2.698	\\
\hline															
	&	CR	&	194.8 	&	193.5 	&	209.2 	&	197.0 	&	206.3 	&	196.3	\\
	&	SR	&	45.08 	&	37.25	&	30.68 	&	25.33	&	34.35 	&	31.14 	\\
	&	PS	&	31.95	&	30.96 	&	27.68 	&	21.39	&	30.42 	&	30.46 	\\
${Imb_2}/{\overline{X_{2}^2}}$	
 &	$\phi$-CAR-Ma	    &	1.836 	&	2.107	&	10.79	&	3.129   &	1.296 	&	1.097 	\\
	&	$\phi$-CAR-BC	&	2.086	&	2.220 	&	8.605 	&	3.581 	&	1.488 	&	1.344	\\
	&	$\phi$-CAR-Con	&	1.290 	&	1.485 	&	6.616 	&	2.377	&	0.9450 	&	0.8688	\\
\hline															
	&	CR	&	196.1	&	197.9 	&	203.5	&	198.8	&	196.8	&	200.1	\\
	&	SR	&	45.84 	&	81.17	&	89.06	&	66.58 	&	120.6	&	75.32 	\\
	&	PS	&	30.91	&	78.85  	&	79.56 	&	64.37 	&	125.5	&	93.96 	\\
${Imb_3}/{\overline{X_{3}^2}}$	
&	$\phi$-CAR-Ma	    &	1.868	&	3.237 	&	9.278 	&	5.640 	&	134.1	&	101.4	\\
	&	$\phi$-CAR-BC	&	2.117	&	3.598 	&	10.05	&	6.051 	&	115.5 	&	88.87 	\\
	&	$\phi$-CAR-Con	&	1.311 	&	2.186 	&	6.987 	&	4.230	&	116.2	&	84.66 	\\
\hline
\hline
		\end{tabular}
 \end{table}

{\em Three-treatment trials.} We consider CAR procedures for balancing three covariates $X_1,X_2,X_3$ in three-treatment trials. We compare four types of randomization procedures: the complete randomization (CR), the stratified randomization (SR),  \cite{Pocock1975}'s procedure (PS), and $\bm\phi$-based CAR procedures ($\phi$-CARs). For SR, PS, $\phi$-CAR-BC and $\phi$-CAR-Ma, we use \cite{Pocock1975}'s allocation probabilities \eqref{eq:PS-allocat} with $\kappa_1=0.8$, $\kappa_2=\kappa_3=0.1$. For $\phi$-CAR-Con, we use the allocation function defined as in \eqref{eq:con-allocat} with $K_0=3$. We compare four imbalance measures $Imb_0=\sum_{t=1}^3|\sum_{i=1}^n (T_i^{(t)}-1/3)|^2$ and $Imb_j=\sum_{t=1}^3|\sum_{i=1}^n (T_i^{(t)}-1/3)X_{i,j}|^2$, $j=1,2,3$. Also, $Imb_0$ is normalized by $(1-1/3)$, $Imb_j$ is normalized by $(1-1/3) \overline{X_{j}^2}$, $j=1,2,3$. Other settings are the same as those in the two-treatment case. Simulation results are given in Tables \ref{tab:balance-2-1} and \ref{tab:balance-2-2}. The findings are the same as those in the two-treatment case.

\begin{table}[ht]
	 \centering
 	\caption{Normalized imbalance under  various setting of covariates  and randomization procedures  with  sample size $n=500$: three-treatment case.}\label{tab:balance-2-1}
	\bigskip
\small
	\begin{tabular}{cc   c  c  c  c  c  c  c  c }
		\hline\hline
		  &  	 &\multicolumn{6}{c}{ Settings of Covariates}       \\
\cline{3-8}
          &               & S1 & S2 & S3 &S4 & S5 &S6    \\
Imbalance & Randomization & Indept  & Inter & ExpInter & Mix & Unobs-C  & Unobs-M  \\
\hline
	&	CR	&	505.4 	&	510.8	&	508.0	&	496.1	&	508.6	&	481.4	\\
	&	SR	&	29.24 	&	16.12 	&	19.54 	&	12.05	&	11.43 	&	8.188	\\
	&	PS	&	3.111	&	3.216	&	3.237	&	2.819	&	3.231	&	2.793 	\\
$\frac{3}{2}Imb_0$	&	
$\phi$-CAR-Ma	&	174.2	&	267.8	&	292.5 	&	160.6	&	260.3	&	172.8 	\\
	&	$\phi$-CAR-BC	&	4.812	&	4.412	&	19.44 	&	5.363 	&	3.397	&	2.804	\\
	&	$\phi$-CAR-Con	&	3.435	&	3.042 	&	12.59 	&	3.857	&	2.511 	&	2.219	\\
\hline
	&	CR	&	510.2 	&	503.9	&	507.2	&	491.9 	&	503.2 	&	481.0	\\
	&	SR	&	193.6	&	173.4	&	161.7	&	13.20 	&	180.4	&	8.306	\\
	&	PS	&	169.8	&	158.2 	&	146.8	&	4.662	&	165.3	&	3.716 	\\
$\frac{3}{2}{Imb_1}/{\overline{X_{1}^2}}$	
   &	$\phi$-CAR-Ma	&	5.308	&	4.920 	&	38.55	&	22.26	&	3.688 	&	4.634 	\\
	&	$\phi$-CAR-BC	&	6.672 	&	6.105 	&	33.74 	&	22.93 	&	4.720	&	6.256	\\
	&	$\phi$-CAR-Con	&	4.478 	&	4.356 	&	25.05	&	15.98	&	3.390	&	5.042	\\
\hline
	&	CR	&	496.8 	&	513.6 	&	504.4 	&	489.6	&	506.9	&	498.5 	\\
	&	SR	&	99.65	&	85.50	&	73.24 	&	56.46	&	79.95	&	74.94 	\\
	&	PS	&	73.60 	&	71.00	&	65.01 	&	51.34	&	71.84	&	72.97 	\\
$\frac{3}{2}{Imb_2}/{\overline{X_{2}^2}}$	
    &	$\phi$-CAR-Ma	&	3.032 	&	3.629 	&	25.28 	&	6.124 	&	2.325 	&	2.178	\\
	&	$\phi$-CAR-BC	&	3.595	&	3.884	&	16.94	&	6.111 	&	2.624 	&	2.376	\\
	&	$\phi$-CAR-Con	&	2.302 	&	2.599 	&	11.92 	&	4.025 	&	1.757 	&	1.614	\\
\hline
	&	CR	&	509.3	&	500.9	&	505.4 	&	504.5 	&	511.6 	&	505.4 	\\
	&	SR	&	100.9	&	193.0 	&	228.9	&	169.5 	&	294.1 	&	199.3 	\\
	&	PS	&	72.85 	&	192.5 	&	211.6	&	166.4	&	320.7 	&	237.7 	\\
$\frac{3}{2}{Imb_3}/{\overline{X_{3}^2}}$	
&	$\phi$-CAR-Ma	&	3.057	&	5.089 	&	20.71	&	11.93 	&	361.2 	&	264.5 	\\
	&	$\phi$-CAR-BC	&	3.392	&	5.657 	&	24.07	&	13.17 	&	315.3 	&	239.0	\\
	&	$\phi$-CAR-Con	&	2.312	&	3.698 	&	15.84 	&	8.686	&	310.3 	&	227.0	\\
\hline
\hline
		\end{tabular}
\end{table}
\newpage

\begin{table}[ht]
	 \centering
	\caption{Normalized imbalance under  various setting of covariates  and randomization procedures  with  sample size $n=200$: three-treatment case.}\label{tab:balance-2-2}
	\bigskip
\small
	\begin{tabular}{cc   c  c  c  c  c  c  c  c }
		\hline\hline
		  &  	 &\multicolumn{6}{c}{ Settings of Covariates}       \\
\cline{3-8}
          &               & S1 & S2 & S3 &S4 & S5 &S6    \\
Imbalance & Randomization & Indept  & Inter & ExpInter & Mix & Unobs-C  & Unobs-M  \\
\hline
	&	CR	&	205.8	&	195.4 	&	201.0 	&	198.1	&	197.1 	&	197.7 	\\
	&	SR	&	25.38	&	14.10	&	16.20 	&	11.44 	&	10.42	&	8.270	\\
	&	PS	&	3.121	&	3.046	&	3.124	&	2.761	&	3.328	&	2.735 	\\
$\frac{3}{2}Imb_0$	&	
$\phi$-CAR-Ma	&	71.62 	&	109.8	&	126.0	&	66.39 	&	103.3	&	70.60	\\
	&	$\phi$-CAR-BC	&	4.731 	&	4.352 	&	16.19	&	5.409 	&	3.259	&	2.867 	\\
	&	$\phi$-CAR-Con	&	3.372 	&	3.036	&	11.26	&	3.749	&	2.422	&	2.271 	\\
\hline
	&	CR	&	201.5	&	201.0 	&	194.7	&	196.5 	&	199.5	&	193.5	\\
	&	SR	&	91.53 	&	78.79 	&	79.37 	&	13.16 	&	77.33	&	8.275	\\
	&	PS	&	73.65	&	70.47 	&	67.26 	&	4.533 	&	73.38	&	3.628	\\
$\frac{3}{2}{Imb_1}/{\overline{X_{1}^2}}$	
&	$\phi$-CAR-Ma	&	5.461 	   &	4.845   	&	24.97 	&	18.89	&	3.698	&	4.805	\\
	&	$\phi$-CAR-BC	&	6.328	&	6.274     	&	22.84	&	20.33 	&	4.735	&	6.584 	\\
	&	$\phi$-CAR-Con	&	4.565 	&	4.289 	    &	18.77  	&	15.08 	&	3.420	&	5.097 	\\
\hline
	&	CR	&	204.7   	&	199.1	&	201.2       &	200.0 	&	202.6 	&	201.0 	\\
	&	SR	&	55.29   	&	45.24   &	36.91 	    &	30.25 	&	39.19 	&	36.77 	\\
	&	PS	&	35.43   	&	34.33 	&	29.95   	&	23.87	&	32.20 	&	33.07 	\\
$\frac{3}{2}{Imb_2}/{\overline{X_{2}^2}}$	
   &	$\phi$-CAR-Ma	&	3.101 	&	3.679 	&	18.39	&	5.741	&	2.246 	&	2.017	\\
	&	$\phi$-CAR-BC	&	3.551	&	4.137	&	12.92 	&	5.769	&	2.558	&	2.452 	\\
	&	$\phi$-CAR-Con	&	2.334	&	2.572	&	9.866	&	3.951 	&	1.801 	&	1.595	\\
\hline
	&	CR	&	201.5	&	203.1	&	198.2	&	198.9	&	202.5	&	199.7 	\\
	&	SR	&	57.62 	&	90.00 	&	106.4	&	72.11	&	126.2	&	82.98 	\\
	&	PS	&	35.30 	&	84.46	&	84.68 	&	70.68	&	128.7	&	98.22	\\
$\frac{3}{2}{Imb_3}/{\overline{X_{3}^2}}$	
   &	$\phi$-CAR-Ma	&	3.009	&	5.047	&	15.19	&	10.47 	&	139.0 	&	101.4	\\
	&	$\phi$-CAR-BC	&	3.467	&	5.820 	&	18.56 	&	11.93 	&	119.7 	&	90.75 	\\
	&	$\phi$-CAR-Con	&	2.271 	&	3.723	&	12.34	&	7.695 	&	117.1	&	86.98	\\
\hline
\hline
		\end{tabular}
\end{table}

\clearpage

\subsection{Simulation Results for Type I Error Rate and Power of Hypothesis Testing under a Regression Model}\label{sec:simulation-test}

In  Section \ref{sec:Inference},  the theories showed that, under most CAR procedures, the traditional test for treatment
effect generates uncorrected type I error rate and usually loses power. If a consistent estimator of the asymptotic variance is used to adjust the test, the test will have precise asymptotic type I error rate and larger power. We verify this phenomenon and compare the power by simulations.

We consider the two settings of underlying models, one of which is the linear model and the other is the nonlinear heteroscedastic model.
\begin{description}
  \item[\rm Setting 1:]  The underlying model is a liner model as follows
 $$ Y_i=T_i\mu_1+(1-T_i)\mu_0+\sum_{j=1}^3 \beta_j X_{i,j}+\epsilon_i$$
 where $\beta_j=1$ for $j = 1, 2,3$, $X_{i,1}\sim N(0, 1)$, $ , X_{i,2}, X_{i,3}\sim N(1, 1)$, and the covariates
are independent of each other. The random error $\epsilon_i\sim  N(0, 2^2)$
is independent of all $X_{i,j}$.
 \item[\rm Setting 2:]  The underlying model is a nonlinear heteroscedastic model as follows
 \begin{align*} Y_i=& T_i\mu_1+(1-T_i)\mu_0+\sum_{j=1}^3 \beta_j X_{i,j}\\
 & \; +T_i\big(g_1(X_{i,1},X_{i,2})+\epsilon_{i,1}\big)+(1-T_i)\big(g_0(X_{i,1},X_{i,2})+\epsilon_{i,0}\big)
 \end{align*}
 where $g_1(x_1,x_2)=e^{\gamma_2 x_2-\gamma_1 x_1-2}$, $g_0(x_1,x_2)=e^{\gamma_2 x_2+\gamma_1 x_1-2}$, $\epsilon_{i,a}=g_a(X_{i,1},X_{i,2})\epsilon_i$, $a=1,0$,
  $X_{i,1}\sim N(0, 1)$, $X_{i,2},X_{i,3}\sim N(1,1)$ and the covariates
are independent of each other, $\beta_1=\beta_2=\beta_3=\gamma_1=\gamma_2=1$. The random error $\epsilon_i\sim  N(0, 1)$
is independent of all $X_{i,j}$. $g_0$ and $g_1$ are not observed both in randomization and in analysis. Under these settings, the conditions for \eqref{eq:nonlinearmodel} are satisfied. In particular, $\Ex[m_1(\bm X) -m_0(\bm X)]=0$.
\end{description}

For the randomization, we use the complete randomization (CR)  procedure and four covariate-adaptive randomization procedures to randomize patients to treatments:   stratified randomization (SR)   with respect to the discretized covariates $\widetilde{\bm X}_i=(d(X_{i,1}), d(X_{i,2}), d(X_{i,3}))$, \cite{Pocock1975}'s procedure (PS) with respect to the discretized covariates $\widetilde{\bm X}_i=(d(X_{i,1}), d(X_{i,2}), d(X_{i,3}))$,   and two $\bm\phi$-based CAR procedures with feature $\bm\phi(\bm X_i)=(1, X_{i,1}, X_{i,2}, X_{i,3})$,  one ($\phi$-CAR-BC) of which has Efron's biased coin allocation \eqref{eq:BCallocation} with $\rho=0.9$,   and the other ($\phi$-CAR-Con) of which has a continuous allocation function \eqref{eq:normallocation} with $D=3$. Here $d(X)=0I\{X\le 0\}+1I\{0< X< 2\}+2I\{X\ge 2\}$.

For each underlying model,  we use the following three working
models to test the treatment effect, that is, $H_0:\mu_1= \mu_0$ and $H_1:\mu_1\ne \mu_0$.
\begin{description}
  \item[\rm W1:] $\Ex[Y_i]=T_i\mu_1+(1-T_i)\mu_0$.
  \item[\rm W2:] $\Ex[Y_i]=T_i\mu_1+(1-T_i)\mu_0+  \beta_1 X_{i,1}$.
  \item[\rm W3:] $\Ex[Y_i]=T_i\mu_1+(1-T_i)\mu_0+\sum_{j=1}^3 \beta_j X_{i,j}$.
\end{description}
When the working model W1 is used,  in the analysis stage only the responses and the results of allocations are observed and so the data is $\{Y_i, T_i,i=1,\ldots,n\}$. When the working model W2 is used,  the observed data in the analysis stage is $\{Y_i, T_i, X_{i,1}, i=1,\ldots,n\}$.
When the working model W3 is used,   the responses, the results of allocations, and the values of all covariates are observed in the analysis stage.

We consider type I error rate under the null hypothesis and powers under alternative hypothesis $\mu_1-\mu_0=\delta/\sqrt{n}$ with $\delta=5,10$ and $15$  of both the traditional test $\mathcal{T}_{LS}(n)$ in \eqref{TS}  and the adjusted tests $\mathcal{T}_{adj}(n)$ with the level of significance $\alpha=0.05$ and the sample size $500$.
For the adjusted tests,  we consider five  kinds of adjusted tests:
\begin{description}
  \item[\sl Adjusted Test $ T_{reg} $.]  The test statistic $\mathcal{T}_{adj}(n)$ is defined in \eqref{TSadj} with $\widehat{\sigma}_{\tau}^2=\widehat{\sigma}_{\tau,reg}^2$ being the regression estimator defined in \eqref{eq:RegEst};
  \item[\sl Adjusted Test $T_{boot}$.]  The test statistic is $\mathcal{T}_{adj}(n)=\widehat{\tau}_n/\sqrt{v_B}$   with $v_B$ being the bootstrap estimator of the asymptotic variance of $\widehat{\tau}_n$;
  \item[\sl Adjusted Test $T_{mb}$.]  The test statistic $\mathcal{T}_{adj}(n)$ is defined in \eqref{TSadj} with $\widehat{\sigma}_{\tau}^2=\widehat{\sigma}^2_{\tau,mb}$ being the moving block estimator defined in \eqref{eq:mbEst}  with block size $l=integer(\sqrt{n})$;
  \item[\sl Adjusted Test $T_{bmj}$.]  The test statistic is  $\mathcal{T}_{adj}(n)=\widehat{\tau}_n/\widehat{\sigma}_{mbj}$  with $\widehat{\sigma}_{mbj}^2$ being the moving block jackknife estimator  of the variance of $\widehat{\tau}_n$ defined in \eqref{eq:mbjest}  with block size $l=integer(\sqrt{n})$;
  \item[\sl Adjusted Test $T_{mbb}$.]  The test statistic is  $\mathcal{T}_{adj}(n)=\widehat{\tau}_n/\widehat{\sigma}_{mbb}$  with $\widehat{\sigma}_{mbb}^2$ being the moving block bootstrap estimator of the variance of $\widehat{\tau}_n$  with block size $l=integer(\sqrt{n})$.
\end{description}
 The bootstrap size in Adjusted Tests $T_{boot}$ and $T_{mbb}$ is chosen to be $500$.  We use the "lm" package of the R software to fit the linear model, "boot" and "tsboot" packages to obtain the bootstrap estimator and the moving block bootstrap estimator.
 The simulation results are given in Tables \ref{tab:compareB-1}--\ref{tab:compareB-3}.  The main findings include

 \begin{description}
   \item[\rm (1) ]  Under the full model W3, the traditional $\mathcal{T}_{LS}(n)$ has similar type I error rates and powers for all randomization procedures;
   \item[\rm (2) ]   For all CAR procedures, under working models W1 and W2 the traditional test $\mathcal{T}_{LS}(n)$ is conservative  and losses power compared with the test under the full model W3;
   \item[\rm (3)]  When $\delta$ is small,  the power of the traditional test $\mathcal{T}_{LS}(n)$ under CAR procedures  is less than that under the complete randomization and, as the value of $\delta$ increases,  the power under CAR procedures increases to be similar to or exceed the power under the complete randomization;
   \item[\rm (4) ] When a consistent estimator of $\sigma_{\tau}^2$ is used to restore the correct Type I error rate, the power of the adjusted tests under the CAR procedures increases a lot so that it is comparable with the power under the full model W3 and, the test under the CAR procedure with feature $\bm\phi=(1, X_1, X_2, X_3)$ is more powerful than that under SR or PS;
   \item[\rm (5) ] When   $\bm\phi(\bm X_i)$s for balancing are observed or the CAR procedure is known at the analysis stage, the  test adjusted by the regression estimator $T_{reg}$  or the bootstrap estimator $T_{boot}$ can restore type I error rate well in most cases under the CAR procedures;
   \item[\rm (6)]  When only the data under the working model are observed,  the test adjusted by the moving block estimator $T_{mb}$, the moving block jackknife estimator $T_{mbj}$ or the moving block bootstrap estimator $T_{mbb}$  can only restore part of the type I error rate, and the type I error rate under the full model W3 are usually inflated a little, because of the slow convergence of estimators by the moving blocking methods.
 \end{description}
The simulations of the bootstrap estimator and the moving block bootstrap estimator require huge account calculations. To obtain the simulation results in the tables for Adjusted Test  $T_{boot}$ and Adjusted Test $T_{mbb}$ with $5000$ replicates, decades of days are needed. The simulations with $1000$ replicates are also considered. The simulation results are given in Tables \ref{tab:compareC-1}--\ref{tab:compareC-3}.
\clearpage

\begin{table}[ht]
\centering
	\caption{Type I error rates and powers of traditional test and the adjusted tests for treatment effect under the linear underlying model
(Setting 1),  sample size $n=500$ and 5000 replicates.}\label{tab:compareB-1}
	\bigskip
\small
	\begin{tabular}{cc   c c c  c  c  c c  c  c c c      }
		\hline\hline
		\multicolumn{2}{c}{$\mu_1-\mu_0=\delta/\sqrt{n}$}  	 &\multicolumn{3}{c}{ Traditional Test } & &\multicolumn{3}{c}{ Adjusted Test $ T_{reg} $ }& &\multicolumn{3}{c}{ Adjusted Test $T_{boot}$ }   \\
\cline{3-5} \cline{7-9}\cline{11-13}
$ \delta$ & Randomization & W1 & W2 & W3 & & W1 & W2 & W3  & & W1 & W2 & W3  \\
\hline
 	&	CR	&	0.048	&	0.045 	&	0.055 	&	&	--	&	--	&	--	&	&	--	&	--	&	--
\\
	&	SR	&	0.023 	&	0.026	&	0.053	&	&	0.058 	&	0.051 	&	0.048	&	&	0.051	& 0.051 	& 0.051 	
\\
0	&	PS	&	0.020 	&	0.023 	&	0.050 	&	&	0.051	&	0.051	&	0.048 	&	&	0.054	& 0.051 	& 0.048

\\
	&	$\phi$-CAR-BC	&	0.010	&	0.018	&	0.049	&	&	0.050 	&	0.051	&	0.051 	&	&	0.053	& 0.053 	& 0.053
\\
	&	$\phi$-CAR-Con	&	0.013	&	0.020 	&	0.055 	&	&	0.057	&	0.055 	&	0.055	&	&	0.053	& 0.054 	& 0.053
	\\
\hline
	&	CR	&	0.159 	&	0.177 	&	0.244	&	&--		&	--	&	--	&	&	--	&	--	&	--	
\\
	&	SR	&	0.114	&	0.144	&	0.233 	&	&	0.211 	&	0.220	&	0.245 	&	&	0.208	& 0.216 	& 0.242	
\\
5	&	PS	&	0.108 	&	0.137	&	0.233	&	&	0.203 	&	0.219	&	0.234 	&	&	0.207	& 0.222  &	0.242
 	
\\
	&	$\phi$-CAR-BC	&	0.083	&	0.119 	&	0.237 	&	&	0.247 	&	0.245 	&	0.243	&	&	0.241	& 0.238 	& 0.238 	
\\
	&	$\phi$-CAR-Con	&	0.084 	&	0.122	&	0.234 	&	&	0.243 	&	0.241 	&	0.239 	&	&	0.233	& 0.233     &  0.234	
\\
\hline
	&	CR	&	0.465 	&	0.517 	&	0.699 	&	&--		&	--	&	--	&	&	--	&	--	&	--
\\
	&	SR	&	0.448	&	0.529 	&	0.694 	&	&	0.626 	&	0.658 	&	0.708	&	&	0.607	& 0.647 &	0.705 	
\\
10	&	PS	&	0.465	&	0.533	&	0.700 	&	&	0.627	&	0.654 	&	0.709	&	&	0.614 	& 0.643 &	0.696	
\\
	&	$\phi$-CAR-BC	&	0.465	&	0.542	&	0.710	&	&	0.705 	&	0.703	&	0.704 	&	&	0.708	& 0.709 &	0.712 	
\\
	&	$\phi$-CAR-Con	&	0.448 	&	0.523	&	0.698	&	&	0.710 	&	0.710	&	0.709 	&	&	0.691 	& 0.692	& 0.695
	\\
\hline
	&	CR	&	0.802 	&	0.857 	&	0.959 	&	&	--	&	--	&	--	&	&	--	&	--	&	--
\\
	&	SR	&	0.847 	&	0.893	&	0.960 	&	&	0.924 	&	0.935 	&	0.964	&	&	0.912	& 0.928	& 0.959
\\
15	&	PS	&	0.855	&	0.898 	&	0.962	&	&	0.929 	&	0.945	&	0.962 	&	&	0.921 	& 0.940 	& 0.961 	
\\
	&	$\phi$-CAR-BC	&	0.882 	&	0.914 	&	0.967	&	&	0.963 	&	0.965	&	0.964 	&	&	0.960 &	0.961 & 	0.962	
\\
	&	$\phi$-CAR-Con	&	0.878	&	0.913 	&	0.964	&	&	0.969	&	0.969	&	0.969	&	&	0.957	& 0.957  &	0.957 	
\\
\hline
\hline
		  &  	 &\multicolumn{3}{c}{ Adjusted Test $T_{mb}$ }& &\multicolumn{3}{c}{ Adjusted Test $T_{mbj}$ } & &\multicolumn{3}{c}{ Adjusted Test $T_{mbb}$ }   \\
\cline{3-5} \cline{7-9}\cline{11-13}
   &   & W1 & W2 & W3  & & W1 & W2 & W3 & & W1 & W2 & W3  \\
\cline{3-13}
	&	SR	&		0.034 	&	0.040 	&	0.064	&	&	0.032	&	0.035 	&	0.053 & & 0.038 	& 0.042 	& 0.061
\\
0	&	PS	&    	0.042 	&	0.046	&	0.063 	&	&	0.032	&	0.038	&	0.056 & & 0.037     & 0.044	 & 0.066

\\
	&	$\phi$-CAR-BC	&	 	0.031 	&	0.037	&	0.060 	&	&	0.027	&	0.036	&	0.058 & & 0.030	& 0.035 	& 0.065
\\
	&	$\phi$-CAR-Con	&	 	0.035 	&	0.042	&	0.061  &	&	0.028	&	0.033	&	0.053 & & 0.038 	& 0.041 	& 0.070
	\\
\hline
	&	SR	&	0.162	&	0.181	&	0.263	&	&	0.146	&	0.169	&	0.253 & & 0.156	    & 0.181	& 0.255
\\
5	&	PS	&	0.154 	&	0.177 	&	0.257 	&	&	0.154	&	0.179 	&	0.253 & & 0.164 	& 0.185	& 0.266

\\
	&	$\phi$-CAR-BC	&		0.151 	&	0.169	&	0.245 	&	&	0.162	&	0.178	&	0.247 & & 0.164 	& 0.182	& 0.260
\\
	&	$\phi$-CAR-Con	&		0.190 	&	0.204	&	0.261 	&	&	0.172	&	0.185 	&	0.244 & & 0.192 	& 0.205 &	0.268
\\
\hline
	&	SR	&	0.527 	&	0.571 	&	0.714	&	&	0.510	&	0.562	&	0.701 & & 0.517 	& 0.577	& 0.712
\\
10	&	PS	&	0.551	&	0.595	&	0.722 	&	&	0.529	&	0.581 	&	0.709 & & 0.543 	& 0.591	& 0.718

\\
	&	$\phi$-CAR-BC	&	0.607 	&	0.635 	&	0.717	&	&	0.581 	&	0.607 	&	0.697 & & 0.595 	& 0.622 	& 0.716
\\
	&	$\phi$-CAR-Con	&	0.627	&	0.647 	&	0.717   &	&	0.616 	&	0.640	&	0.718 & & 0.624 &	0.645	& 0.717
	\\
\hline
	&	SR	&	0.868 	&	0.899 	&	0.959	&	&	0.857	&	0.892	&	0.955 & & 0.874	& 0.908	& 0.961
\\
15	&	PS	&	0.886	&	0.914 	&	0.961 	&	&	0.882	&	0.911	&	0.960 & & 0.887    & 0.911	& 0.961

\\
	&	$\phi$-CAR-BC	&	0.925 	&	0.931	&	0.959	&	&	0.926	&	0.936 	&	0.963 & & 0.930 &	0.937 	& 0.962
\\
	&	$\phi$-CAR-Con	&	0.938	&	0.944	&	0.965	&	&	0.924 	&	0.930	&	0.954 & & 0.936  &	0.943	& 0.964
\\
\hline
\hline
		\end{tabular}
\end{table}

\clearpage

\begin{table}[ht]
	 \centering
	\caption{Type I error rates and powers of the traditional test and adjusted test for treatment effect
 under the heteroscedasticity nonlinear underlying model (Setting 2),  sample size $n=500$ and 5000 replicates.}\label{tab:compareB-3}
	\bigskip
\small
	\begin{tabular}{cc   c c c  c  c  c c  c  c c c  c  }
		\hline\hline
\multicolumn{2}{c}{$\mu_1-\mu_0=\delta/\sqrt{n}$} 	 &\multicolumn{3}{c}{ Traditional Test } & &\multicolumn{3}{c}{ Adjusted Test $  T_{reg}  $ }& &\multicolumn{3}{c}{ Adjusted Test $T_{boot}$ }   \\
\cline{3-5} \cline{7-9}\cline{11-13}
$ \delta$ & Randomization & W1 & W2 & W3 & & W1 & W2 & W3  & & W1 & W2 & W3   \\
\hline
	&	CR	&	0.046	&	0.048	&	0.042	&	&	--	&	--	&	--	&	&	--	&	--	&	--
\\
	&	SR	&	0.015 	&	0.018 	&	0.042 	&	&	0.079 	&	0.072 	&	0.067	&	&	0.043 	& 0.042	& 0.039
	\\
0	&	PS	&	0.014 	&	0.017 	&	0.041	&	&	0.047 	&	0.046 	&	0.047 	&	&  0.047 	& 0.048	& 0.044
	\\
 	&	$\phi$-CAR-BC	&	0.006 	&	0.009 	&	0.047 	&	&	0.046	&	0.046 	&	0.042 	&	&	 0.046	& 0.046	& 0.046
	\\
	&	$\phi$-CAR-Con	&	0.006	&	0.009	&	0.040	&	&	0.044 	&	0.043	&	0.042 	&	&	 0.042	& 0.042	& 0.041
\\
\hline
	&	CR	&	0.098 	&	0.109 	&	0.156 	&	&	--	&	--	&	--	&	&	 --	&	--	&	--
\\
	&	SR	&	0.052	&	0.072	&	0.153	&	&	0.175	&	0.181	&	0.206	&	&	 0.123 	& 0.133 	& 0.151
\\
5	&	PS	&	0.048 	&	0.068	&	0.159	&	&	0.148	&	0.156 	&	0.175 	&	&	 0.132 	& 0.147	& 0.161
	\\
 	&	$\phi$-CAR-BC	&	0.040 	&	0.059 	&	0.168	&	&	0.172 	&	0.170 	&	0.169 	&	&	 0.146 	& 0.148	& 0.151
\\
	&	$\phi$-CAR-Con	&	0.031 	&	0.051	&	0.156 	&	&	0.164	&	0.164 	&	0.162 	&	&	 0.146	& 0.146 	& 0.147
\\
\hline
	&	CR	&	0.263 	&	0.295 	&	0.439 	&	&	--	&	--	&	--	&	&	  --	&	--	&	--
\\
	&	SR	&	0.222 	&	0.268 	&	0.439	&	&	0.443 	&	0.459	&	0.507 	&	& 0.354 	& 0.371	& 0.432
	\\
10	&	PS	&	0.235	&	0.272 	&	0.450 	&	&	0.404 	&	0.417	&	0.465	&	&	 0.386	& 0.397 	& 0.461
\\
   &	$\phi$-CAR-BC	&	0.214	&	0.255 	&	0.435 	&	&	0.448 	&	0.448	&	0.444 	&	&	 0.427	& 0.427	& 0.431
\\
	&	$\phi$-CAR-Con	&	0.210	&	0.254	&	0.443 	&	&	0.459 	&	0.459 	&	0.458 	&	&	 0.437	& 0.435 	& 0.439
\\
\hline
	&	CR	&	0.514	&	0.538 	&	0.708	&	&	--	&	--	&	--	&	&	--	&	--	&	--
\\
	&	SR	&	0.512	&	0.546 	&	0.697	&	&	0.694	&	0.701	&	0.750 	&	&	 0.631 	& 0.637	& 0.660
\\
15	&	PS	&	0.521	&	0.555	&	0.705	&	&	0.678	&	0.687	&	0.732 	&	&	 0.652 	& 0.654 	& 0.713
\\
 	&	$\phi$-CAR-BC	&	0.549 	&	0.583 	&	0.723 	&	&	0.717	&	0.716 	&	0.715 	&	&	0.695 	& 0.694	& 0.700
\\
	&	$\phi$-CAR-Con	&	0.539  	&	0.568	&	0.713 	&	&	0.710	&	0.711	&	0.712 	&	&	 0.694	& 0.695	& 0.698
\\
\hline
\hline
&  	 & \multicolumn{3}{c}{ Adjusted Test $T_{mb}$ }& &\multicolumn{3}{c}{ Adjusted Test $T_{mbj}$ } & &\multicolumn{3}{c}{ Adjusted Test $T_{mbb}$ } \\
\cline{3-5} \cline{7-9}\cline{11-13}
  &  & W1 & W2 & W3 & & W1 & W2 & W3  & & W1 & W2 & W3    \\
\cline{3-13}
	&	SR	&   0.028	&	0.031 	&	0.054  &	&	0.025	&	0.026 	&	0.050 &	&	0.031	&	0.033	&	0.060
	\\
0	&	PS	&	0.030 	&	0.035 	&	0.061  &	&	0.025	&	0.027	&	0.046 &	&	0.031	&	0.038	&	0.063
	\\
 	&	$\phi$-CAR-BC	&	0.024 	&	0.027	&	0.054  &	&	0.022 	&	0.026	&	0.054 &	&	0.021 	&	0.025	&	0.052
	\\
	&	$\phi$-CAR-Con	&	0.027	&	0.029 	&	0.057	&	&	0.023 	&	0.026	&	0.049 &	&	0.026	&	0.030	&	0.055
\\
\hline
	&	SR	&	0.078 	&	0.095 	&	0.170	&	&	0.070	&	0.089 	&	0.152 &	&	0.086	&	0.104 	&	0.172
\\
5	&	PS	&	0.091	&	0.112 	&	0.182  &	&	0.085	&	0.099	&	0.173 &	&	0.089 	&	0.107 	&	0.182
	\\
 	&	$\phi$-CAR-BC	&	0.095	&	0.106	&	0.176	&	&	0.078 	&	0.090	&	0.156 & & 	0.101 	&	0.114 	&	0.185
\\
	&	$\phi$-CAR-Con	&	0.120 	&	0.129 	&	0.193 	&	&	0.096 	&	0.105	&	0.165 &	&  0.109	&	0.117 	&	0.176
\\
\hline
	&	SR	&	0.289	&	0.317	&	0.457 	&	&	0.276 	&	0.303	&	0.450 &	&	0.280 	&	0.321	&	0.465
\\
10	&	PS	&   0.313 	&	0.345	&	0.475 	&	&	0.295	&	0.324	&	0.459 &	&	0.304	&	0.335	&	0.463
\\
 	&	$\phi$-CAR-BC	&	0.349 	&	0.365 	&	0.475	&	&	0.318	&	0.335	&	0.435 &	&	0.330 	&	0.346	&	0.446
\\
	&	$\phi$-CAR-Con	&	0.358	&	0.372	&	0.457	&	&	0.345	&	0.356 	&	0.437 &	&	0.358	&	0.370	&	0.455
\\
\hline
	&	SR	&	0.564	&	0.586 	&	0.727	&	&	0.568	&	0.585 	&	0.714 &	&	0.574	&	0.602	&	0.735
\\
15	&	PS	&	0.592	&	0.614 	&	0.725 	&	&	0.587	&	0.606	&	0.721 &	&	0.593 	&	0.616 	&	0.732
\\
 	&	$\phi$-CAR-BC	&	0.644	&	0.654 	&	0.724 	&	&	0.622	&	0.631	&	0.704 &	&	0.636	&	0.641 	&	0.727
\\
	&	$\phi$-CAR-Con	&	0.648	&	0.654	&	0.711 	&	&	0.641	&	0.646	&	0.707 &	&	0.656	&	0.663	&	0.723
\\
\hline
\hline
		\end{tabular}
\end{table}
\clearpage

\clearpage

\begin{table}[ht]
\centering
	\caption{Type I error rates and powers of the traditional test and adjusted tests for treatment effect under the linear underlying model
(Setting 1),  sample size $n=500$ and 1000 replicates.}\label{tab:compareC-1}
	\bigskip
\small
	\begin{tabular}{cc   c c c  c  c  c c  c  c c c      }
		\hline\hline
\multicolumn{2}{c}{$\mu_1-\mu_0=\delta/\sqrt{n}$}  	 &\multicolumn{3}{c}{ Traditional Test } & &\multicolumn{3}{c}{ Adjusted Test $  T_{reg}  $ }& &\multicolumn{3}{c}{ Adjusted Test $T_{boot}$ }   \\
\cline{3-5} \cline{7-9}\cline{11-13}
$ \delta$ & Randomization & W1 & W2 & W3 & & W1 & W2 & W3  & & W1 & W2 & W3  \\
\hline
 		&	CR	&	0.032	&	0.040	&	0.041	&	&	-	&	-	&	-	&	&	-	&	-	&	-	\\
	&	SR	&	0.023	&	0.026	&	0.058	&	&	0.053	&	0.053	&	0.043	&	&	0.059	&	0.062	&	0.052	\\
0	&	PS	&	0.017	&	0.026	&	0.044	&	&	0.057	&	0.049	&	0.061	&	&	0.045	&	0.047	&	0.043	\\
	&	$\phi$-CAR-BC	&	0.011	&	0.025	&	0.058	&	&	0.058	&	0.056	&	0.054	&	&	0.054	&	0.052	&	0.051	\\
	&	$\phi$-CAR-Con	&	0.012	&	0.015	&	0.045	&	&	0.058	&	0.059	&	0.059	&	&	0.050	&	0.049	&	0.050	\\
		\hline																					
	&	CR	&	0.156	&	0.164	&	0.232	&	&	-	&	-	&	-	&	&	-	&	-	&	-	\\
	&	SR	&	0.131	&	0.146	&	0.268	&	&	0.215	&	0.222	&	0.223	&	&	0.201	&	0.227	&	0.241	\\
5	&	PS	&	0.106	&	0.138	&	0.234	&	&	0.208	&	0.222	&	0.235	&	&	0.212	&	0.232	&	0.253	\\
	&	$\phi$-CAR-BC	&	0.103	&	0.142	&	0.247	&	&	0.241	&	0.250	&	0.243	&	&	0.262	&	0.264	&	0.266	\\
	&	$\phi$-CAR-Con	&	0.090	&	0.121	&	0.209	&	&	0.240	&	0.238	&	0.237	&	&	0.236	&	0.235	&	0.239	\\
		\hline																					
	&	CR	&	0.443	&	0.511	&	0.681	&	&	-	&	-	&	-	&	&	-	&	-	&	-	\\
	&	SR	&	0.463	&	0.541	&	0.721	&	&	0.603	&	0.634	&	0.683	&	&	0.614	&	0.635	&	0.701	\\
10	&	PS	&	0.473	&	0.542	&	0.710	&	&	0.630	&	0.665	&	0.712	&	&	0.628	&	0.640	&	0.681	\\
	&	$\phi$-CAR-BC	&	0.495	&	0.552	&	0.719	&	&	0.703	&	0.703	&	0.704	&	&	0.690	&	0.691	&	0.711	\\
	&	$\phi$-CAR-Con	&	0.456	&	0.530	&	0.691	&	&	0.689	&	0.688	&	0.686	&	&	0.703	&	0.701	&	0.707	\\
		\hline																					
	&	CR	&	0.828	&	0.867	&	0.961	&	&	-	&	-	&	-	&	&	-	&	-	&	-	\\
	&	SR	&	0.837	&	0.893	&	0.962	&	&	0.916	&	0.928	&	0.954	&	&	0.924	&	0.940	&	0.970	\\
15	&	PS	&	0.874	&	0.908	&	0.973	&	&	0.923	&	0.943	&	0.960	&	&	0.926	&	0.936	&	0.960	\\
	&	$\phi$-CAR-BC	&	0.855	&	0.894	&	0.963	&	&	0.972	&	0.972	&	0.974	&	&	0.961	&	0.961	&	0.964	\\
	&	$\phi$-CAR-Con	&	0.850	&	0.895	&	0.964	&	&	0.959	&	0.959	&	0.960	&	&	0.955	&	0.955	&	0.954	\\
\hline\hline
		  &  	 &\multicolumn{3}{c}{ Adjusted Test $T_{mb}$ }& &\multicolumn{3}{c}{ Adjusted Test $T_{mbj}$ } & &\multicolumn{3}{c}{ Adjusted Test $T_{mbb}$ }   \\
\cline{3-5} \cline{7-9}\cline{11-13}
 &  & W1 & W2 & W3 & & W1 & W2 & W3  & & W1 & W2 & W3      \\
\cline{3-13}
	&	SR	&	0.038	&	0.046	&	0.066	&	&	0.041	&	0.038	&	0.059	&	&	0.042	&	0.041	&	0.070	\\
0	&	PS	&	0.043	&	0.045	&	0.067	&	&	0.036	&	0.037	&	0.067	&	&	0.041	&	0.047	&	0.059	\\
	&	$\phi$-CAR-BC	&	0.031	&	0.038	&	0.057	&	&	0.035	&	0.041	&	0.063	&	&	0.032	&	0.044	&	0.071	\\
	&	$\phi$-CAR-Con	&	0.040	&	0.044	&	0.060	&	&	0.034	&	0.037	&	0.063	&	&	0.034	&	0.036	&	0.055	\\
		\hline																					
	&	SR	&	0.167	&	0.187	&	0.262	&	&	0.144	&	0.171	&	0.260	&	&	0.174	&	0.202	&	0.296	\\
5	&	PS	&	0.153	&	0.178	&	0.265	&	&	0.159	&	0.164	&	0.239	&	&	0.177	&	0.191	&	0.276	\\
	&	$\phi$-CAR-BC	&	0.168	&	0.183	&	0.259	&	&	0.154	&	0.171	&	0.245	&	&	0.189	&	0.209	&	0.276	\\
	&	$\phi$-CAR-Con	&	0.173	&	0.186	&	0.251	&	&	0.189	&	0.203	&	0.256	&	&	0.195	&	0.208	&	0.264	\\
		\hline																					
	&	SR	&	0.522	&	0.592	&	0.721	&	&	0.519	&	0.571	&	0.701	&	&	0.529	&	0.575	&	0.712	\\
10	&	PS	&	0.544	&	0.588	&	0.720	&	&	0.524	&	0.55	&	0.717	&	&	0.557	&	0.595	&	0.708	\\
	&	$\phi$-CAR-BC	&	0.601	&	0.622	&	0.725	&	&	0.581	&	0.613	&	0.714	&	&	0.578	&	0.611	&	0.702	\\
	&	$\phi$-CAR-Con	&	0.614	&	0.640	&	0.709	&	&	0.623	&	0.641	&	0.709	&	&	0.633	&	0.651	&	0.720	\\
		\hline																					
	&	SR	&	0.87	&	0.902	&	0.956	&	&	0.868	&	0.906	&	0.954	&	&	0.876	&	0.906	&	0.970	\\
15	&	PS	&	0.887	&	0.914	&	0.966	&	&	0.878	&	0.917	&	0.968	&	&	0.881	&	0.903	&	0.958	\\
	&	$\phi$-CAR-BC	&	0.926	&	0.935	&	0.964	&	&	0.921	&	0.930	&	0.960	&	&	0.924	&	0.931	&	0.965	\\
	&	$\phi$-CAR-Con	&	0.939	&	0.944	&	0.961	&	&	0.941	&	0.945	&	0.971	&	&	0.932	&	0.943	&	0.960	\\
\hline
\hline
		\end{tabular}
\end{table}

\clearpage

\begin{table}[ht]
	 \centering
	\caption{Type I error rates and powers of the traditional test and adjusted test for treatment effect
 under the heteroscedasticity nonlinear underlying model  (Setting 2),  sample size $n=500$ and 1000 replicates.}\label{tab:compareC-3}
	\bigskip
\small
	\begin{tabular}{cc   c c c  c  c  c c  c  c c c  c  }
		\hline\hline
 \multicolumn{2}{c}{$\mu_1-\mu_0=\delta/\sqrt{n}$}		 &\multicolumn{3}{c}{ Traditional Test } & &\multicolumn{3}{c}{ Adjusted Test $  T_{reg}  $ }& &\multicolumn{3}{c}{ Adjusted Test $T_{boot}$ }   \\
\cline{3-5} \cline{7-9}\cline{11-13}
$ \delta$ & Randomization & W1 & W2 & W3 & & W1 & W2 & W3  & & W1 & W2 & W3   \\
\hline
	&	CR	&	0.045	&	0.045	&	0.037	&	&	-	&	-	&	-	&	&	-	&	-	&	-	\\
	&	SR	&	0.025	&	0.024	&	0.045	&	&	0.068	&	0.064	&	0.052	&	&	0.040	&	0.040	&	0.040	\\
0	&	PS	&	0.017	&	0.012	&	0.035	&	&	0.056	&	0.055	&	0.048	&	&	0.048	&	0.049	&	0.039	\\
	&	$\phi$-CAR-BC	&	0.008	&	0.011	&	0.041	&	&	0.039	&	0.040	&	0.038	&	&	0.047	&	0.047	&	0.053	\\
	&	$\phi$-CAR-Con	&	0.006	&	0.007	&	0.050	&	&	0.045	&	0.043	&	0.044	&	&	0.033	&	0.035	&	0.035	\\
		\hline																					
	&	CR	&	0.083	&	0.090	&	0.141	&	&	-	&	-	&	-	&	&	-	&	-	&	-	\\
	&	SR	&	0.050	&	0.072	&	0.154	&	&	0.181	&	0.190	&	0.211	&	&	0.103	&	0.118	&	0.136	\\
5	&	PS	&	0.059	&	0.080	&	0.149	&	&	0.146	&	0.151	&	0.160	&	&	0.148	&	0.145	&	0.161	\\
	&	$\phi$-CAR-BC	&	0.032	&	0.055	&	0.152	&	&	0.169	&	0.175	&	0.167	&	&	0.168	&	0.166	&	0.161	\\
	&	$\phi$-CAR-Con	&	0.034	&	0.064	&	0.166	&	&	0.160	&	0.159	&	0.165	&	&	0.157	&	0.158	&	0.157	\\
		\hline																					
	&	CR	&	0.257	&	0.287	&	0.446	&	&	-	&	-	&	-	&	&	-	&	-	&	-	\\
	&	SR	&	0.246	&	0.279	&	0.444	&	&	0.428	&	0.445	&	0.497	&	&	0.385	&	0.403	&	0.464	\\
10	&	PS	&	0.235	&	0.284	&	0.468	&	&	0.402	&	0.402	&	0.466	&	&	0.353	&	0.374	&	0.435	\\
	&	$\phi$-CAR-BC	&	0.212	&	0.255	&	0.407	&	&	0.412	&	0.413	&	0.413	&	&	0.416	&	0.418	&	0.423	\\
	&	$\phi$-CAR-Con	&	0.209	&	0.268	&	0.449	&	&	0.459	&	0.456	&	0.457	&	&	0.446	&	0.452	&	0.453	\\
		\hline																					
	&	CR	&	0.529	&	0.556	&	0.733	&	&	-	&	-	&	-	&	&	-	&	-	&	-\\
	&	SR	&	0.510	&	0.547	&	0.721	&	&	0.705	&	0.721	&	0.762	&	&	0.615	&	0.625	&	0.692	\\
15	&	PS	&	0.525	&	0.558	&	0.716	&	&	0.683	&	0.696	&	0.742	&	&	0.660	&	0.662	&	0.699	\\
	&	$\phi$-CAR-BC	&	0.553	&	0.592	&	0.740	&	&	0.710	&	0.710	&	0.711	&	&	0.692	&	0.691	&	0.691	\\
	&	$\phi$-CAR-Con	&	0.533	&	0.554	&	0.680	&	&	0.711	&	0.713	&	0.710	& &		0.694	&	0.696	&	0.699	\\
\hline
\hline
&  	 & \multicolumn{3}{c}{ Adjusted Test $T_{mb}$ }& &\multicolumn{3}{c}{ Adjusted Test $T_{mbj}$ } & &\multicolumn{3}{c}{ Adjusted Test $T_{mbb}$ } \\
\cline{3-5} \cline{7-9}\cline{11-13}
 &  & W1 & W2 & W3 & & W1 & W2 & W3  & & W1 & W2 & W3      \\
\cline{3-13}
	&	SR	&	0.024	&	0.033	&	0.049	&	&	0.023	&	0.027	&	0.063	&	&	0.023	&	0.031	&	0.044	\\
0	&	PS	&	0.029	&	0.039	&	0.057	&	&	0.029	&	0.032	&	0.048	&	&	0.026	&	0.025	&	0.059	\\
	&	$\phi$-CAR-BC	&	0.030	&	0.036	&	0.066	&	&	0.021	&	0.026	&	0.055	&	&	0.025	&	0.035	&	0.062	\\
	&	$\phi$-CAR-Con	&	0.030	&	0.031	&	0.055	&	&	0.027	&	0.027	&	0.046	&	&	0.031	&	0.031	&	0.068	\\
		\hline																					
	&	SR	&	0.082	&	0.099	&	0.169	&	&	0.070	&	0.084	&	0.143	&	&	0.092	&	0.106	&	0.180	\\
5	&	PS	&	0.087	&	0.104	&	0.158	&	&	0.075	&	0.091	&	0.163	&	&	0.092	&	0.114	&	0.194	\\
	&	$\phi$-CAR-BC	&	0.096	&	0.113	&	0.173	&	&	0.085	&	0.093	&	0.154	&	&	0.092	&	0.101	&	0.173	\\
	&	$\phi$-CAR-Con	&	0.094	&	0.108	&	0.174	&	&	0.094	&	0.103	&	0.146	&	&	0.111	&	0.121	&	0.179	\\
		\hline																					
	&	SR	&	0.307	&	0.326	&	0.470	&	&	0.258	&	0.286	&	0.445	&	&	0.286	&	0.311	&	0.457	\\
10	&	PS	&	0.312	&	0.356	&	0.478	&	&	0.280	&	0.308	&	0.450	&	&	0.294	&	0.325	&	0.451	\\
	&	$\phi$-CAR-BC	&	0.344	&	0.354	&	0.471	&	&	0.322	&	0.340	&	0.447	&	&	0.333	&	0.352	&	0.476	\\
	&	$\phi$-CAR-Con	&	0.377	&	0.387	&	0.478	&	&	0.330	&	0.341	&	0.424	&	&	0.366	&	0.373	&	0.455	\\
		\hline																					
	&	SR	&	0.578	&	0.596	&	0.720	&	&	0.574	&	0.601	&	0.727	&	&	0.565	&	0.592	&	0.709	\\
15	&	PS	&	0.613	&	0.629	&	0.749	&	&	0.586	&	0.605	&	0.708	&	&	0.583	&	0.620	&	0.760	\\
	&	$\phi$-CAR-BC	&	0.628	&	0.641	&	0.703	&	&	0.622	&	0.626	&	0.716	&	&	0.626	&	0.635	&	0.711	\\
	&	$\phi$-CAR-Con	&	0.660	&	0.666	&	0.725	&	&	0.616	&	0.619	&	0.701	&	&	0.657	&	0.662	&	0.732	\\
\hline
\hline
		\end{tabular}
\end{table}

\clearpage

 \subsection{Simulation Results for Type I Error Rate and Power of Hypothesis Testing under a Logistic  Model}\label{subsect:B-4}
  Suppose the underlying model is the logistic regression model:
 \begin{align}\label{eq:logistic:real} \Prob(Y_i=1|T_i,\bm X_i)= & h\left(T_i\mu_1+(1-T_i)\mu_0+\sum_{j=1}^3 \beta_j X_{i,j}\right)\nonumber \\
 = & h\left(\frac{1}{2}(\mu_1+\mu_0)+(\mu_1-\mu_0)(T_i-\frac{1}{2})+\sum_{j=1}^3 \beta_j X_{i,j}\right),
 \end{align}
 and no covariates are observed in the analysis stage so that the working model is
 \begin{align}\label{eq:logistic:work1} \Prob(Y_i=1|T_i)= &  h\left(T_i\mu_1+(1-T_i)\mu_0\right)\nonumber\\
 =&h\left(\frac{1}{2}(\mu_1+\mu_0)+(\mu_1-\mu_0)(T_i-\frac{1}{2})\right) ,\end{align}
where $h(t)=e^t/(1+e^t)$, $(\beta_1,\beta_2,\beta_3)=(-1,1,2)$, $\mu_0=0$. For the covariates, we consider two settings.
\begin{description}
  \item[\rm Setting 1:]  $X_{i,1}, X_{i,2}, X_{i,3}$ are independent standard normal random variables;
  \item[\rm Setting 2:]  $X_{i,1}, X_{i,2}$ are independent standard normal random variables, $X_{i,3}=X_{i,1}X_{i,2}$.
\end{description}
 For the randomization, we use CR procedure,  SR, and PS procedures with respect to the discretized covariates $\widetilde{\bm X}_i=(d(X_{i,1}), d(X_{i,2}), d(X_{i,3}))$,   and two $\bm\phi$-based CAR procedures with feature $\bm\phi(\bm X_i)=(1, X_{i,1}, X_{i,2}, X_{i,3})$, the same as those in Section B.2.

For testing the treatment effect, that is, $H_0:\mu_1-\mu_0=0$ and $H_1:\mu_1- \mu_0\ne 0$, we consider the traditional tests and adjusted tests.
For the traditional tests, we consider two tests. One is the test with statistic $\mathcal{T}_{LS}(n)$ basing on the linear working model
\begin{equation}\label{eq:logistic:work2} Y_i=T_ih(\mu_1)+(1-T_i)h(\mu_0)+e_i.
 \end{equation}  We refer it to Traditional Test $T_{ls}$. The other is the one by fitting the logistic model  \eqref{eq:logistic:work1} to obtain the $p$-value for testing the coefficient of $T_i-\frac{1}{2}$ to be zero. We refer it to Traditional Test $T_{logi}$. For comparing, we also consider a test by fitting the full logistic model \eqref{eq:logistic:real}  to obtain the $p$-value for testing the coefficient of $T_i-\frac{1}{2}$ to be zero. This is an oracular test and we refer it to $T_{oracle}$. We use the "lm" and "glm" packages of the R software to fit the linear model and logistic model, respectively.
Five adjusted tests which are adjusted by the regression estimator, the bootstrap estimator, the moving block estimator, the moving block jackknife estimator, and the moving block bootstrap estimator the same as those in Section B.2, respectively, are considered under the linear working model \eqref{eq:logistic:work2}.

We consider type I error rates under the null hypothesis and powers under alternative hypothesis $\mu_1-\mu_0=\delta/\sqrt{n}$ with $\delta=5,10$ and $15$  of both the traditional tests and the adjusted tests $\mathcal{T}_{adj}(n)$ with the level of significance $\alpha=0.05$.   The simulation results are given in Table \ref{tab:compareD-1}-\ref{tab:compareD-2}.  Also, to obtain the simulation results in the tables for bootstrap tests   $T_{boot}$ and   $T_{mbb}$ with $5000$ replicates, decades of days are needed. The simulations with $1000$ replicates are     given in Tables \ref{tab:compareD-3}--\ref{tab:compareD-4}. The main findings include

 \begin{description}
   \item[\rm (1) ]   The test $T_{ls}$  by fitting the linear model \eqref{eq:logistic:work2} and the test $T_{logi}$ by fitting the logistic model  \eqref{eq:logistic:work1} perform equivalently as expected, and both loss power comparing with the oracular test $T_{oracle}$. They are   conservative in controlling type I error rates under   CAR procedures including SR, PS, and $\phi$-based CAR procedures;
   \item[\rm (2)]  For the traditional tests, when $\delta$ is small,  the power under CAR procedures  is less than that under the complete randomization and, as the value of $\delta$ increases,  the power under CAR procedures increases to  be similar to or exceed the power under the complete randomization;
   \item[\rm (3)] Under   CAR procedures including SR, PS, and $\phi$-based CAR procedures, the adjusted tests have restored a lot of types I errors and increased the powers a lot. The powers of the adjusted tests under the CAR procedures with feature $\bm \phi=(1, X_1, X_2, X_3)$ are larger than those under SR and PS and are comparable with the powers of the oracular test $T_{oracle}$.
 \end{description}

\begin{table}[ht]
	 \centering
	\caption{Type I error rates and powers of the tests for treatment effect  under logistic model and  various  randomization procedures with independent covariates (Setting 1), sample size $n=500$ and    5000 replicates.}\label{tab:compareD-1}
	\bigskip
	\small
	\begin{tabular}{cc cccccccc    }
		\hline\hline
	$\mu_0=0$		  &  $\mu_1-\mu_0=\delta/\sqrt{n}$	 &\multicolumn{2}{c}{ Traditional Tests  } & & \multicolumn{5}{c}{ Adjusted Tests   }     \\
\cline{3-4} \cline{6-10}
$ \delta$ & Randomization & $T_{ls}$ & $T_{logi}$ & $T_{oracle}$   & $T_{reg}$ & $T_{boot}$ & $T_{mb}$ & $T_{mbj}$ & $T_{mbb}$    \\
\hline
	&	CR	&	0.047	&	0.046 	&	0.053	&	 	-	&	-	&	-	&	-	&	-	\\
	&	SR	&	0.022	&	0.022	&	0.049	&	 	0.056	&	0.053 	&	0.040 	&	0.033 	&	0.042	\\
0	&	PS	&	0.021	&	0.021	&	0.051	&	 	0.052 	&	0.048	&	0.039 	&	0.034	&	0.041	\\
	&	$\phi$-CAR-BC	&	0.012	&	0.012	&	0.049	 	&	0.046 	&	0.051	&	0.047	&	0.039	&	0.045 	\\
	&	$\phi$-CAR-Con	&	0.011 	&	0.011 	&	0.054	 	&	0.048 	&	0.056	&	0.048 	&	0.047 	&	0.051 	\\
\hline	 		 		 		 		 	 		 		 		 		 		
	&	CR	&	0.111 	&	0.111	&	0.152	&	 	-	&	-	&	-	&	-	&	-		\\
	&	SR	&	0.060	&	0.060	&	0.139 	&	 	0.123	&	0.131	&	0.102 	&	0.099 	&	0.099	\\
5	&	PS	&	0.064	&	0.065	&	0.152	&	 	0.139	&	0.126 	&	0.106	&	0.106	&	0.108	\\
	&	$\phi$-CAR-BC	&	0.055 	&	0.055 	&	0.157 	 	&	0.149	&	0.141	&	0.112 	&	0.111	&	0.126	\\
	&	$\phi$-CAR-Con	&	0.049 	&	0.049 	&	0.146	 	&	0.143	&	0.137	&	0.142	&	0.122	&	0.137	\\
\hline	 		 		 		 		 	 		 		 		 		 		
	&	CR	&	0.271 	&	0.271	&	0.449	&	 	-	&	-	&	-	&	-	&	-		\\
	&	SR	&	0.227 	&	0.226 	&	0.446	&	 	0.363 	&	0.365	&	0.311 	&	0.291 	&	0.309	\\
10	&	PS	&	0.233 	&	0.233 	&	0.428	&	 	0.376 	&	0.356	&	0.316 	&	0.308	&	0.316 	\\
	&	$\phi$-CAR-BC	&	0.224 	&	0.224 	&	0.454 	 	&	0.424	&	0.394 	&	0.386	&	0.365 	&	0.367 	\\
	&	$\phi$-CAR-Con	&	0.212 	&	0.212 	&	0.444 	 	&	0.406	&	0.416 	&	0.405	&	0.372	&	0.397	\\
\hline	 		 		 		 		 	 		 		 		 		 		
	&	CR	&	0.508 	&	0.508	&	0.781 	&	 	-	&	-	&	-	&	-	&	-		\\
	&	SR	&	0.510	&	0.509	&	0.767	&	 	0.668 	&	0.664	&	0.611	&	0.592	&	0.595	\\
15	&	PS	&	0.521 	&	0.521 	&	0.773	&	 	0.665	&	0.666	&	0.614 	&	0.617 	&	0.633	\\
	&	$\phi$-CAR-BC	&	0.534 	&	0.534 	&	0.772	 	&	0.742	&	0.739	&	0.682	&	0.666	&	0.699 	\\
	&	$\phi$-CAR-Con	&	0.533	&	0.533	&	0.773	 	&	0.740 	&	0.729	&	0.719	&	0.692	&	0.724 	\\
\hline
\hline
		\end{tabular}
\end{table}
\clearpage

\begin{table}[ht]
	 \centering
	\caption{Type I error rates and powers of the tests for treatment effect  under logistic model and  various  randomization procedures with interactive covariates (Setting 2), sample size $n=500$ and    5000 replicates..}\label{tab:compareD-2}
	\bigskip
	\small
	\begin{tabular}{cc cccccccc    }
		\hline\hline
	$\mu_0=0$		  &  $\mu_1-\mu_0=\delta/\sqrt{n}$	 &\multicolumn{2}{c}{ Traditional Tests  } & & \multicolumn{5}{c}{ Adjusted Tests   }     \\
\cline{3-4} \cline{6-10}
$ \delta$ & Randomization & $T_{ls}$ & $T_{logi}$ & $T_{oracle}$   & $T_{reg}$ & $T_{boot}$ & $T_{mb}$ & $T_{mbj}$ & $T_{mbb}$    \\
\hline
	&	CR	&	0.055 	&	0.055 	&	0.051 	&	-	&	-	&	- 	&	-	&	-	\\
	&	SR	&	0.033	&	0.033	&	0.051	&	0.053 	&	0.050	&	0.049	&	0.045	&	0.052	\\
0	&	PS	&	0.026	&	0.026	&	0.045 	&	0.047	&	0.051 	&	0.054	&	0.044 	&	0.052 	\\
	&	$\phi$-CAR-BC	&	0.023	&	0.023	&	0.049 	&	0.055	&	0.050	&	0.052 	&	0.042 	&	0.044 	\\
	&	$\phi$-CAR-Con	&	0.027	&	0.027	&	0.056 	&	0.056	&	0.054	&	0.058	&	0.050 	&	0.053	\\
\hline
	&	CR	&	0.123 	&	0.123	&	0.173 	&	-	&	-	&	- 	&	-	&	- 	\\
	&	SR	&	0.104  	&	0.104 	&	0.172	&	0.143 	&	0.155	&	0.154	&	0.141	&	0.137 	\\
5	&	PS	&	0.109 	&	0.109 	&	0.182	&	0.149	&	0.156 	&	0.151	&	0.134 	&	0.148	\\
	&	$\phi$-CAR-BC	&	0.092	&	0.092	&	0.176	&	0.162	&	0.157	&	0.159 	&	0.140	&	0.150	\\
	&	$\phi$-CAR-Con	&	0.098	&	0.098	&	0.174	&	0.162 	&	0.166	&	0.159	&	0.148	&	0.161 	\\
\hline
	&	CR	&	0.353	&	0.352 	&	0.505 	&	-	&	-	&	- 	&	-	&	-	\\
	&	SR	&	0.360	&	0.360	&	0.522 	&	0.452	&	0.457	&	0.424	&	0.414 	&	0.429	\\
10	&	PS	&	0.362	&	0.362	&	0.512 	&	0.461	&	0.450	&	0.420	&	0.408	&	0.417	\\
	&	$\phi$-CAR-BC	&	0.348	&	0.348	&	0.509	&	0.489	&	0.466	&	0.444 	&	0.429 	&	0.457 	\\
	&	$\phi$-CAR-Con	&	0.341	&	0.341	&	0.508	&	0.477 	&	0.470 	&	0.465	&	0.453	&	0.469	\\
\hline
	&	CR	&	0.665	&	0.664	&	0.838 	&	-	&	-	&	- 	&	-	&	-	\\
	&	SR	&	0.687 	&	0.687	&	0.846 	&	0.777	&	0.772	&	0.749 	&	0.731	&	0.766	\\
15	&	PS	&	0.685	&	0.685	&	0.832	&	0.773 	&	0.775 	&	0.758	&	0.741	&	0.748 	\\
	&	$\phi$-CAR-BC	&	0.702	&	0.702	&	0.838	&	0.801 	&	0.791	&	0.772 	&	0.761	&	0.770	\\
	&	$\phi$-CAR-Con	&	0.693	&	0.693	&	0.842 	&	0.800	&	0.798	&	0.788	&	0.771 	&	0.786 	\\
\hline
\hline
		\end{tabular}
\end{table}
\clearpage

\begin{table}[h]
	 \centering
	\caption{Type I error rates and powers of traditional tests and adjusted tests for treatment effect  under logistic model and  various  randomization procedures with dependent covariates (Setting 1), sample size $n=500$ and 1000 replicates.}\label{tab:compareD-3}
	\bigskip
	\small
	\begin{tabular}{cc cccccccc    }
		\hline\hline
	$\mu_0=0$		  &  $\mu_1-\mu_0=\delta/\sqrt{n}$	 &\multicolumn{2}{c}{ Traditional Tests  } & & \multicolumn{5}{c}{ Adjusted Tests   }     \\
\cline{3-4} \cline{6-10}
$ \delta$ & Randomization & $T_{ls}$ & $T_{logi}$ & $T_{oracle}$   & $T_{reg}$ & $T_{boot}$ & $T_{mb}$ & $T_{mbj}$ & $T_{mbb}$    \\
\hline
	&	CR	&	0.052	&	0.051	&	0.035	&	-	&	-	&	-	&	-	&	-	\\
	&	SR	&	0.019	&	0.019	&	0.047	&	0.045	&	0.057	&	0.042	&	0.031	&	0.045	\\
0	&	PS	&	0.020	&	0.020	&	0.052	&	0.062	&	0.052	&	0.040	&	0.038	&	0.038	\\
	&	$\phi$-CAR-BC	&	0.015	&	0.015	&	0.044	&	0.048	&	0.051	&	0.041	&	0.029	&	0.037	\\
	&	$\phi$-CAR-Con	&	0.015	&	0.015	&	0.051	&	0.052	&	0.041	&	0.050	&	0.044	&	0.044	\\
\hline	 																		
	&	CR	&	0.088	&	0.088	&	0.149	&	-	&	-	&	-	&	-	&	-	\\
	&	SR	&	0.058	&	0.058	&	0.150	&	0.125	&	0.123	&	0.085	&	0.099	&	0.095	\\
5	&	PS	&	0.064	&	0.064	&	0.131	&	0.124	&	0.142	&	0.133	&	0.119	&	0.103	\\
	&	$\phi$-CAR-BC	&	0.055	&	0.055	&	0.172	&	0.156	&	0.141	&	0.120	&	0.124	&	0.125	\\
	&	$\phi$-CAR-Con	&	0.044	&	0.044	&	0.117	&	0.148	&	0.129	&	0.127	&	0.114	&	0.112	\\
\hline	 																		
	&	CR	&	0.267	&	0.267	&	0.435	&	-	&	-	&	-	&	-	&	-	\\
	&	SR	&	0.239	&	0.239	&	0.452	&	0.370	&	0.345	&	0.310	&	0.279	&	0.282	\\
10	&	PS	&	0.250	&	0.250	&	0.444	&	0.356	&	0.367	&	0.334	&	0.305	&	0.330	\\
	&	$\phi$-CAR-BC	&	0.180	&	0.180	&	0.406	&	0.423	&	0.404	&	0.368	&	0.367	&	0.389	\\
	&	$\phi$-CAR-Con	&	0.213	&	0.213	&	0.432	&	0.407	&	0.408	&	0.404	&	0.421	&	0.399	\\
\hline	 																		
	&	CR	&	0.492	&	0.489	&	0.787	&	-	&	-	&	-	&	-	&	-	\\
	&	SR	&	0.505	&	0.504	&	0.762	&	0.665	&	0.656	&	0.608	&	0.579	&	0.577	\\
15	&	PS	&	0.502	&	0.502	&	0.762	&	0.681	&	0.695	&	0.626	&	0.579	&	0.612	\\
	&	$\phi$-CAR-BC	&	0.522	&	0.522	&	0.777	&	0.739	&	0.745	&	0.704	&	0.671	&	0.683	\\
	&	$\phi$-CAR-Con	&	0.524	&	0.524	&	0.790	&	0.758	&	0.767	&	0.737	&	0.714	&	0.713	\\
\hline
\hline
		\end{tabular}
\end{table}
\clearpage

\begin{table}[h]
	 \centering
	\caption{Type I error rates and powers of traditional tests and adjusted tests for treatment effect  under logistic model and  various  randomization procedures with dependent covariates (Setting 2), sample size $n=500$ and 1000 replicates.}\label{tab:compareD-4}
	\bigskip
	\small
	\begin{tabular}{cc cccccccc    }
		\hline\hline
	$\mu_0=0$		  &  $\mu_1-\mu_0=\delta/\sqrt{n}$	 &\multicolumn{2}{c}{ Traditional Tests  } & & \multicolumn{5}{c}{ Adjusted Tests   }     \\
\cline{3-4} \cline{6-10}
$ \delta$ & Randomization & $T_{ls}$ & $T_{logi}$ & $T_{oracle}$   & $T_{reg}$ & $T_{boot}$ & $T_{mb}$ & $T_{mbj}$ & $T_{mbb}$    \\
\hline
	&	CR	&	0.050	&	0.050	&	0.051	&	-	&	-	&	-	&	-	&	-	\\
	&	SR	&	0.044	&	0.044	&	0.067	&	0.051	&	0.049	&	0.056	&	0.043	&	0.052	\\
0	&	PS	&	0.026	&	0.026	&	0.049	&	0.049	&	0.048	&	0.031	&	0.043	&	0.050	\\
	&	$\phi$-CAR-BC	&	0.036	&	0.036	&	0.054	&	0.058	&	0.041	&	0.046	&	0.039	&	0.051	\\
	&	$\phi$-CAR-Con	&	0.020	&	0.020	&	0.056	&	0.065	&	0.054	&	0.057	&	0.046	&	0.062	\\
	 																		\hline
	&	CR	&	0.128	&	0.128	&	0.164	&	-	&	-	&	-	&	-	&	-	\\
	&	SR	&	0.112	&	0.111	&	0.186	&	0.158	&	0.143	&	0.139	&	0.142	&	0.135	\\
5	&	PS	&	0.112	&	0.112	&	0.175	&	0.157	&	0.156	&	0.136	&	0.128	&	0.152	\\
	&	$\phi$-CAR-BC	&	0.089	&	0.089	&	0.167	&	0.155	&	0.149	&	0.158	&	0.145	&	0.142	\\
	&	$\phi$-CAR-Con	&	0.099	&	0.099	&	0.189	&	0.160	&	0.148	&	0.125	&	0.141	&	0.151	\\
	 																		\hline
	&	CR	&	0.362	&	0.362	&	0.516	&	-	&	-	&	-	&	-	&	-	\\
	&	SR	&	0.356	&	0.356	&	0.529	&	0.443	&	0.431	&	0.415	&	0.434	&	0.425	\\
10	&	PS	&	0.355	&	0.354	&	0.522	&	0.461	&	0.453	&	0.447	&	0.391	&	0.449	\\
	&	$\phi$-CAR-BC	&	0.366	&	0.366	&	0.543	&	0.489	&	0.461	&	0.431	&	0.428	&	0.453	\\
	&	$\phi$-CAR-Con	&	0.337	&	0.337	&	0.511	&	0.468	&	0.475	&	0.452	&	0.449	&	0.468	\\
	 																		\hline
	&	CR	&	0.666	&	0.665	&	0.849	&	-	&	-	&	-	&	-	&	-	\\
	&	SR	&	0.671	&	0.669	&	0.833	&	0.764	&	0.796	&	0.768	&	0.737	&	0.746	\\
15	&	PS	&	0.664	&	0.664	&	0.821	&	0.778	&	0.784	&	0.746	&	0.727	&	0.771	\\
	&	$\phi$-CAR-BC	&	0.689	&	0.689	&	0.840	&	0.812	&	0.800	&	0.766	&	0.738	&	0.778	\\
	&	$\phi$-CAR-Con	&	0.667	&	0.667	&	0.836	&	0.815	&	0.793	&	0.777	&	0.760	&	0.781	\\
\hline
\hline
		\end{tabular}
\end{table}

\end{document}